\documentclass[10pt]{article}
\usepackage{authblk}

\usepackage{bbm}
\usepackage{xcolor}
\definecolor{azure}{rgb}{0.25, 0.41, 0.88}
\usepackage{times}
\usepackage{mathtools}
\usepackage[
  colorlinks   = true,   
  linkcolor    = azure,  
  citecolor    = azure, 
  urlcolor     = azure,  
  linktoc      = page    
]{hyperref}
\usepackage[shortlabels]{enumitem}
\usepackage{mathrsfs}
\usepackage{mhequ}

\RequirePackage{color}
\RequirePackage[a4paper,top=1.6in,bottom=1.6in,left=1.6in,right=1.6in]{geometry}
\RequirePackage{lastpage}
\usepackage[english]{babel}
\RequirePackage{amsmath,amsthm,amssymb,mathtools}
\RequirePackage{enumitem}
\RequirePackage{placeins}

\colorlet{darkblue}{blue!90!black}
\colorlet{darkorange}{orange!90!black}
\definecolor{darkpink}{rgb}{0.98, 0.38, 0.5}

\usepackage{fancyhdr}
\usepackage{tikz}
\usetikzlibrary{arrows,decorations.pathmorphing,backgrounds,positioning,fit,petri}
\usetikzlibrary{shapes}

\usepackage[font=small]{caption}
\usepackage{titlesec}
\titleformat{\section}
  {\normalfont\fontsize{14}{15}\bfseries}{\thesection}{1em}{}
\titleformat{\subsection}
  {\normalfont\fontsize{12}{15}\bfseries}{\thesubsection}{1em}{}

\newcommand{\expected}{\mathbb{E}}
\newcommand{\prob}{\mathbb{P}}
\newcommand{\eps}{\varepsilon}

\makeatletter
\DeclareFontFamily{OMX}{MnSymbolE}{}
\DeclareSymbolFont{MnLargeSymbols}{OMX}{MnSymbolE}{m}{n}
\SetSymbolFont{MnLargeSymbols}{bold}{OMX}{MnSymbolE}{b}{n}
\DeclareFontShape{OMX}{MnSymbolE}{m}{n}{
    <-6>  MnSymbolE5
   <6-7>  MnSymbolE6
   <7-8>  MnSymbolE7
   <8-9>  MnSymbolE8
   <9-10> MnSymbolE9
  <10-12> MnSymbolE10
  <12->   MnSymbolE12
}{}
\DeclareFontShape{OMX}{MnSymbolE}{b}{n}{
    <-6>  MnSymbolE-Bold5
   <6-7>  MnSymbolE-Bold6
   <7-8>  MnSymbolE-Bold7
   <8-9>  MnSymbolE-Bold8
   <9-10> MnSymbolE-Bold9
  <10-12> MnSymbolE-Bold10
  <12->   MnSymbolE-Bold12
}{}

\let\llangle\@undefined
\let\rrangle\@undefined
\DeclareMathDelimiter{\llangle}{\mathopen}
                     {MnLargeSymbols}{'164}{MnLargeSymbols}{'164}
\DeclareMathDelimiter{\rrangle}{\mathclose}%
                     {MnLargeSymbols}{'171}{MnLargeSymbols}{'171}
\makeatother

\DeclareMathOperator*{\mymax}{\text{\normalfont max}}
\DeclareMathOperator*{\mymin}{\text{\normalfont min}}
\DeclareMathOperator*{\mysup}{\text{\normalfont sup}}
\DeclareMathOperator*{\myinf}{\text{\normalfont inf}}
\DeclareMathOperator*{\mylim}{\text{\normalfont lim}}
\DeclareMathOperator*{\mylog}{\text{\normalfont log}}
\DeclareMathOperator*{\mylimsup}{\text{\normalfont limsup}}

\DeclareMathOperator*{\de}{\hspace{-.1em}\text{\normalfont d\hspace{-.15em}}}

\numberwithin{equation}{section}
\newtheorem{theorem}{Theorem}[section]
\newtheorem{corollary}[theorem]{Corollary}
\newtheorem{lemma}[theorem]{Lemma}
\newtheorem{proposition}[theorem]{Proposition}
\theoremstyle{definition}
\newtheorem{definition}[theorem]{Definition}
\newtheorem{remark}[theorem]{Remark}
\newtheorem*{acknowledgements}{Acknowledgements}

\DeclareFontEncoding{LS1}{}{}
\DeclareFontSubstitution{LS1}{stix}{m}{n}
\DeclareSymbolFont{stixletters}{LS1}{stix}{m}{it}
\DeclareMathAccent{\cev}{\mathord}{stixletters}{"91}
\DeclareMathAccent{\vec}{\mathord}{stixletters}{"92}

\makeatletter
\newcommand\@affiliationlist{}

\renewcommand\affil[2][]{%
  \g@addto@macro\@affiliationlist{%
    \item[\textsuperscript{\scriptsize #1}] #2%
  }%
}

\newlength{\affillabelwidth}

\renewcommand{\maketitle}{
  \vspace*{0em}%
  \begin{center}{\LARGE\bfseries \@title \par}\end{center}
  \vspace{1em}%
  \begin{center}{\small \scshape\@date \par}\end{center}\par 
  \vspace{4em}
  \begin{list}{}{\setlength{\leftmargin}{1.5em}\setlength{\rightmargin}{1.5em}}
    \item[] {\raggedright\large\@author \par}
  \end{list}

  \settowidth{\affillabelwidth}{\textsuperscript{\scriptsize 8}} 
  \begin{list}{}{%
      \setlength{\leftmargin}{2.2em}
      \setlength{\rightmargin}{1.5em}
      \setlength{\labelwidth}{\affillabelwidth}
      \setlength{\labelsep}{0.2em}
      \setlength{\itemsep}{.2em}
      \setlength{\parsep}{0pt}
      \setlength{\topsep}{0pt}
  }
    {\raggedright\small \@affiliationlist}
  \end{list}\vspace{0.5em}
}

\renewenvironment{abstract}
  {\par                               
   \begin{list}{}{                        
      \setlength{\leftmargin}{1.5em}
      \setlength{\rightmargin}{1.5em}}
   \item[]                                   
   {\bfseries\abstractname}\par               
   \small\noindent\ignorespaces}             
  {\end{list}}

\begin{document}
\allowdisplaybreaks
\date{\today}
\title{Hydrodynamic Limit of the Symmetric Zero-Range Process with Slow Boundary}

\author{\bfseries O. Araújo\textsuperscript{1}, P. Gonçalves\textsuperscript{2}, A. Neumann\textsuperscript{3}, M. C. Ricciuti\textsuperscript{4}
}

\affil[1]{Department of Mathematics, Universidade Federal da Paraíba, Cidade Universitária, 58059-900, João Pessoa, PB, Brazil, Email: \href{mailto:oslenne.nogueira@ccen.ufpb.br}{\texttt{oslenne.nogueira@ccen.ufpb.br}}}

\affil[2]{Center for Mathematical Analysis, Geometry and Dynamical Systems, Instituto Superior Técnico, Universidade de Lisboa, Av. Rovisco Pais, 1049-001 Lisboa, Portugal, Email: \href{mailto:pgoncalves@tecnico.ulisboa.pt}{\texttt{pgoncalves@tecnico.ulisboa.pt}}}
 
\affil[3]{Mathematics and Statistics Institute, Universidade Federal do Rio Grande do Sul, Av. Bento Gonçalves, 9500, Porto Alegre, RS, Brazil, Email: \href{mailto:aneumann@mat.ufrgs.br}{\texttt{aneumann@mat.ufrgs.br}}}

\affil[4]{Department of Mathematics, Imperial College London, 180 Queen's Gate, London SW7 2AZ, United Kingdom, Email: \href{mailto:maria.ricciuti18@imperial.ac.uk}{\texttt{maria.ricciuti18@imperial.ac.uk}}}

\clearpage\maketitle
\thispagestyle{empty}

\begin{abstract} 
We study the hydrodynamic behaviour of the symmetric zero-range process on the finite interval $\{1, \ldots, N-1\}$ in contact with slow reservoirs at the boundary. Particles are injected and removed at sites $1$ and $N-1$ at rates that scale like $N^{-\theta}$ with $\theta\ge1$. Under mild assumptions on the jump rate and the sequence of initial measures, we show that the empirical density evolves on the diffusive scale according to a nonlinear heat equation, with boundary conditions reflecting the strength of the reservoirs.
\end{abstract}


\section{Introduction}
One of the central challenges in statistical physics is to understand heat transfer through a thermal bath. A common approach is to discretise the macroscopic domain with a parameter $N$, and then study systems of random particles evolving on this discrete space. These particles perform continuous-time random walks, with exponentially distributed waiting times that confer the Markov property, but their motion is constrained to reflect the underlying physical rules. A widely studied example is the boundary-driven \textit{exclusion process}, in which each site can host at most one particle: attempted jumps to occupied sites are suppressed, while jumps to empty sites occur according to a transition probability $p(x,y)$. The number of particles at site $x$ and at time $t$ is denoted by $\xi_t(x)$, so that $\xi_t(x)\in\{0,1\}$ for each $x$. This model has served as a fundamental testing ground for understanding nonequilibrium steady states and hydrodynamic behaviour.

In this work, we focus instead on a different particle system in which no restriction is imposed on the number of particles per site, so if $\eta_t(x)$ denotes the number of particles at site $x$ and time $t$, then $\eta_t(x)\in\mathbb{N}$. After the ring of an exponentially distributed clock, a particle jumps from $x$ to $y$ with probability $p(x,y)$, but at a rate weighted by a factor $g(\eta(x))$, where $g$ satisfies mild assumptions to ensure well-defined dynamics. This model, introduced by Spitzer in 1970 \cite{spi70}, is known as the \textit{zero-range process}, since the jump rate depends only on the occupation number at the departure site through $g$, and not on the state of the destination site. As in the exclusion process, this model can be coupled to boundary reservoirs, with injection and removal rates chosen consistently with the bulk dynamics to obtain meaningful information about stationary measures.

Our primary interest is to describe the hydrodynamic limit of the zero-range process in contact with stochastic reservoirs. In particular, we restrict attention to the symmetric, nearest-neighbour case, where the transition probability is given by $p(x,y)=\boldsymbol{1}_{\{x\pm1\}}(y)$. At the macroscopic level, we consider the interval $[0,1]$, which we discretise by $N$ at the microscopic level, so that particles evolve in the bulk $I_N := \{1,\dots,N-1\}$. The boundary sites $x=1$ and $x=N-1$ are connected to reservoirs that inject and remove particles. Without these additional dynamics, the number of particles would be conserved; however, the reservoirs destroy this conservation law, and our goal is to analyse the macroscopic consequences of this interaction. 

As previously mentioned, the exclusion process in contact with stochastic reservoirs has been extensively studied; see, for example, \cite{bmns17}, \cite{bgj19}, \cite{gon19}, \cite{bgj21}, \cite{gs22} and \cite{bcgs23}. The intensity of the boundary dynamics can be tuned by a parameter $\theta \in \mathbb{R}$: if $\theta<0$ (resp. $\theta>0$), the boundary dynamics are faster (resp. slower) than the bulk dynamics, while $\theta=0$ corresponds to equal intensities. For the exclusion process with nearest-neighbour jumps, the limiting PDE is supplemented with boundary conditions that depend on $\theta$. When $\theta < 1$, the boundary dynamics dominate and Dirichlet conditions arise, fixing the density at the boundary. When $\theta > 1$, the boundary evolves more slowly, and the system behaves macroscopically as isolated, giving rise to Neumann conditions with vanishing flux. The critical regime $\theta = 1$ is the most delicate: it leads to Robin-type conditions, which in the exclusion process turn out to be linear. We emphasise, however, that in this setting the invariant measure of this process is not of product form, and no explicit expression is available, except for the recent results of \cite{fg23}. 

By contrast, much less is known about the boundary-driven zero-range process, despite its physical relevance for modelling transport between a system and its environment. On the torus or the infinite lattice, the hydrodynamic limit is well understood (see \cite[Chapter 6]{kl99} and references therein), and it is given by the nonlinear parabolic equation $\partial_t \rho_t(u) =\Delta\Phi(\rho_t(u))$, where the flux function $\Phi$ is determined by the expectation of $g$ under the invariant measure. We also note that, recently, \cite{mmm24} obtained the hydrodynamic limit and convergence rates of the process evolving on the torus. In the case of stochastic reservoirs, some results exist for long-range jumps, but the nearest-neighbour case remains largely unexplored. A partial analysis was given in \cite{fmn21}, where the limiting PDE was identified heuristically, conditional on certain replacement lemmas at both the bulk and the boundary. For the exclusion process, such lemmas are usually unnecessary, since the discrete evolution equations are already closed in terms of the density of particles. For the zero-range process, however, the nonlinearity of the jump rate prevents this closure, making these replacements the main obstacle to a rigorous derivation. On the other hand, unlike exclusion, the zero-range process admits a fully explicit description of its invariant measure even with reservoirs and, remarkably, these are of product form; see \cite{lms05}. 

\vspace{1em}
\textbf{Our Contribution.} 
In this article, we rigorously establish the hydrodynamic limit of the symmetric, nearest-neighbour zero-range process in contact with stochastic reservoirs in the regimes $\theta\ge1$. To the best of our knowledge, this is the first rigorous result in this setting, closing the gap left open in \cite{fmn21}. Our result imposes two conditions on the sequence of initial measures $\{\mu^N\}_N$: that $\mu^N$ has a relative entropy of order at most $N$ with respect to the invariant state $\bar\nu_N$, and that $\mu^N$ is stochastically dominated by $\bar\nu_N$. The entropy condition is quite natural, as it allows us to make a change of measure to a reference measure which provides useful information about the dynamics. The stochastic domination condition is more technical, and we impose it because, under mild assumptions on $g$, it implies a monotonicity property -- \emph{attractiveness} -- which permits bounding expectations of monotone functions under $\mu^N$ by those under $\bar{\nu}_N$. Since $\bar{\nu}_N$ is explicit and of product form, such expectations can readily be controlled. Indeed, one of the main difficulties is that occupation numbers are unbounded and $g$ is generic, but these bounds allow us to introduce cutoffs on high-density configurations and restrict only to bounded configurations. 

Our approach loosely follows the program of \cite{kl99} for the zero-range process on the torus, but with substantial adaptations. The general idea is to show a \textit{replacement lemma} for the dynamics, which allows to substitute local averages involving the microscopic jump rate $g$ with local averages of the macroscopic flux $\Phi$. This is achieved in two stages: the \emph{one-block estimate}, which replaces such averages in microscopically small boxes, and the \emph{two-block estimate}, which bridges the substitution to macroscopically small boxes. In the torus case, the reference measure is the invariant one, and its translation invariance is extensively exploited, with an additional ingredient -- the \emph{equivalence of ensembles} -- also playing a crucial role. Since our invariant measure lacks these fundamental properties, we choose to work with averages under a different reference measure, namely the invariant measure of the dynamics restricted to the bulk. One of the key points is that the conditions imposed on the sequence of initial measures with respect to the invariant one can be transferred to this chosen reference measure. Although this measure is not invariant for the full dynamics, it retains the fundamental structure -- translation invariance and equivalence of ensembles -- to make this program applicable. The restriction to $\theta\ge1$ then arises because, as this measure lacks stationarity, there is a discrepancy between the Dirichlet form of the process and the carré-du-champ operator, which is of order $N^{1-\theta}$ and can thus only be controlled for $\theta\ge1$. 

Nonetheless, since we still lack full translation invariance of this reference measure, in both the one-block and two-block estimates we need a different approach to make the program work. Our strategy takes inspiration from the Boltzmann-Gibbs principle of \cite{fmts23}. There, the authors make use of the mixing properties of their process in the form of a \textit{spectral gap bound}, which allows them to apply the Rayleigh estimate of \cite[Theorem A3.1.1]{kl99}. The same spectral gap bound is verified for our zero-range process restricted to the bulk for a large class of functions $g$; see, for example, \cite{lsv96}, \cite{mor06} and \cite{nag10}. Here, we apply a similar argument to the one given for \cite[Lemma 3.1]{fmts23} to show the one-block estimate, and extend it to the two-block setting. To this end, our main tool is the derivation of a spectral gap bound for a ``coupled" version of our process, that is, a system evolving as a zero-range process on two separate boxes of the same size, but with particles also allowed to jump from the midpoint of one of the boxes to the midpoint of the other, and vice versa. We show that, provided that $g$ is such that an appropriate spectral gap bound holds for the original process, such a bound also holds for the coupled version. Then, a generalisation of the Rayleigh estimate to functions which are not necessarily mean zero, namely \cite[Corollary A3.1.2]{kl99}, allows us to conclude the bound.

Finally, we note that the hydrostatic limit follows as a corollary of our main result. Indeed, we are able to show that the stationary measure, which is explicit, satisfies the conditions of our main theorem, namely the association to a measurable and integrable profile (the stationary solution of the hydrodynamic equation), and the trivially verified entropy bound and stochastic domination.

\vspace{1em}
\textbf{Future Work.} 
Some related questions remain open. The first and perhaps most natural one is the derivation of the hydrodynamic limit when $\theta<1$. In this regime, because of the discrepancy between Dirichlet forms and carré-du-champ operators mentioned above, if one were to apply our strategy of proof, the reference measure would have to be chosen to be the stationary one. This would come at the cost of losing translation invariance and, additionally, an equivalence of ensembles and a spectral gap bound would need to be derived. 

A second direction concerns long-range dynamics. For the exclusion process, it is well known that finite-range jumps lead to the heat equation, while long-range jumps give rise to the fractional heat equation. In contrast, the hydrodynamic behaviour of the zero-range process with long jumps remains open. In \cite{bgjs22}, the hydrodynamic limit was established in expectation via connections with the exclusion process. We believe that our approach could be extended to the long-range case, yielding convergence in probability.

Finally, an important next step is the analysis of fluctuations around the hydrodynamic limit studied in this work. Starting from the invariant state, one would expect Gaussian fluctuations described by a generalised Ornstein-Uhlenbeck equation with coefficients depending on the stationary density profile. For the analogue version of the exclusion process, this problem has been addressed (see \cite{fgn17} and \cite{fgn19}), but for the zero-range process, a novel version of the replacement lemma (this time in the form of a Boltzmann-Gibbs principle) will once again be necessary.  

\vspace{1em}
\textbf{Outline.} 
In Section \ref{sec:model_stat}, we introduce the model and state our main result, namely Theorem \ref{thm:main}. Section \ref{sec:proof} is dedicated to its proof, which is organised into several parts. We begin by constructing auxiliary martingales obtained by applying Dynkin’s formula to suitable test functions, which will serve as the building blocks of our argument. Section \ref{sec:proof_tightness} establishes tightness of the sequence of empirical density measures. In Section \ref{sec:abs_cont}, we show that limit points are concentrated on absolutely continuous measures with respect to the Lebesgue measure. Section \ref{sec:characterisation} contains the characterisation of limit points as weak solutions to the hydrodynamic equation \eqref{eq:pde_zr}. Finally, the Appendix collects auxiliary results: an energy estimate, uniqueness of weak solutions of the hydrodynamic PDE, bounds on Dirichlet forms, and an entropy inequality.

\vspace{.6em}
\begin{acknowledgements} \small{O.A. acknowledges support from the National Council for Scientific and Technological Development (CNPq, Brazil) through a Junior Postdoctoral Fellowship (Grant No. 152638/2022-9), and from the Fundação para a Ciência e a Tecnologia (FCT, Portugal) under the project IT137-22-035 - UIDP/00324/2020 - Center for Mathematics of the University of Coimbra. P.G. expresses warm thanks to Fundação para a Ciência e Tecnologia FCT/Portugal for financial support through the projects UIDB/04459/2020, UIDP/04459/2020 and ERC/FCT. A.N. thanks to National Council for Scientific and Technological Development (CNPq
Brazil) through a Bolsa de Produtividade number 313916/2021-7. A.N. also thanks Ricardo Misturini and Susana Frómeta for valuable discussions on this model in \cite{fmn21}, which motivated the present study. M.C.R. gratefully acknowledges support from the Dean's PhD Scholarship at Imperial College London, the kind hospitality of Instituto Superior Técnico in March-April 2025 (when part of this work was carried out), and G-Research for supporting this visit through a travel grant.}
\end{acknowledgements}

\section{The Model and Statement of Results}\label{sec:model_stat}
For $N$ positive integer, let $I_N$ denote the discrete interval $\{1,\ldots, N-1\}$. This constitutes the discrete space where a set of indistinguishable particles will be moving around, and we will call it \textit{bulk}. For each $x \in I_N$, the occupation variable $\eta(x)$ will represent the number of particles at site $x$. The \textit{zero-range process} describes the random evolution of particles with no restriction on the total number of particles per site, and therefore its state space of configurations $\eta$ is the set $\Omega_N := \mathbb{N}^{I_N}$.

The rates of the process are defined through a function $g :\mathbb{N}\to [0, \infty)$ satisfying  $g(0)=0$. This last identity mirrors the fact that if there are no particles at the site $x$, then the rate for a jump is null.  Throughout this work, we assume that $g$ is strictly positive on the set of positive integers and that it has bounded variation in the following sense:
\begin{equation}\label{eq:function_g}
    g^*:=\mysup_{k\ge1}|g(k+1)-g(k)|<\infty.
\end{equation}
The previous condition ensures that the process is well-defined; see \cite{lig05} for details. 

We consider the Markov process whose infinitesimal generator $\mathcal{L}^N$ is given by
\begin{equation}\label{eq:generator}
    \mathcal{L}^N:=\mathcal{L}^N_0 + \kappa \mathcal{L}^N_b,
\end{equation} 
where $\kappa>0$ and where $\mathcal{L}^N_0$ and $\mathcal{L}^N_b$ represent the infinitesimal generators of the bulk dynamics and the boundary dynamics, respectively. These generators act on maps $f:\Omega_N \to \mathbb{R}$ via
\begin{equation*}
    \mathcal{L}^N_0f(\eta):=\sum_{\substack{x, y\in I_N \\ |x-y|=1}} g(\eta(x)) \left[f(\eta^{x,y})-f(\eta)\right]
\end{equation*} 	
and $\mathcal{L}^N_b:=\mathcal{L}^N_l+\mathcal{L}^N_r$, where 
\begin{equation}\label{eq:boundary_gen}
    \begin{split}
    &\mathcal{L}^N_l f(\eta):=\frac{\alpha}{N^\theta}\left[f(\eta^{1+})-f(\eta)\right] +\frac{\lambda g(\eta(1))}{N^\theta}\left[f(\eta^{1-})-f(\eta)\right],
    \\&\mathcal{L}^N_r f(\eta):=\frac{\beta}{N^\theta}\left[f(\eta^{(N-1)+})-f(\eta)\right]+\frac{\delta g(\eta(N-1))}{N^\theta}\left[f(\eta^{(N-1)-})-f(\eta)\right],
    \end{split}
\end{equation}
with
\begin{equation*}
    \eta^{x,y}(z):=\left\{\begin{array}{ll}
    \eta(z)-1, & \mbox{~~if~~} z=x,\\
    \eta(z)+1, &\mbox{~~if~~} z=y,\\
    \eta(z), & \mbox{~~if~~} z\ne x,y
    \end{array}\right.
\end{equation*}
and 
\begin{equation*}
    \eta^{x\pm}(z):=\left\{\begin{array}{ll}
    \eta(z) \pm 1, & \mbox{~~if~~} z=x,\\
    \eta(z), & \mbox{~~if~~} z\ne x.
    \end{array}\right.
\end{equation*}

Above, the parameters $\alpha, \lambda, \beta, \delta\ge 0$ represent the density of the reservoirs, while the parameter $\theta\ge 0$ plays the role of tuning the speed of the boundary dynamics. We will denote by $\{\eta_t=\eta_t^N, t\ge0\}$ the continuous-time Markov process on $\Omega_N$ with generator $N^2\mathcal{L}^N$. 

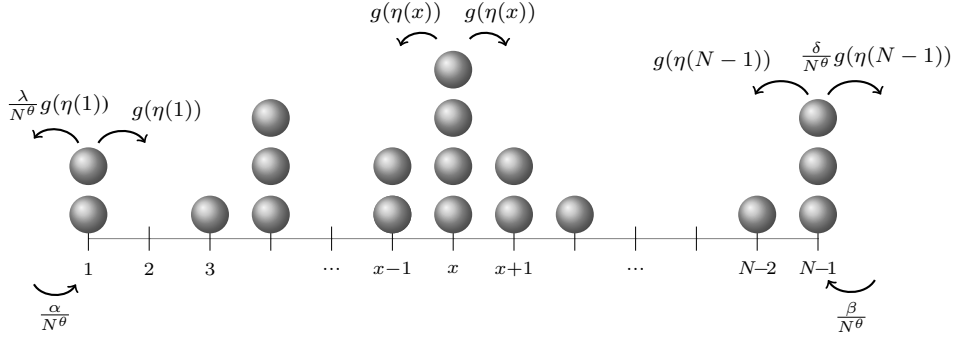
\begin{figure}
\begin{center}
\begin{tikzpicture}[thick, scale=0.8][h!]    
    \draw[step=1cm,gray,very thin] (-6,0) grid (6,0);   
    \foreach \y in {-6,...,6}{
    \draw[color=black,very thin] (\y,-0.2)--(\y,0.2);}    
    \draw (-6,-0.5) node [color=black] {$\scriptstyle{1}$};
    \draw (-5,-0.5) node [color=black] {$\scriptstyle{2}$};
    \draw (-4,-0.5) node [color=black] {$\scriptstyle{3}$};
    \draw (-2,-0.5) node [color=black] {$\scriptstyle{...}$};
    \draw (-1,-0.5) node [color=black] {$\scriptstyle{x-1}$};
    \draw (0,-0.5) node [color=black] {$\scriptstyle{x}$};
    \draw (1,-0.5) node [color=black] {$\scriptstyle{x+1}$};
    \draw (3,-0.5) node [color=black] {$\scriptstyle{...}$};
    \draw (5,-0.5) node [color=black] {$\scriptstyle{N\!-\!2}$};
    \draw (6,-0.5) node [color=black] {$\scriptstyle{N\!-\!1}$};
        
    \node[ball color=black!30!, shape=circle, minimum size=0.5cm] at (-6,1.2) {};
    \node[ball color=black!30!, shape=circle, minimum size=0.5cm]  at (-6.,0.4) {};
    \node[shape=circle,minimum size=0.63cm] (A) at (-6.,2) {};
    \node[shape=circle,minimum size=0.5cm] (C) at (-6.,1.3) {};
    \node[shape=circle,minimum size=0.5cm] (D) at (-4.9,1.3) {};
    \path [->] (C) edge[bend left=60] node[above right] {{\footnotesize{$g(\eta(1))$}}} (D);
    \node[shape=circle,minimum size=0.63cm] (OO) at (-7.1,1.3) {};
    \path [<-] (OO) edge[bend left=60] node[above] {{\footnotesize{$\frac{\lambda}{N^\theta}g(\eta(1)$)}}} (C);	
    \node[shape=circle,minimum size=0.63cm] (OO0) at (-7.1,-0.4) {};
    \node[shape=circle,minimum size=0.63cm] (C0) at (-6,-0.4) {};		
    \path [<-] (C0) edge[bend left=60] node[below ] {{\footnotesize{$\frac{\alpha}{N^\theta}$}}} (OO0);
    
    \node[ball color=black!30!, shape=circle, minimum size=0.5cm]  at (6.,0.4) {};
    \node[ball color=black!30!, shape=circle, minimum size=0.5cm] at (6.,1.2) {};
    \node[ball color=black!30!, shape=circle, minimum size=0.5cm] at (6.,2) {};
    \node[shape=circle,minimum size=0.5cm] (E) at (6.,2.1) {};
    \node[shape=circle,minimum size=0.5cm] (F) at (4.8,2.1) {};
    \path [->] (E) edge[bend right =60] node[above left] {{\footnotesize{$g(\eta(N-1))$}}} (F);
    \node[ball color=black!30!, shape=circle, minimum size=0.5cm]  at (5.,0.4) {};
    \node[shape=circle,minimum size=0.5cm] (H) at (7.2,2.1) {};
    \path [->] (E) edge[bend left=60] node[above] {
    {{\footnotesize{$\qquad\frac{\delta }{N^\theta}g(\eta(N-1))$}}}}(H);		
    \node[shape=circle,minimum size=0.5cm] (H0) at (7.1,-0.4) {}; 
    \node[shape=circle,minimum size=0.5cm] (E0) at (6,-0.4) {};		
    \path [->] (H0) edge[bend left=60] node[below ] {{\footnotesize{$\frac{\beta}{N^\theta}$}}} (E0);
    
    \node[ball color=black!30!, shape=circle, minimum size=0.5cm] (O) at (-4,0.4) {};
    \node[ball color=black!30!, shape=circle, minimum size=0.5cm]  at (-3.,0.4) {};
    \node[ball color=black!30!, shape=circle, minimum size=0.5cm]  at (-3.,1.2) {};
    \node[ball color=black!30!, shape=circle, minimum size=0.5cm]  at (-3.,2) {};
    \node[ball color=black!30!, shape=circle, minimum size=0.5cm]  at (-1.,0.4) {};
    \node[ball color=black!30!, shape=circle, minimum size=0.5cm]  at (-1.,1.2) {};		
    \node[ball color=black!30!, shape=circle, minimum size=0.5cm]  at (0.,0.4) {};
    \node[ball color=black!30!, shape=circle, minimum size=0.5cm]  at (0.,1.2) {};
    \node[ball color=black!30!, shape=circle, minimum size=0.5cm]  at (0.,2) {};
    \node[ball color=black!30!, shape=circle, minimum size=0.5cm]  at (0.,2.8) {};		
    \node[ball color=black!30!, shape=circle, minimum size=0.5cm]  at (1.,0.4) {};
    \node[ball color=black!30!, shape=circle, minimum size=0.5cm]  at (1.,1.2) {};		
    \node[ball color=black!30!, shape=circle, minimum size=0.5cm]  at (2.,0.4) {};		
    \node[ball color=black!30!, shape=circle, minimum size=0.5cm]  at (-1,0.4) {};
        
    \node[shape=circle,minimum size=0.63cm] (G1) at (0,3) {};
    \node[shape=circle,minimum size=0.63cm] (H1) at (1,2.8) {};
    \path [->] (G1) edge[bend left =60] node[above] {{\footnotesize{$\quad g(\eta(x))$}}} (H1);		
    \node[shape=circle,minimum size=0.63cm] (J1) at (-1,2.8) {};
    \path [<-] (J1) edge[bend left =60] node[above] {{\footnotesize{$g(\eta(x))\quad$}}} (G1);
\end{tikzpicture}
\caption{Microscopic dynamics of the symmetric zero-range process with open boundaries.}\label{fig.1}
\end{center}
\end{figure}

\subsection{Invariant Measures}
Let $\varphi_N:I_N\to[0, \infty)$ be a bounded, measurable function. We define the product measure $\bar{\nu}_{\varphi_N}$ on $\Omega_N$ with marginals given by 
\begin{equation}\label{eq:prod_measure}
    \bar{\nu}_{\varphi_{N}}\{\eta: \eta(x)=k\}:=\frac{1}{Z(\varphi_N(x))}\frac{(\varphi_N(x))^k}{g(k)!},
\end{equation}
for $x\in I_N$. Here, $g(k)!$ stands for $\prod_{1\le j\le k} g(j)$, with $g(0)!:=1$, and $Z$ is the partition function
\begin{equation*}
    Z(\varphi):=\sum_{k\geq0}\frac{\varphi^k}{g(k)!}.
\end{equation*} 
Calling $\varphi^*$ its radius of convergence, in order for the definition \eqref{eq:prod_measure} to make sense, we need the function $\varphi_N$ to satisfy
\begin{equation*}
    \varphi_N(x)\le \varphi^*
\end{equation*}
uniformly in $x\in I_N$ and $N\ge 1.$ 
Denoting by $E_{\nu}[\,\cdot\,]$ the expectation with respect to a measure $\nu$ on $\Omega_N$, a simple computation shows that
\begin{equation*}
    E_{\bar{\nu}_{\varphi_N}}[\eta(x)] = \frac{1}{Z(\varphi_N(x))}\sum_{k \ge 0} k\frac{(\varphi_N(x))^k}{g(k)!}.
\end{equation*}
Let now $R:[0,\varphi^*)\to\mathbb{R}$ be the function defined as the power series
\begin{equation*}
    R(\varphi ):= \frac{1}{Z(\varphi)}\sum_{k\ge0} k\frac{\varphi^k}{g(k)!},
\end{equation*}
where $\varphi^{*}$ is again the radius of convergence of the partition function $Z$. Then, we can rewrite the expected value of the occupation variable at site $x$ under $\bar{\nu}_{\varphi_N}$ as
\begin{equation}\label{eq:R}
    E_{\bar\nu_{\varphi_N}}[\eta(x)] = R(\varphi_N(x)),\ \ \ x\in I_N.
\end{equation}

The next lemma is reminiscent of Lemma 3.1 of \cite{fmn21}; however, in this case we are considering more general rates in the definition of the model, which allow injection and removal of particles on both sides of the segment. This result can be found in \cite[Equation 20]{lms05}.
\begin{lemma}\label{lemma:inv_measure}
For $\alpha, \lambda, \beta, \delta, \theta$ and $N$ satisfying 
\begin{equation}\label{eq:params_condition}
    \mymax\left\{\frac{\alpha\delta(N-2)+(\alpha+\beta)N^\theta}{\lambda\delta(N-2)+(\lambda+\delta)N^\theta}, \frac{\beta\lambda(N-2)+(\alpha+\beta)N^\theta}{\lambda\delta(N-2)+(\lambda+\delta)N^\theta}\right\}<\varphi^*,
\end{equation}
the measure $\bar\nu_{\bar{\varphi}_N}$, which will be denoted by $\bar{\nu}_{N}$, defined in \eqref{eq:prod_measure} with fugacity profile
\begin{equation}\label{eq:fugacity}
    \bar \varphi_N(x):=\displaystyle\frac{-(\alpha \delta - \beta \lambda)(x-1)+\alpha \delta (N-2) +(\alpha+\beta)N^{\theta}}{\lambda \delta (N-2)+(\lambda+ \delta )N^{\theta}}, \quad x \in I_N,
\end{equation}
is the unique invariant probability for the Markov process on $\Omega_N$ with infinitesimal generator $L_N$, defined in \eqref{eq:generator}. In what follows, we will call $\bar{\nu}_N$ the non-equilibrium stationary state (NESS).
\end{lemma}

\begin{remark}\label{remark:limit_fugacity_density}
Note that, if $\alpha, \lambda, \beta, \delta>0$, the asymptotic fugacity profile $\bar\varphi_\theta:[0, 1]\to[0, \infty)$ is given by:
\begin{enumerate}[i)]
    \item $0\le \theta<1$: $\bar\varphi_\theta(u)=\frac{-(\alpha\delta-\beta\lambda)u+\alpha\delta}{\lambda\delta}$;
    \item $\theta=1$: $\bar\varphi_\theta(u)=\frac{-(\alpha\delta-\beta\lambda)u+\alpha\delta+\alpha+\beta}{\lambda\delta+\lambda+\delta}$;
    \item $\theta>1$: $\bar\varphi_\theta(u)=\frac{\alpha+\beta}{\lambda+\delta}$.
\end{enumerate}
Then, the limit density profile $\bar\rho_\theta:[0, 1]\to[0, \infty)$ is given $\bar\rho_\theta(u)=R(\bar\varphi_\theta(u))$. In particular, a linear fugacity profile does \textit{not}, in general, imply a linear density profile. 
\end{remark}
 
Throughout, we will assume that $\mylim_{\varphi\uparrow\varphi*}Z(\varphi)=\infty.$ Then, as proved in \cite[Section 2.3]{kl99}, the range of $R$ is $[0, \infty)$, and $R$ is strictly increasing; then the map 
\begin{equation}\label{eq:Phi}
    \Phi:=R^{-1}:[0, \infty)\to[0, \varphi^*)
\end{equation} is well defined, and is a strictly increasing function.
Now, for any $x\in I_N$, let $\alpha_x \ge 0$: setting $\varphi_N(x):=\Phi(\alpha_x)$, we can define the measure $\nu_{\alpha}$ on $\Omega_N$ by 
\begin{equation}\label{eq:measure_alpha}
    \nu_{\alpha}(\cdot):={\bar\nu}_{\Phi(\alpha_\cdot)}(\cdot).
\end{equation}
This outputs a family $\{\nu_{\alpha}\}_{\alpha \ge 0}$ of product measures parametrised by the density: indeed, the expected value of the occupation variables $\eta(x)$ under $\nu_{\alpha}$ is equal to $\alpha$, namely
\begin{equation}\label{eq:expectation_alpha_eta}
    E_{\nu_{\alpha}}[\eta(x)]=R(\Phi(\alpha_x))=\alpha_x 
\end{equation}
for all $x\in I_N$ and any choice of parameters $\alpha_x\ge0$. Moreover, a simple computation shows that the function $\Phi(\alpha_x)$ is the expected value of the jump rate $g(\eta(x))$ under the measure:
\begin{equation}\label{eq:expectation_alpha_g}
    \Phi(\alpha_x)=E_{\nu_\alpha}[g(\eta(x))]\
\end{equation}
for every $x\in I_N$ and any $\alpha_x \geq 0$.
A particular case of the above happens when the values $\alpha_x$ are all equal to the same $\alpha\ge0$: then, equations \eqref{eq:expectation_alpha_eta} and \eqref{eq:expectation_alpha_g} become
\begin{equation*}
    E_{\nu_{\alpha}}[\eta(x)]=R (\Phi(\alpha))=\alpha\ \ \ \text{and} \ \ \ \Phi(\alpha)=E_{\nu_\alpha}[g(\eta(x))]
\end{equation*}
for every $x\in I_N$.

We are now ready to show the following result, which will be used repeatedly.
\begin{lemma}\label{lemma:Phi_lips} 
Recall the definition of $g^*$ given in \eqref{eq:function_g}. The function $\Phi:[0, \infty)\to [0,\varphi^*)$ appearing in \eqref{eq:Phi} is uniformly Lipschitz on $[0, \infty)$ with constant $g^*$:
\begin{equation*}
   \left|\Phi(\varphi_2)-\Phi(\varphi_1)\right|=\left|E_{\nu_{\varphi_2}}[g(\eta(x))]-E_{\nu_{\varphi_1}}[g(\eta(x))]\right| \le g^*|\varphi_2-\varphi_1|
\end{equation*}
for all non-negative reals $\varphi_1$ and $\varphi_2$.
\end{lemma}
\begin{proof}
Fix $\varphi_1 \leq \varphi_2$. Recall from \eqref{eq:expectation_alpha_g} that $\Phi(\alpha)$ is the expected value of $g(\eta(x))$, $x \in I_N$ under the measure $\nu_{\alpha}$:
\begin{equation}\label{eq:alpha}
    \Phi(\alpha)=E_{\nu_{\alpha}}[g(\eta(x))], \ \ \ x \in I_N.
\end{equation}
Denote by $\nu_{\varphi_1,\varphi_2}$ a measure on $\mathbb{N}^{I_N} \times \mathbb{N}^{I_N}$ whose first marginal is equal to $\nu_{\varphi_1}$, the second marginal is equal to $\nu_{\varphi_2}$, and which is concentrated on configurations $(\xi_1,\xi_2)$ such that $\xi_1 \le \xi_2$. The existence of this measure is guaranteed by \cite[Theorem II.2.4]{lig05} because $\nu_{\varphi_1} \le \nu_{\varphi_2}$. Hence, we have 
\begin{equation*}
    \left| \Phi(\varphi_2)-\Phi(\varphi_1)\right|=  E_{\nu_{\varphi_1, \varphi_2}} \left| g(\xi_2(x))-g(\xi_1(x))\right| \le g^* E_{\nu_{\varphi_1,\varphi_2}}\left|\xi_2(x)-\xi_1(x)\right|.
\end{equation*}
Since the measure $\nu_{\varphi_2,\varphi_1}$ is concentrated on configurations $\xi_1 \le \xi_2$ we can remove the absolute value in the last expression and obtain that is equal to $g^*(\varphi_2-\varphi_1)$.
\end{proof}

\begin{remark}\label{remark:Psi} 
The cylinder function $g(\eta(x))$ does not play any particular role in Lemma \ref{lemma:Phi_lips}. The statement applies to a broad class of cylinder functions that we now introduce. We say that a cylinder function $\Psi: \mathbb{N}^{I_N} \to \mathbb{R}$ is Lipschitz if there exists a constant $C_0$ such that 
\begin{equation}\label{eq:constant_C0}
    \left|\Psi(\eta)-\Psi(\xi)\right| \le C_0 \sum_{x \in I_N} |\eta(x)-\xi(x)|
\end{equation}
for all configurations $\eta, \xi\in \Omega_N$. Given a Lipschitz function $\Psi$, if we take $\xi$ to be the configuration with no particles in \eqref{eq:constant_C0}, we obtain that there exist finite constants $C_0$ and $C_1$ such that 
\begin{equation}\label{eq:Psi_bounded}
    \left|\Psi(\eta)\right| \le C_1+C_0 \sum_{x \in I_N} \eta(x) 
\end{equation}
for every configuration $\eta$. A cylinder function $\Psi$ is said to have sub-linear growth if, for each $\delta>0$, there exists a finite constant $C(\delta)$ such that
\begin{equation*}
    |\Psi(\eta)| \le C(\delta)+ \delta\sum_{x \in I_N} \eta(x)
\end{equation*}
for every configuration $\eta\in\Omega_N$. Every bounded cylinder function has a sub-linear growth. For a cylinder function $\Psi$ satisfying \eqref{eq:Psi_bounded}, we denote by $\tilde{\Psi}:[0, \infty)] \to \mathbb{R}$ the function whose value at some density $\alpha \ge 0$ is equal to the expectation of $\Psi$ with respect to the invariant measure $\nu_{\alpha}$, namely
\begin{equation}\label{eq:tildePsi}
    \tilde{\Psi}(\alpha) := E_{\nu_{\alpha}} [\Psi(\eta)].
\end{equation}
With the notation introduced, we have the identity $\tilde{g}=\Phi$, because of \eqref{eq:alpha}.
\end{remark}

The proof of Lemma \ref{lemma:Phi_lips} and Remark \ref{remark:Psi} show that $\tilde{\Psi}$ is a uniform Lipschitz function provided $\Psi$ is a Lipschitz cylinder function:

\begin{corollary}\label{corol:Psi}
Let $\Psi$ be a Lipschitz cylinder function. The function $\tilde{\Psi}:[0, \infty)\to\mathbb{R}$ defined by \eqref{eq:tildePsi} is uniformly Lipschitz on $[0, \infty)$: there exists a finite constant $C_0=C_0(\Psi)$ such that
\begin{equation*}
    \left|\tilde{\Psi}(\alpha_2)-\tilde{\Psi}(\alpha_1)\right| \le C_0|\alpha_2-\alpha_1|
\end{equation*}
for all $\alpha_1,\alpha_2\ge 0$.
\end{corollary}

\subsection{Assumptions on the Jump Rate}
\hspace{.5cm} \textsc{\bfseries Assumption (ND).} We assume the following condition on the jump rate: 
\begin{equation}\label{g_non_decreasing}
    g(\cdot) \text{ is non decreasing}.
\end{equation}
As proved below, under this assumption our zero-range processes are \textit{attractive}. This notion means that the processes are monotone in the sense that  their semigroups $S(t):=\mathbb E_{\eta}[f(\eta_t)]$ preserve the partial order: $\mu_1\le\mu_2\Rightarrow \mu_1S(t)\le \mu_2S(t)$ for all $t\ge0$.\footnote{For a measure $\mu$ on $\Omega_N$, the measure $\mu S(t)$ represents the time-$t$ distribution of the process with semigroup $S(t)$, and $\expected_{\eta}[f(\eta_t)]$ denotes the expectation of $f(\eta_t)$ with respect to $\prob_{\mu^N}$ where the starting measure is $\mu^N=\delta_\eta$ for a fixed configuration $\eta$.} Above, the partial order between measures is defined as follows. We say that a function $f:\Omega_N\to\mathbb R$ is  \textit{monotone} if $f(\eta)\le f(\xi)$ for all $\eta\le\xi$ (this last inequality means that  $\eta(x)\le \xi(x)$ for every $x\in I_N$), and we say that 
\begin{equation}\label{eq:def_domination}
    \mu_1\le\mu_2  \ \ \ \text{ if }\ \ \ \int f \de \mu_1\le \int f\de \mu_2
\end{equation}
for all monotone functions $f:\Omega_N\to\mathbb R$. 

\begin{lemma}\label{lemma:attractiveness} Under assumption (\ref{g_non_decreasing}), the zero-range process on $I_N$ generated by $\mathcal{L}^N$, given in \eqref{eq:generator}, is attractive.  
\end{lemma}
\begin{proof} We adapt the proof of \cite[Theorem 2.5.2]{kl99}. Let $\mu_1, \mu_2$ be two probability measures on $\mathbb{N}^{I_N}$ such that $\mu_1\le \mu_2$. By \cite[Theorem II.2.4]{lig05}, there exists a probability measure $\bar\mu$ on $\mathbb{N}^{I_N}\times \mathbb{N}^{I_N}$ such that the first marginal is $\mu_1$, the second one is $\mu_2$, and $\bar\mu\{(\eta, \xi): \eta\le \xi\}=1$. Consider the Markov process $\{(\eta_t, \xi_t), t\ge0\}$ on $\mathbb{N}^{I_N}\times \mathbb{N}^{I_N}$ starting from $\bar\mu$ and with generator $\mathcal{\bar L}^N= \mathcal{\bar L}^N_0+ \kappa\mathcal{\bar L}^N_b$, where
\begin{align*}
    \mathcal{\bar L}^N_0f(\eta, \xi)&:=\sum_{\substack{x, y\in I_N \\ |x-y|=1}} \mymin\{g(\eta(x)) g(\xi(x))\}\left[f(\eta^{x, y}, \xi^{x, y})-f(\eta, \xi)\right]
    \\&\phantom{:=}+\sum_{\substack{x, y\in I_N \\ |x-y|=1}} \{g(\eta(x))-g(\xi(x))\}^+\left[f(\eta^{x, y}, \xi)-f(\eta, \xi)\right]
    \\&\phantom{:=}+\sum_{\substack{x, y\in I_N \\ |x-y|=1}} \{g(\xi(x))-g(\eta(x))\}^+\left[f(\eta, \xi^{x, y})-f(\eta, \xi)\right]
\end{align*}
and
\begin{align*}
    \mathcal{\bar L}^N_bf(\eta, \xi)&:=\frac{\alpha}{N^\theta} \left[f(\eta^{1+}, \xi^{1+})-f(\eta, \xi)\right]
    \\&\phantom{:=}+\frac{\lambda \mymin\{g(\eta(1)), g(\xi(1))\}}{N^\theta}\left[f(\eta^{1-}, \xi^{1-})-f(\eta, \xi)\right]
    \\&\phantom{:=}+\frac{\lambda \{g(\eta(1))-g(\xi(1))\}^+}{N^\theta}\left[f(\eta^{1-}, \xi)-f(\eta, \xi)\right]
    \\&\phantom{:=}+\frac{\lambda \{g(\xi(1))-g(\eta(1))\}^+}{N^\theta}\left[f(\eta, \xi^{1-})-f(\eta, \xi)\right]
    \\&\phantom{:=}+\frac{\beta}{N^\theta} \left[f(\eta^{(N-1)+}, \xi^{(N-1)+})-f(\eta, \xi)\right]
    \\&\phantom{:=}+\frac{\delta \mymin\{g(\eta(N-1)), g(\xi(N-1))\}}{N^\theta}\left[f(\eta^{(N-1)-}, \xi^{(N-1)-})-f(\eta, \xi)\right]
    \\&\phantom{:=}+\frac{\delta \{g(\eta(N-1))-g(\xi(N-1))\}^+}{N^\theta}\left[f(\eta^{(N-1)-}, \xi)-f(\eta, \xi)\right]
    \\&\phantom{:=}+\frac{\delta \{g(\xi(N-1))-g(\eta(N-1))\}^+}{N^\theta}\left[f(\eta, \xi^{(N-1)-})-f(\eta, \xi)\right].
\end{align*}
It is not hard to check that both coordinates evolve as zero-range processes with open boundaries on $I_N$ with generator $\mathcal{L}^N$ given in (\ref{eq:generator}). 

Now let $\bar{\prob}_{\mu}$ denote the measure on the Skorokhod space of càdlàg paths $\mathcal{D}([0, \infty), \mathbb{N}^{I_N}\times \mathbb{N}^{I_N})$  corresponding to the Markov process with generator $\mathcal{\bar L}^N$ starting from the measure $\mu$: by the same argument given in \cite{kl99}, it suffices to show that, if $\eta\le\xi$, then
\begin{equation*}
    \bar{\prob}_{\delta_{(\eta, \xi)}}\left\{\eta_t\le \xi_t\right\}=1
\end{equation*}
for all $t\ge0$, where $\delta_{(\eta, \xi)}$ is the Dirac probability concentrated on $(\eta, \xi)$. 

Let $\mathcal{A}:=\{(\eta, \xi): \eta\le \xi\}$. Since the jump rate is non-decreasing, the generator $\mathcal{\bar L}^N$ does \textit{not} admit jumps outside of $\mathcal{A}$, in the sense that $\mathcal{\bar L}^N\boldsymbol{1}_\mathcal{A}\ge0$. This inequality restricted to the bulk generator follows from the argument in \cite{kl99}. As for the boundary generator, note that if $\boldsymbol{1}_\mathcal{A}(\eta, \xi)=0$, then clearly $\mathcal{\bar L}^N_b\boldsymbol{1}_{\mathcal{A}}(\eta, \xi)\ge0$; on the event $\mathcal{A}$, instead, we see that
\begin{align*}
    \mathcal{\bar L}^N_b\boldsymbol{1}_{\mathcal{A}}(\eta, \xi)&=\frac{\alpha}{N^\theta} \left[\boldsymbol{1}_{\mathcal{A}}(\eta^{1+}, \xi^{1+})-\boldsymbol{1}_{\mathcal{A}}(\eta, \xi)\right]
    \\&\phantom{=}+\frac{\lambda g(\eta(1))}{N^\theta}\left[\boldsymbol{1}_{\mathcal{A}}(\eta^{1-}, \xi^{1-})-\boldsymbol{1}_{\mathcal{A}}(\eta, \xi)\right]
    \\&\phantom{=}+\frac{\lambda \{g(\xi(1))-g(\eta(1))\}}{N^\theta}\left[\boldsymbol{1}_{\mathcal{A}}(\eta, \xi^{1-})-\boldsymbol{1}_{\mathcal{A}}(\eta, \xi)\right]
    \\&\phantom{=}+\frac{\beta}{N^\theta} \left[\boldsymbol{1}_{\mathcal{A}}(\eta^{(N-1)+}, \xi^{(N-1)+})-\boldsymbol{1}_{\mathcal{A}}(\eta, \xi)\right]
    \\&\phantom{=}+\frac{\delta  g(\eta(N-1))}{N^\theta}\left[\boldsymbol{1}_{\mathcal{A}}(\eta^{(N-1)-}, \xi^{(N-1)-})-\boldsymbol{1}_{\mathcal{A}}(\eta, \xi)\right]
    \\&\phantom{=}+\frac{\delta \{g(\xi(N-1))-g(\eta(N-1))\}}{N^\theta}\left[\boldsymbol{1}_{\mathcal{A}}(\eta, \xi^{(N-1)-})-\boldsymbol{1}_{\mathcal{A}}(\eta, \xi)\right].
\end{align*}
But now, $\{g(\xi(1))-g(\eta(1))\}$ can only be strictly positive if $\xi(1)$ is strictly greater than $\eta(1)$, in which case $\boldsymbol{1}_{\mathcal{A}}(\eta, \xi^{1-})=1$, and a similar argument applies to the right boundary term, which yields $\mathcal{\bar L}^N_b\boldsymbol{1}_{\mathcal{A}}\ge0$. From here, the claim follows from the exact same argument given in the proof of \cite[Theorem 2.5.2]{kl99}.\end{proof}

\textsc{\bfseries Assumption (SG).} Let $\text{gap}(j, \ell)$ denote the spectral gap of the generator of a zero-range process on $\{1, \ldots, \ell\}$ with $j$ particles and jump rate $g$. We assume that $g$ satisfies the following: 
\begin{equation}\label{spectral_gap}
    \text{there exists } C=C(g)>0 \text{ such that gap} (\ell, j)\ge \frac{C}{\ell^2} \text{ for all } \ell\ge2 \text{ and } j\ge0.
\end{equation}
Examples of jump rates $g$ for which \eqref{spectral_gap} has been proved include the following:
\begin{enumerate}[i)]  
    \item $g(k)=k^\gamma$ for $0<\gamma<1$ \cite{nag10};
    \item $g(k)=\boldsymbol{1}_{[1, \infty)}(k)$ \cite{mor06};
    \item $g$ satisfying \eqref{eq:function_g} and such that there exist $M\in\mathbb{N}$ and $a>0$ such that $g(k)-g(j)\ge a$ for all $k\ge j+M$ \cite{lsv96}.
\end{enumerate}

Now let $\ell$ be a positive integer and define the size-$\ell$ boxes $\Lambda_1^\ell:=\{1, \ldots, \ell\}$ and $\Lambda_w^\ell:=\{w, \ldots, w+\ell\}$ for some integer $w>\ell$. Consider a ``connected" zero-range process evolving as a standard symmetric zero-range process on $\Lambda_1^\ell$ and $\Lambda_w^\ell$, but with particles also allowed to jump from $1$ to $w$ and vice versa. More precisely, the generator of this process is defined via $\mathcal{\bar G}^{\ell, \ell}:=\mathcal{G}^{1, \ell}+\mathcal{G}^{w, \ell}+\mathcal{G}^{\text{couple}}$, where these generators act on maps $f:(\Lambda_1^\ell\cup \Lambda_w^\ell)^\mathbb{N}\to\mathbb{R}$ via
\begin{align*}
    &\mathcal{G}^{1, \ell}f(\eta):=\sum_{\substack{x, x'\in\Lambda_1^\ell \\ |x-x'|=1}}g(\eta(x))[f(\eta^{x, x'})-f(\eta)],
    \\&\mathcal{G}^{w, \ell}f(\eta):=\sum_{\substack{x, x'\in\Lambda_w^\ell \\ |x-x'|=1}}g(\eta(x))[f(\eta^{x, x'})-f(\eta)],   
    \\&\mathcal{G}^{\text{couple}}f(\eta):=g(\eta(1))[f(\eta^{1, w})-f(\eta)]+g(\eta(w))f(\eta^{w, 1})-f(\eta)].
\end{align*}
Given a positive integer $j$, let $\mathcal{\bar G}^{\ell, \ell}_j$ denote the restriction of $\mathcal{\bar G}^{\ell, \ell}$ to the configuration space $\Theta_j:=\{\eta: \sum_{x\in{\Lambda_1^\ell\cup \Lambda_w^\ell}}\eta(x)=j\}$, and let $\bar\mu_j$ denote its canonical invariant measure. The next result will be important in the proof of the two-block estimate in Section \ref{sec:replacement_bulk}, and it essentially states that  
the spectral gap of this generator satisfies a lower bound depending only on $g, j$ and $\ell$.
\begin{lemma}\label{lemma:coupled_gap} Under assumption \eqref{spectral_gap}, 
the spectral gap of $\mathcal{\bar G}^{\ell, \ell}_j$ satisfies 
\begin{equation*}
    \text{gap}(\ell, \ell, j)\ge\mymin\left\{\frac{C(g)}{\ell^2}, C(g, j, \ell)\right\},
\end{equation*}   
where $C(g)$ is the same constant as in \eqref{spectral_gap} and $C(g, j, \ell)$ denotes a positive constant which only depends on $g, j, \ell$.
\end{lemma}
\begin{proof} Let $f$ be a function in $L^2(\bar\mu_j)$ and let $N_1(\cdot)$ be the map on $\Theta_j$ that counts the number of particles in the box $\Lambda_1$, namely $N_1(\eta):=\sum_{x\in\Lambda_1^\ell}\eta(x)$. By the law of total variance, we can write
\begin{equation}\label{eq:total_variance}
    \text{Var}_{\bar\mu_j}(f)=E_{\bar\mu_j}[\text{Var}_{\bar\mu_j}(f|N_1)]+\text{Var}_{\bar\mu_j}(E_{\bar\mu_j}[f|N_1]).
\end{equation}
Now, conditionally on $N_1=k$, the dynamics on the two boxes are independent, and in particular each evolves as a symmetric zero range process with $k$ and $j-k$ particles, respectively. Moreover, the conditional measure $\bar\mu_j\{\cdot\,| N_1=k\}$ can be written as 
\begin{equation*}
    \bar\mu_j\{(\xi_1, \xi_w)| N_1(\xi_1, \xi_w)=k\}=\bar\mu^1_k(\xi_1)\bar\mu^w_{j-k}(\xi_w), \ \ \ \ \ (\xi_1, \xi_w)\in(\Lambda_1^\ell)^\mathbb{N}\times(\Lambda_w^\ell)^\mathbb{N},
\end{equation*}
with $\bar\mu^1_k$ being the canonical invariant measure of $\mathcal{G}^{1, \ell}$ restricted to $\{\xi_1: \sum_{x\in \Lambda_1^\ell}\eta(x)=k\}$ and with $\bar\mu^w_{j-k}$ defined analogously. Hence, by definition of spectral gap, we have that
\begin{equation*}
    \text{Var}_{\bar\mu_j}\left(f|N_1=k\right)\le\frac{\langle-\mathcal{G}^{1, \ell}f, f\rangle_{\bar\mu_j\{\cdot\,|N_1=k\}}}{\text{gap}(\ell, k)}+\frac{\langle-\mathcal{G}^{w, \ell}f, f\rangle_{\bar\mu_j\{\cdot\,|N_1=k\}}}{\text{gap}(\ell, j-k)}.
\end{equation*} 
Taking expectations with respect to $\bar\mu_j$
on both sides and using \eqref{spectral_gap}, we get
\begin{equation}\label{eq:single_variance}                  
    E_{\bar\mu_j}\left[\text{Var}_{\bar\mu_j}\left(f|N_1\right)\right]\le \frac{\ell^2}{C(g)}\left[\langle-\mathcal{G}^{1, \ell}f, f\rangle_{\bar\mu_j}+\langle-\mathcal{G}^{w, \ell}f, f\rangle_{\bar\mu_j}\right].
\end{equation}
Define now the map $h:\{0, \ldots, j\}\to\mathbb{R}$ as $h(k):=E_{\bar\mu_j}[f|N_1=k]$, so that $H:=h\circ N_1$ recovers $H=E_{\bar\mu_j}[f|N_1]$. Note that
\begin{equation*}
    E_{\bar\mu_j}(H)=\sum_{\eta\in\Theta_j} H(\eta)\bar\mu_j(\eta)=\sum_{k=0}^j h(k)\sum_{\eta: N_1(\eta)=k}\bar\mu_j(\eta)=E_{\pi_j}[h],
\end{equation*}
where $\pi_j(k):=\bar\mu_j\{N_1=k\}$ denotes the marginal law of $N_1$. A similar computation shows that $E_{\bar\mu_j}(H^2)=E_{\pi_j}[h^2]$, so that
\begin{equation}\label{eq:variance_equality}
    \text{Var}_{\mu_j}\left(E_{\bar\mu_j}[f|N_1]\right)=\text{Var}_{\pi_j}(h).
\end{equation}
Now, defining  $\nabla h(\cdot):=h(\cdot+1)-h(\cdot)$, we get
\begin{align*}
    \text{Var}_{\pi_j}(h)&=\frac{1}{2}\sum_{m, k=0}^j [h(k)-h(m)]^2\pi_j(k)\pi_j(m)
    \\&\le j\sum_{l=0}^{j-1} (\nabla h(l))^2\sum_{m\le l<k}\pi_j(m)\pi_j(k)
    \\&\le j \mymax_{l=1, \ldots, j} \frac{1}{\pi_j(l)} E_{\pi_j}[(\nabla h)^2].
\end{align*}
But now, it is easy to see that, for each fixed $j$, $\mymin_{l=1, \ldots, j} \pi_j(l)\ge C(g, j, \ell)$  for some strictly positive constant depending only on $g, j$ and $\ell$, so the last inequality yields 
\begin{equation}\label{eq:poincare}
    \text{Var}_{\pi_j}(h)\le \frac{j}{C(g, j, \ell)}E_{\pi_j}[(\nabla h)^2],
\end{equation}
On the other hand, since $H=E_{\bar\mu_j}[f|N_1]$ and by convexity of the Dirichlet form, we see that
\begin{align*}
    \langle-\mathcal{G}^{\text{couple}}f, f\rangle_{\bar\mu_j}&\ge \langle-\mathcal{G}^{\text{couple}}H, H\rangle_{\bar\mu_j}
    \\&\ge g(1)\expected_{\bar\mu_j}\big[[H(\eta^{1, w})-H(\eta)]^2+[H(\eta^{w, 1})-H(\eta)]^2\big]
    \\&\ge g(1)\sum_{k=0}^{j-1}[h(k+1)-h(k)]^2\sum_{\eta: N_1(\eta)=k}\bar\mu_j(\eta)
    \\&=g(1) E_{\pi_j}[(\nabla h)^2],
\end{align*}
so that, by \eqref{eq:poincare}, $\text{Var}_{\pi_j}(h)\le \frac{j}{g(1)C(g, j, \ell)}\langle-\mathcal{G}^{\text{couple}}f, f\rangle_{\bar\mu_j}$. Finally, absorbing the factor $\frac{g(1)}{j}$ into the constant $C(g, j, \ell)$, by \eqref{eq:total_variance}, \eqref{eq:single_variance} and \eqref{eq:variance_equality} we get
\begin{equation*}
    \text{Var}_{\bar\mu_j}(f)\le \mymax\left\{\frac{\ell^2}{C(g)}, \frac{1}{C(g, j, \ell)} \right\}\langle-\mathcal{\bar G}^{\ell, \ell}_j f, f\rangle_{\bar\mu_j}.
\end{equation*}
By definition of spectral gap, the claim follows.
\end{proof}

\begin{remark} It is worth noting that \textsc{Assumption (SG)} and Lemma \ref{lemma:coupled_gap} are only used in the proofs of the one-block and two-block estimates in Sections \ref{sec:replacement_bulk} and \ref{sec:replacement_boundary}. In fact, we do not need a lower bound of order $\frac{1}{\ell^2}$ for the spectral gap on a box of size $\ell$, as long as we have \textit{any} positive lower bound depending only the rate $g$, the size of the box, and the total number of particles in the box. 
\end{remark}

\subsection{The Hydrodynamic Equation}
Let $C^{m,n}([0, T]\times[0, 1])$ denote the set of continuous functions on $[0, T]\times[0, 1]$ which are $m$ times differentiable in the first variable and $n$ in the second, with continuous derivatives, where $m$ and $n$ are positive integers. We will also denote by $\mathcal{H}^1$ the Sobolev space $\mathcal H^1(0, 1)$, equipped with the norm
\begin{equation*}
    \|\cdot\|_{\mathcal{H}^1}^2:=\|\cdot\|^2_{L^2((0, 1))}+\|\nabla\cdot \|^2_{L^2((0, 1))},
\end{equation*}
and by $L^2([0, T], \mathcal{H}^1)$ the set of measurable functions $f:[0, T]\to\mathcal{H}^1$ with $\int_0^T\|f_t\|_{\mathcal{H}^1}^2\de t<\infty$. Our goal is to show that $\rho:[0, T]\times [0, 1]\to[0, \infty)$ is the unique weak solution of the PDE with boundary conditions 
\begin{equation}\label{eq:pde_zr}
    \left\{ 
    \begin{array}{ll}
    \partial_t \rho_t(u) = \Delta \Phi(\rho_t(u)) &\mbox{ for } u \in (0,1) \mbox{ and } t \in (0,T],
    \\\partial_u \Phi (\rho_t(0)) = \tilde\kappa(\lambda \Phi(\rho_t(0))-\alpha) &\mbox{ for } t \in (0,T],
    \\\partial_u \Phi (\rho_t(1)) = \tilde\kappa(\beta-\delta\Phi(\rho_t(1))) &\mbox{ for } t \in (0,T],
    \\\rho_0(u) = \gamma(u) &\mbox{ for } u \in [0,1],
    \end{array}
    \right.
\end{equation}
where $\tilde\kappa=\kappa\boldsymbol{1}_{\{1\}}(\theta)$ and the function $\Phi$ was defined in \eqref{eq:Phi}. Recall that $\Phi$ is a strictly increasing function that satisfies a Lipschitz condition.

\begin{definition}\label{def:weak_solutions}  Let $\gamma:[0,1]\to[0, \infty)$ be a measurable and integrable function. We say that $\rho:[0, T]\times [0, 1]\to[0, \infty)$ is a weak solution of the PDE (\ref{eq:pde_zr}) if:
\begin{enumerate}[i)]
\item $\rho \in L^2([0, T]\times [0, 1])$,
\item $\Phi(\rho)\in L^2([0,T], \mathcal H^1)$,
\item for any $t\in[0,T]$ and $G\in C^{1,2}([0,T]\times [0,1])$,
\begin{equation} \label{eq:intergal_pde}
    \begin{split}
    \langle\rho_t, G_t\rangle &- \langle\gamma, G_0\rangle - \int_0^t\{\langle \rho_s, \partial_s G_s\rangle + \langle\Phi(\rho_s), \Delta G_s\rangle\}\de s
    \\&+\int_0^t \{\Phi(\rho_s(1))\partial_u G_s(1)-\Phi(\rho_s(0))\partial_u G_s(0) \}\de s
    \\&-\tilde{\kappa}\int_0^t \big\{(\beta-\delta\Phi(\rho_s(1)))G_s(1) - (\lambda \Phi(\rho_s(0))-\alpha)G_s(0) \big\}\de s = 0.
    \end{split}
\end{equation}
\end{enumerate}
\end{definition}

Above, we used the notation $\langle \cdot, \cdot\rangle$ for the inner product $\langle f, h\rangle:=\int_0^1 f(u)h(u)\de u$ on the set of measurable functions in $L^2([0,1])$. 

\begin{lemma}\label{lemma:uniqueness}
There exists a unique weak solution of \eqref{eq:pde_zr} in the sense of the definition above. 
\end{lemma}

The proof of this lemma is given in Appendix
\ref{sec:app_uniqueness}.

\subsection{The Hydrodynamic Limit}
Let $\mathcal{M}_+$ be the space of finite, positive measures on $[0,1]$ endowed with the weak topology. Given a configuration $\eta\in\Omega_N$, we associate to it the empirical measure on $[0,1]$ defined by
\begin{equation*}
    \pi^{N}_t(\de u) := \frac{1}{N} \sum_{x\in I_N} \eta_t(x) \delta_{\frac{x}{N}}(\de u),
\end{equation*}
where $\delta_u$ stands for the Dirac measure at $u\in[0,1]$. Note that the empirical-density process $\{\pi^N(\eta_t), t\ge0\}$ is a Markov process on the space $\mathcal{M}_+$.

Given a measurable function $H:[0,1]\to\mathbb R$, let us denote its integral with respect to the measure $\pi_t^N$  by
\begin{equation*}
    \langle \pi_t^N, H\rangle := \frac{1}{N}\sum_{x\in I_N} \eta_t(x) H\Big(\frac xN\Big).
\end{equation*} 
Given a measure $\mu^N$ on $\Omega^N$, we denote by $\prob_{\mu^N}$ the probability on the Skorokhod space of càdlàg trajectories $\mathcal D([0,T], \Omega_N)$ corresponding to the jump process $\{\eta_t: t\in[0, T]\}$ with generator $N^2\mathcal{L}^N$ and initial distribution $\mu^N$. Expectations with respect to $\prob_{\mu^N}$ will be denoted by $\expected_{\mu^N}[\,\cdot\,]$, and we will denote by $Q^N$ the probability on $\mathcal D([0,T], \mathcal M_+)$ defined by $Q^N:=\prob_{\mu^N}(\pi^N)^{-1}$.

The conservation of the number of particles is an extensively used property in the proof of tightness for the classical zero-range process on the torus, together with a hypothesis that controls the relative entropy of the initial distribution $\mu^N$ with respect to the invariant measure; see \cite[Lemma 5.1.5]{kl99}. Since we do not have  the conservation law in our case, a different approach is necessary: together with a relative entropy bound, we also assume a stochastic domination hypothesis, namely
\begin{equation}\label{eq:stoch_domination}
    \mu^N\le \bar\nu_N
\end{equation}
in the sense of \eqref{eq:def_domination}, where $\bar\nu_N$ is the invariant measure. Hypothesis \eqref{eq:stoch_domination}, along with attractiveness, provides us with a way to control the number of particles in the system as time evolves. 

In order to properly state our result, we first impose an additional condition on the sequence of initial measures. To that end, we introduce the following notion: 

\begin{definition}\label{def:associated_profile}
A sequence $\{\mu^N\}_N$ of probability measures on $\Omega_N$ is said to be \textit{associated} to a measurable and integrable profile $\gamma:[0,1]\to\mathbb[0, \infty)$ if,  for any $ \delta>0$ and any continuous function $H:[0,1]\to\mathbb{R}$, the following limit holds:
\begin{equation}\label{eq:initial_profile}
    \mylim_{N\to\infty}
    \mu^N \left\{\eta\in\Omega_N: \left|\frac{1}{N} \sum_{x=1}^{N-1} H\left(\frac{x}{N}\right)\eta(x)
    - \int_0^1 H(u)\gamma(u)\de u \right| > \delta\right\} = 0.
\end{equation}
\end{definition}
Our main result is the following.
\begin{theorem}\label{thm:main}
Let $\gamma:[0,1]\to[0, \infty)$ be a  measurable  and integrable function and let $\{\mu^N\}_N$ be a sequence of probability measures on $\Omega_N$ associated to $\gamma$ and satisfying the two conditions 
\begin{equation*}
    H(\mu^N|\bar\nu_N)\lesssim N \ \ \ \textrm{and} \ \ \ \mu^N\le\bar\nu_N.
\end{equation*}
Then, for $\theta\ge1$ and for every $t\in [0,T]$, the sequence of empirical measures $\{\pi^{N}_t\}_N$ converges in probability to the absolutely continuous measure $\rho(t, u)\de u$, that is, for any $\delta>0$ and for any $H\in C^2(\mathbb{T})$, we have that
\begin{equation*}
    \mylim_{N\to\infty}\prob_{\mu^N} \left\{\left|\langle \pi^{N}_{t},H\rangle-\int_0^1\rho(t,u)H(u)\de u \right|>\delta\right\} =0,
\end{equation*}
where $\rho(t,u)$ is the unique weak solution of \eqref{eq:pde_zr} with initial condition $\gamma$.   
\end{theorem}

\begin{remark} As a consequence of \cite[Lemma 2.3.5]{kl99}, the stochastic domination hypothesis \eqref{eq:stoch_domination} holds if, for instance, $\mu^N$ is a product measure of the form \eqref{eq:prod_measure} associated to a fugacity function bounded from above by the fugacity of the stationary measure $\bar\nu_N$. Moreover, because of this hypothesis, the profile $\gamma$ appearing in the statement of Theorem \ref{thm:main} needs to be bounded from above by the profile $\bar\rho_\theta$ given in Remark \ref{remark:limit_fugacity_density}. A natural sequence $\{\mu^N\}_N$ satisfying both \eqref{eq:stoch_domination} and \eqref{eq:initial_profile} is the sequence $\{\nu^N_{\gamma(\cdot)}\}_N$ of product measures with slowly varying parameter associated to a profile $\gamma:[0,1]\to[0, \infty)$ such that $\gamma(u)\le\bar\rho_\theta(u)$ for all $u\in[0,1]$; see \cite[Definition 4.2 and Proposition 4.1]{fmn21} for details.
\end{remark}

The proof of this theorem follows the so-called entropy method developed by \cite{gpv88} and it consists in two main steps. We first prove that the sequence of empirical measures $\{\pi_t^N(\de u)\}_N$ is tight, so that from Prohorov's Theorem we know that there exists a limit point $\pi_t(\de u)$. Then, we uniquely characterise this limit point by showing that it is absolutely continuous with respect to the Lebesgue measure, namely $\pi_t(\de u)=\rho(t,u)\de u$ where $\rho(t,u)$ is the unique weak solution of the hydrodynamic equation \eqref{eq:pde_zr}.

\subsection{The Hydrostatic Limit}
We note that, as a corollary of the previous result, we can obtain the hydrostatic limit of our process. All we need to prove is that the invariant measure $\bar\nu_N$ satisfies the conditions of Theorem \ref{thm:main}. Note that the entropy bound and the stochastic domination hypothesis trivially hold, so it only remains to show that $\bar\nu_N$ is associated to a measurable and integrable profile. The natural candidate is the stationary solution of the hydrodynamic equation, namely the solution $\bar\rho_\theta(u)$ to
\begin{equation*}
    \left\{ 
    \begin{array}{ll}
    \Delta \Phi(\bar\rho_\theta(u))=0 &\mbox{ for } u \in (0,1),
    \\\partial_u \Phi (\bar\rho_\theta(0)) = \tilde\kappa(\lambda \Phi(\bar\rho_\theta(0))-\alpha),
    \\\partial_u \Phi (\bar\rho_\theta(1)) = \tilde\kappa(\beta-\delta\Phi(\bar\rho_\theta(1))).
    \end{array}
    \right.
\end{equation*}
Then, we need to control
\begin{equation*}
    \bar\nu_N \left\{\eta\in\Omega_N: \left|\frac{1}{N} \sum_{x=1}^{N-1} H\left(\frac{x}{N}\right)\eta(x)
    - \int_0^1 H(u)\bar\rho_\theta(u)\de u \right| > \delta\right\},
\end{equation*}
which can be bounded from above by the sum of 
\begin{equation}\label{eq:hydrostatic_1}
    \bar\nu_N \left\{\eta\in\Omega_N: \left|\frac{1}{N} \sum_{x=1}^{N-1} H\left(\frac{x}{N}\right)\big(\eta(x)-R(\varphi_N(x))\big)
    \right| > \frac \delta 2\right\}
\end{equation}
and
\begin{equation}\label{eq:hydrostatic_2}
    \bar\nu_N \left\{\eta\in\Omega_N: \left|\frac{1}{N} \sum_{x=1}^{N-1} H\left(\frac{x}{N}\right)R(\varphi_N(x))
    - \int_0^1 H(u)\bar\rho_\theta(u)\de u \right| > \frac \delta 2\right\},
\end{equation}
with the function $R$ defined in \eqref{eq:R}. Since $R(\varphi_N(x))=E_{\bar\nu_N}[\eta(x)]$, by the Chebyshev's inequality and since $\bar\nu_N$ is a product measure, we see that
\begin{align*}
    \eqref{eq:hydrostatic_1}&\le \frac{4}{\delta^2} E_{\bar\nu_N}\left[ \left(\frac{1}{N}\sum_{x=1}^{N-1} H\left(\frac{x}{N}\right)\big(\eta(x)-R(\varphi_N(x)\big)\right)^2\right]
    \\& \le\frac{4}{\delta^2N^2} \sum_{x=1}^{N-1} H\left(\frac{x}{N}\right)^2 \text{Var}_{ \bar\nu_N}\big(\eta(x)\big)\lesssim\frac{1}{N},
\end{align*}
where in the last inequality we used Lemma \ref{lemma:bounded_moments}. Now, to treat \eqref{eq:hydrostatic_2}, note that the expression inside $\bar\nu_N$ is deterministic: then, recalling Remark \ref{remark:limit_fugacity_density} and since
$\bar\rho_\theta(u)=R(\bar\varphi_\theta(u))$, we need to prove that
\begin{equation*}
    \mylim_{N\to\infty}\frac{1}{N} \sum_{x=1}^{N-1} H\left(\frac{x}{N}\right)R(\varphi_N(x))
    = \int_0^1 H(u)R(\bar\varphi_\theta(u))\de u.
\end{equation*}
This can be readily done by noting that $\mylim_{N\to\infty}\left|R(\bar\varphi_N(x))-R(\bar\varphi_\theta(\frac{x}{N}))\right|=0$; 
we leave the details to the reader.

\section{Proof of Theorem \ref{thm:main}} \label{sec:proof} 
We present in this section the proof of our main theorem. Before we do that, we introduce a sequence of martingales related to the process studied in this work, which will constitute the building blocks of the whole proof.

\subsection{Dynkin's Martingales}
{Fix $G\in C^2([0,1])$, namely a twice continuously differentiable function. Note that, for simplicity, here $G$ is only space dependent.} By Dynkin's formula, the process $\{M_t^N(G), t\ge0\}$ defined via
\begin{equation}\label{eq:MGt}
    M_t^N(G):=\left\langle \pi^N_t, G\right\rangle-\left\langle \pi^N_0, G\right\rangle-\int_0^tN^2\mathcal{L}^N\left\langle \pi^N_s, G\right\rangle \de s
\end{equation}
is a martingale. The dynamics play a role in the compensator term, namely the integral term above. Recalling the definition of the generator \eqref{eq:generator}, we write 
\begin{align}
    N^2\mathcal{L}^N&\left\langle \pi^N_s,G\right\rangle  = \frac{1}{N}\sum_{x=2}^{N-2}g(\eta_s(x))\Delta_N G\left(\frac{x}{N}\right)\nonumber
    \\&\phantom{=} +g(\eta_s(1))\nabla^+_N G\left(\frac{1}{N}\right)- g(\eta_s(N-1))\nabla^-_NG\left(\frac{N-1}{N}\right)\label{eq:N^2L^N}
    \\&\phantom{=} +\kappa\left(\frac{\alpha-\lambda g(\eta_s(1))}{N^{\theta-1}}\right) G\left(\frac{1}{N}\right)+\kappa\left(\frac{\beta-\delta g(\eta_{s}(N-1))}{N^{\theta-1}}\right)G\left(\frac{N-1}{N}\right),\nonumber
\end{align}
where
\begin{equation}\label{eq:discrete_operators}
    \begin{split}
    &\Delta_N G\left( \frac{x}{N} \right) := N^2 \left[ G\left( \frac{x+1}{N} \right) + G\left( \frac{x-1}{N} \right) - 2G\left( \frac{x}{N} \right) \right],
    \\&\nabla_N^+ G\left( \frac{x}{N} \right) := N \left[ G\left( \frac{x+1}{N} \right) - G\left( \frac{x}{N} \right) \right],
    \\&\nabla_N^- G\left( \frac{x}{N} \right) := N \left[ G\left( \frac{x}{N} \right) - G\left( \frac{x-1}{N} \right) \right].
    \end{split}
\end{equation}
Moreover, the   quadratic variation of the martingale $\{M_t^N(G), t\ge0\}$ is given by
\begin{equation}\label{eq:quadratic_variation}
    \left\langle M^N(G)\right\rangle_t=\int_0^t\left\{N^2\mathcal{L}^N\left\langle \pi^N_s, G\right\rangle^2-2N^2\left\langle \pi^N_s, G\right\rangle \mathcal{L}^N \left\langle \pi^N_s, G\right\rangle\right\}\de s.
\end{equation}
After standard calculations, we can see that $\left \langle M^N(G)\right\rangle_t=\int_0^tB^N(G)_s\de s$, where
\begin{equation}\label{eq:BN}
\begin{split}
    B^N(G)_s& :=  \sum_{\substack{x, y\in I_N \\ |x-y|=1}}g(\eta_s(x))\left[G\left(\frac{y}{N}\right)-G\left(\frac{x}{N}\right)\right]^2 
    \\& +\kappa\left(\frac{\alpha +\lambda g(\eta_s(1))}{N^\theta}\right)G^2\left(\frac{1}{N}\right)+\kappa\left(\frac{\beta + \delta g(\eta_s(N-1))}{N^\theta}\right) G^2\left(\frac{N-1}{N}\right).
    \end{split}
\end{equation}
We are now ready for show tightness and then characterizing the limit points. We start with the former.

\subsection{Tightness}\label{sec:proof_tightness}
\begin{proposition}\label{prop:tightness} Let $\theta\ge 1$ and suppose that the rate function $g$ satisfies \eqref{g_non_decreasing}. Assume that the sequence $\{\mu^N\}_N$ is associated to a measurable and integrable initial profile $\gamma:[0,1]\to[0, \infty)$ in the sense of \eqref{eq:initial_profile} and satisfies \eqref{eq:stoch_domination}. Then, the sequence of measures $\{Q^N\}_N$ is tight. 
\end{proposition}

\begin{remark} We highlight that tightness of the sequence $\{Q^N\}_{N\ge 0}$ is also true if drop the hypotheses \eqref{g_non_decreasing} and \eqref{eq:stoch_domination} and require instead that the function $g$ is bounded; see Remark \ref{remark:g_bounded_instead} for more details.
\end{remark}

By \cite[Proposition 4.1.7]{kl99}, in order to prove Proposition \ref{prop:tightness}, it suffices to show tightness of the sequence of measures corresponding to the sequence of real-valued processes $\{\langle \pi^N_t, G\rangle\}_N$ for every $G$ in a dense subset of $C([0,1])$ with respect to the uniform topology. Here, we consider $G\in C^2([0,1])$. By Aldous' criterion \cite{ald78}, it is therefore enough to show that the following conditions are satisfied:  

\begin{enumerate}
\item[] \textsc{Condition 1}. For every $t\in[0,T]$,
\begin{equation*}
    \mylim_{A\to\infty}\mylimsup_{N\to\infty}\prob_{\mu^N}\left\{\frac{1}{N}\sum_{x\in I_N}\eta_t(x)\ge A\right\}=0.
\end{equation*}
\item[] \textsc{Condition 2}. For every $\delta>0$,
\begin{equation*}
    \mylim_{\gamma\to0}\mylimsup_{N\to\infty}\mysup_{\substack{\tau \in \mathfrak{T}_T \\ \omega \le \gamma}}\prob_{\mu^N}\left\{\left|\frac{1}{N}\sum_{x\in I_N}G\left(\frac{x}{N}\right)\left(\eta_{\tau+\omega}(x)-\eta_{\tau}(x)\right)\right|>\delta\right\}=0,
\end{equation*}
where $\mathfrak T_T$ is the family of all stopping times bounded by $T$.
\end{enumerate}

\textbf{Proof of \textsc{Condition 1}.} Define the process $Y_t:=Y^l_t + Y^r_t$, where $Y^l_t$ and $Y^r_t$ denote the number of particles created at the left boundary and the right boundary, respectively, up to time $tN^2$. We have the following natural bound:
\begin{equation}\label{eq:bound_particles}
    \sum_{x=1}^{N-1}\eta_t(x)\le \sum_{x=1}^{N-1}\eta_0(x)+Y_t,
\end{equation}
and thus
\begin{align*}
    \prob_{\mu^N}\left\{\frac{1}{N}\sum_{x=1}^{N-1}\eta_t(x)\ge A\right\}&\le \prob_{\mu^N}\left\{\frac{1}{N}\sum_{x=1}^{N-1}\eta_0(x)\geq \frac{A}{2}\right\}+\prob_{\mu^N}\left\{\frac{Y_t}{N}\geq\frac{A}{2}\right\}
    \\&=: A_N+B_N.
\end{align*}
Note that $\mylim_{A\to\infty}\mylimsup_{N\to\infty}A_N=0$, as $\mu^N$ is associated to an integrable profile $\gamma$. 
On the other hand, since the process is accelerated by $N^2$, under $\prob_{\mu^N}$ we have that $Y^l_t$ and $Y^r_t$ are Poisson processes with intensity $N^{2-\theta}\alpha$ and $N^{2-\theta}\beta$, respectively, and hence, since $\theta\ge1$,
\begin{align*}
    B_N\le\frac{2}{AN}\expected_{\mu^N}[Y_t]  =\frac{2}{AN}\expected_{\mu^N}\left[Y^l_t\right] + \frac{2}{AN}\expected_{\mu^N}\left[Y^r_t\right]
    &=\frac{2}{AN}\left(N^{2-\theta}\alpha t + N^{2-\theta}\beta t\right) 
    \\&\le\frac{2(\alpha+\beta)t}{A},
\end{align*}
which goes to zero as $A\to\infty$. 
We note that in the last inequality in the previous display we used the fact that $\theta\geq 1$.

\vspace{1em}
\textbf{Proof of \textsc{Condition 2}.} By \eqref{eq:MGt}, it is enough to show the following.
\begin{enumerate}
\item[]\textsc{Condition 2.1}. For every $\delta>0$,
\begin{equation*}          
    \mylim_{\gamma\to0}\mylimsup_{N\to\infty}\mysup_{\tau \in \mathfrak T_T\atop \omega \le \gamma}\prob_{\mu^N}\left\{\left|\int_{\tau}^{\tau+\omega}N^2\mathcal{L}^N\langle\pi^N_s, G\rangle \de s\right|>\delta\right\}=0.
\end{equation*}
\item[] \textsc{Condition 2.2}. For every $\delta>0$,
\begin{equation*}
    \mylim_{\gamma\to0}\mylimsup_{N\to\infty}\mysup_{\tau \in \mathfrak T_T\atop \omega \le \gamma}\prob_{\mu^N}\left\{\left|M^N(G)_{\tau+\omega}-M^N(G)_{\tau}\right|>\delta\right\}=0.
\end{equation*}
\end{enumerate}
By \eqref{eq:N^2L^N}, in order to show \textsc{Condition 2.1}, it suffices to show that, for all $\delta>0$,
\begin{align} 
    &\mylim_{\gamma\to0}\mylimsup_{N\to\infty}\mysup_{\tau \in \mathfrak T_T\atop \omega \le \gamma}\prob_{\mu^N}\left\{\left|\int_{\tau}^{\tau+\omega}\frac{1}{N}\sum_{x=2}^{N-2}g(\eta_s(x))\Delta_N G\left(\frac{x}{N}\right)\de s\right|>\delta\right\} = 0,\label{eq:c2.1.1}
    \\&\mylim_{\gamma\to0}\mylimsup_{N\to\infty}\mysup_{\tau \in \mathfrak T_T\atop \omega \le \gamma}\prob_{\mu^N}\left\{\left|\int_{\tau}^{\tau+\omega}g(\eta_s(1))\nabla^+_NG\left(\frac{1}{N}\right)\de s\right|>\delta\right\} = 0,\label{eq:c2.1.2}
    \\&\mylim_{\gamma\to0}\mylimsup_{N\to\infty}\mysup_{\tau \in \mathfrak T_T\atop \omega \le \gamma}\prob_{\mu^N}\left\{\left|\int_{\tau}^{\tau+\omega}g(\eta_{s}(N-1))\nabla_N^-G\left(\frac{N-1}{N}\right)\de s\right|>\delta\right\} = 0,\label{eq:c2.1.3}
    \\&\mylim_{\gamma\to0}\mylimsup_{N\to\infty}\mysup_{\tau \in \mathfrak T_T\atop \omega \le \gamma}\prob_{\mu^N}\left\{\left|\int_{\tau}^{\tau+\omega}\left(\frac{\alpha-\lambda g(\eta_s(1))}{N^{\theta-1}}\right) G\left(\frac{1}{N}\right)\de s\right|>\delta\right\} = 0,\label{eq:c2.1.4}
    \\&\mylim_{\gamma\to0}\mylimsup_{N\to\infty}\mysup_{\tau \in \mathfrak T_T\atop \omega \leq \gamma}\prob_{\mu^N}\left\{\left|\int_{\tau}^{\tau+\omega}\left(\frac{\beta - \delta g(\eta_{s}(N-1))}{N^{\theta-1}}\right)G\left(\frac{N-1}{N}\right)\de s\right|>\delta\right\} = 0.\label{eq:c2.1.5}
\end{align}

We start by showing \eqref{eq:c2.1.1}. Since $G$ is of class $C^2$ and $g$ increases at most linearly (recall hypothesis \eqref{eq:function_g}, {which implies $g(x)\le g^*x$}), the integral in \eqref{eq:c2.1.1} is bounded by
\begin{equation*}
    C(g^*, G)\int_{\tau}^{\tau+\omega}\frac{1}{N}\sum_{x=2}^{N-2}\eta_s(x)\de s.
\end{equation*}
By \eqref{eq:bound_particles}, this is bounded from above by
\begin{equation*}
    C(g^*, G)\left[\frac{\omega}{N}\sum_{x=1}^{N-1}\eta_0(x)+\int_{\tau}^{\tau+\omega}\frac{Y_s}{N} \de s\right].
\end{equation*}
Then, observing that $Y_s$ is non decreasing,  it is enough to show that, for any $\delta>0$
\begin{equation}\label{eq:c2.1.1.1}
    \mylim_{\omega\to0}\mylimsup_{N\to\infty}\prob_{\mu^N}\left\{\frac{\omega}{N}\sum_{x=1}^{N-1}\eta_0(x)>\delta\right\}=0
\end{equation}
and
\begin{equation}\label{eq:c2.1.1.2}
    \mylim_{\omega\to0}\mylimsup_{N\to\infty}\prob_{\mu^N}\left\{\frac{\omega}{N}Y_{T+\omega}>\delta\right\}=0.
\end{equation}
As in the proof of \textsc{Condition 1}, \eqref{eq:c2.1.1.1} holds because $\mu^N$ is associated to an integrable profile $\gamma$,  and \eqref{eq:c2.1.1.2} follows from noting that, since $\theta\ge1$,
\begin{equation*}
    \prob_{\mu^N}\left\{\frac{\omega}{N}Y_{T+\omega}>\delta\right\}\le\frac{\omega}{\delta N}\expected_{\mu^N}\left[Y_{T+\omega}\right]=\frac{\omega (\alpha +\beta) (T+\omega)}{\delta N^{\theta-1}}\leq\frac{\omega (\alpha +\beta) (T+\omega)}{\delta},
\end{equation*}
which goes to zero as $\omega\to 0$.

As for the proof of \eqref{eq:c2.1.2}, \eqref{eq:c2.1.3}, \eqref{eq:c2.1.4} and \eqref{eq:c2.1.5}, we will use the following lemma.

\begin{lemma}\label{lemma:g_moments} Under assumptions \eqref{g_non_decreasing} and \eqref{eq:stoch_domination}, for every $s\ge0$ and $x\in I_N$, it holds
\begin{align}
    &\expected_{\mu^N}[g(\eta_s(x))]\le \bar\varphi_N(x), \label{eq:g_moment_1}
    \\&\expected_{\mu^N}[g(\eta_s(x))^2]\le g^*\bar\varphi_N(x)+\bar\varphi_N(x)^2. \label{eq:g_moment_2}
\end{align}
Consequently, for $\ell=1, 2$, 
\begin{equation*}
    \expected_{\mu^N}[g(\eta_s(x))^\ell]\le C, 
\end{equation*}
where $C$ is a positive constant that only depends on the parameters $\alpha, \lambda, \beta, \delta$.
\end{lemma}

\begin{remark}\label{remark:g_bounded_instead}
In the proof of Proposition \ref{prop:tightness},  hypotheses \eqref{g_non_decreasing} and \eqref{eq:stoch_domination} are only used in Lemma \ref{lemma:g_moments} above. Since this result is trivial when $g$ is bounded, in this case such hypotheses are not needed for tightness to hold. 
\end{remark}

\begin{proof}[Proof of Lemma \ref{lemma:g_moments}] For every $x\in I_N$, by \eqref{g_non_decreasing} the function $h_x:\Omega_N\to\mathbb R$, defined by $h_x(\eta):=[g(\eta(x))]^\ell$, is monotone. Hence, by \eqref{g_non_decreasing}, \eqref{eq:stoch_domination} and Lemma \ref{lemma:attractiveness}, we see that
\begin{equation*}
    \expected_{\mu^N}[g(\eta_s(x))^\ell]\le \expected_{\bar\nu_N}[g(\eta_s(x))^\ell]=\expected_{\bar\nu_N}[g(\eta_0(x))^\ell]=E_{\bar\nu_N}\left[g(\eta(x))^\ell\right].
\end{equation*}
To conclude the proof of \eqref{eq:g_moment_1}, we recall that $E_{\bar\nu_N}\left[g(\eta(x))\right]=\bar\varphi_N(x)$.
As for the proof of \eqref{eq:g_moment_2}, we write
\begin{align*}
    E_{\bar\nu_N}\left[g(\eta(x))^2\right]  &= \frac{1}{Z(\bar\varphi_N(x))}\sum_{k=0}^\infty g(k)^2\frac{\bar\varphi_N(x)^k}{g(k)!}
    = \frac{\bar\varphi_N(x)}{Z(\bar\varphi_N(x))}\sum_{k=1}^\infty g(k)\frac{\bar\varphi_N(x)^{k-1}}{g(k-1)!}.
\end{align*}
By \eqref{eq:function_g}, we have that $g(k)\le g^* +g(k-1)$, and thus
\begin{align*}
    E_{\bar\nu_N}\left[g(\eta(x))^2\right]  & \le   g^*\bar\varphi_N(x) + \frac{\bar\varphi_N(x)}{Z(\bar\varphi_N(x))}\sum_{k=1}^\infty g(k-1)\frac{\bar\varphi_N(x)^{k-1}}{g(k-1)!}
    \\& = g^*\bar\varphi_N(x) + \frac{\bar\varphi_N(x)^2}{Z(\bar\varphi_N(x))}\sum_{k=2}^\infty \frac{\bar\varphi_N(x)^{k-2}}{g(k-2)!}
    \\& =  g^*\bar\varphi_N(x)+\bar\varphi_N(x)^2.
\end{align*}
Since $\bar\varphi_N(x)$ is uniformly bounded, the proof is concluded.
\end{proof}

We are now ready to show \eqref{eq:c2.1.2}, \eqref{eq:c2.1.3}, \eqref{eq:c2.1.4} and \eqref{eq:c2.1.5}. Since $G$ is of class $C^2 $, the integrals in \eqref{eq:c2.1.2} and \eqref{eq:c2.1.3} are bounded from above by
\begin{equation*}
    C(g^*,G)\int_{\tau}^{\tau+\omega}g(\eta_s(x))\de s,
\end{equation*}
where $x=1$ or $x=N-1$. For all $x\in I_N$, we have that
\begin{equation*}
    \prob_{\mu^N}\left\{\int_\tau^{\tau+\omega}g(\eta_s(x))\de s>\delta\right\}\le \frac{1}{\delta}\expected_{\mu^N}\left[\int_\tau^{\tau+\omega}g(\eta_s(x))\de s\right].
\end{equation*}
By the Cauchy-Schwarz inequality,
\begin{align*}
    \expected_{\mu^N}\left[\int_{\tau}^{\tau+\omega}g(\eta_s(x))\de s\right] & =  \expected_{\mu^N}\left[\int_0^T \boldsymbol{1}_{[\tau,\tau+\omega]}(s)g(\eta_s(x))\de s\right]
    \\& \le \sqrt{\omega} \left\{\expected_{\mu^N}\left[\int_0^T g(\eta_s(x))^2\de s\right]\right\}^{1/2}
    \\& = \sqrt{\omega} \left\{\int_0^T \expected_{\mu^N}\left[g(\eta_s(x))^2\right]\de s\right\}^{1/2}.
\end{align*}
Then, using Lemma \ref{lemma:g_moments}, we obtain 
\begin{equation}\label{eq:cauchy_schwarz}
    \expected_{\mu^N}\left[\int_{\tau}^{\tau+\omega}g(\eta_s(x))\de s\right]\le \{\omega TC\}^{1/2}.
\end{equation}
Sending $\omega\to0$, we conclude that \eqref{eq:c2.1.2} and \eqref{eq:c2.1.3} hold. Since $\theta\ge1$, \eqref{eq:c2.1.4} and \eqref{eq:c2.1.5} follow from the same argument.

Finally, we show \textsc{Condition 2.2}. Using Chebychev's inequality and the explicit formula for the quadratic variation given in \eqref{eq:quadratic_variation}, we have that
\begin{align}
    \prob_{\mu^N}\left\{\left|M^N(G)_{\tau+\omega}-M^N(G)_{\tau}\right|>\delta\right\} & \le  \frac{1}{\delta^2}\expected_{\mu^N}\left[(M^N(G)_{\tau+\omega}-M^N(G)_{\tau})^2\right]\nonumber
    \\& = \frac{1}{\delta^2} \expected_{\mu^N}\left[\int_\tau^{\tau+\omega}B^N(G)_s\de s\right], \label{eq:BN_bound}
\end{align}
where $B^N(G)_s$ was defined in \eqref{eq:BN}. Using the fact that both $G$ and its derivative are bounded, and well as \eqref{eq:cauchy_schwarz}, we can see that  \eqref{eq:BN_bound} is bounded from above by $\frac{C\omega}{N}$, where $C$ is a constant that does not depend on $N$ and $\omega$. This completes the proof.

\subsection{Limit Points are Concentrated on Absolutely Continuous Measures}\label{sec:abs_cont}
We now show that any limit point of the sequence of measures $\{Q^N\}_N$ is concentrated on paths which are absolutely continuous with respect to the Lebesgue measure. The lemma that follows is very similar to \cite[Lemma 7.4]{fmn21}; however, since our model is a generalisation of the one considered therein, for the sake of completeness we include a full proof.

\begin{lemma}\label{lemma:concentration_lebesgue} Let $\{\mu^N\}_N$ be a sequence of probability measures  satisfying \eqref{eq:stoch_domination}. Denote by $R_{\mu^N}$ the probability measure induced by $\pi^N$ on the space $\mathcal{M}_+$  of finite, positive measures on $[0,1]$ endowed with the weak topology. Then, any limit point $R^*$ of the sequence $\{R_{\mu^N}\}_N$ is concentrated on measures which are absolutely continuous with respect to the Lebesgue measure, namely
\begin{equation*}
    R^*\left\{\pi: \pi(\de u)=\rho(u)\de u\right\}=1.
\end{equation*}
\end{lemma}

Before proving this result, we highlight that it indeed implies concentration of limit points of $\{Q^N\}_N$ on paths which are absolutely continuous with respect to the Lebesgue measure. Indeed, under \textsc{Assumption (ND)} and \eqref{eq:stoch_domination}, by Lemma \ref{lemma:attractiveness} we have that $\mu^NS^N(t)\le \bar\nu_NS^N(t)=\bar\nu_N$ for each $t\in [0, T]$, where $S^N$ denotes the semigroup associated to $N^2\mathcal{L}^N$. Hence, by Lemma \ref{lemma:concentration_lebesgue} applied to $\mu^NS^N(t)$, we get that any limit point $Q^*$ of $\{Q^N\}_N$ satisfies
\begin{equation*}
    Q^*\left\{\pi: \pi_t(\de u)=\rho_t(u)\de u\right\}=1
\end{equation*}
for all $t\in [0, T]$. A straightforward application of Fubini's lemma
\begin{equation*}
    Q^*\left\{\pi: \pi_t(\de u)=\rho_t(u)\de u\ \ \text{for all} \ t\in [0, T]\right\}=1,
\end{equation*}
as required.

\begin{proof}[Proof of Lemma \ref{lemma:concentration_lebesgue}] It suffices  to show that there exists $\eps>0$ such that, for any continuous function $G:[0, 1]\to[0, \infty)$,
\begin{equation*}
    R^*\left\{\pi: \langle \pi, G\rangle \le \int_0^1 G(u)[\bar\rho_\theta(u)+\eps]\de u\right\}=1,
\end{equation*}
where $\bar\rho_\theta$ is the limit density profile given in Remark \ref{remark:limit_fugacity_density}. Let $\{R_{\mu^{N_k}}\}_k$ denote a convergent subsequence of $\{R_{\mu^N}\}_N$: then
\begin{equation}\label{eq:measure_inequality}
    \begin{split}
    &R^*\left\{\pi: \langle \pi, G\rangle \le \int_0^1 G(u)[\bar\rho_\theta(u)+\eps]\de u\right\}
    \\&\ge \mylimsup_{k\to\infty} R_{\mu^{N_k}}\left\{\pi: \langle \pi, G\rangle \le \int_0^1 G(u)[\bar\rho_\theta(u)+\eps]\de u\right\}
    \\&= \mylimsup_{k\to\infty} \mu^{N_k}\left\{\eta: \langle \pi^N(\eta), G\rangle \le \int_0^1 G(u)[\bar\rho_\theta(u)+\eps]\de u\right\}.
    \end{split}
\end{equation}
From \eqref{eq:stoch_domination} and \cite[Theorem II.2.4]{lig05}, there exists a coupling $\tilde{\mu}^N$ of the two measures $\mu^N$ and $\bar\nu_N$ on $\Omega_N\times \Omega_N$ such that $\tilde\mu^N\left\{(\eta, \xi): \eta\le\xi\right\}=1$, and hence, since $G$ is non negative,
\begin{equation*}
    \tilde\mu^N\left\{(\eta, \xi): \langle\pi^N(\eta), G\rangle \le \langle \pi^N(\xi), G\rangle \right\}=1.
\end{equation*}
By \eqref{eq:measure_inequality}, this implies that
\begin{align*}
    &R^*\left\{\pi: \langle \pi, G\rangle \le \int_0^1 G(u)[\bar\rho_\theta(u)+\eps]\de u\right\}
    \\&\ge \mylimsup_{k\to\infty} \bar\nu_{N_k}\left\{\eta: \langle \pi^N(\eta), G\rangle \le \int_0^1 G(u)[\bar\rho_\theta(u)+\eps]\de u\right\}=1,   
\end{align*}
where we used the fact that the sequence $\{\bar\nu_N\}_N$ is associated to the profile $\bar\rho_\theta$ in the sense of Definition \ref{def:associated_profile}.\end{proof}

\subsection{Characterisation of Limit Points}\label{sec:characterisation}
In this section we finally characterise the limit point of the sequence $\{Q^N\}_N$, whose existence is guaranteed by Proposition \ref{prop:tightness}.
We first note that, from the results of Section \ref{sec:abs_cont}, we know that any limit point $Q^*$ is supported on measures $\pi_t(\de u)$ which are absolutely continuous with respect to the Lebesgue measure, that is, $\pi_t(\de u)=\rho_t(u)\de u$. We start by identifying the density as solving item iii) of Definition \ref{def:weak_solutions}.

\subsubsection{The Density Solves Item iii) of Definition \ref{def:weak_solutions}}
From \cite[Lemma A1.5.10]{kl99}, we note that for any $G\in C^{1,2}([0,T]\times [0,1])$, the process $\{M^N_t(G), t\in[0, T]\}$ defined via
\begin{equation}
    \begin{split}
    M^N_t&(G):=\langle \pi^N_t, G_t \rangle-\langle \pi^N_0, G_0\rangle - \int_0^t\langle \pi_s^N, \partial_s G_s\rangle \de s 
    \\&- \int_0^t \frac{1}{N} \sum_{x=2}^{N-2}g(\eta_s(x))\Delta_N G_s\left(\frac{x}{N}\right) \de s
    \\&-\int_0^t \left\{g(\eta_s(1))\nabla^+_N G\left(\frac{1}{N}\right) -  g(\eta_s(N-1))\nabla^-_N G_s\left(\frac{N-1}{N}\right)\right\}\de s
    \\&-\kappa\int_0^t \left\{ \left(\frac{\alpha -\lambda g(\eta_s(1))}{N^{\theta-1}}\right) G_s\left(\frac{1}{N}\right) + \left(\frac{ \beta-\delta g(\eta_s(N-1))}{N^{\theta - 1}}\right) G_s\left(\frac{N-1}{N}\right)\right\}\de s
    \end{split}
\end{equation}
is a martingale with respect to the natural filtration of the process, where $\Delta_N, \nabla_N^+$ and $\nabla_N^-$ are defined in (\ref{eq:discrete_operators}). Since $G \in C^2([0,1])$, we can rewrite $M^N_t(G)$ as
\begin{equation*}
    \begin{split}
    M^N_t(G)=&\langle \pi^N_t, G_t \rangle-\langle \pi^N_0, G_0\rangle - \int_0^t\langle \pi_s^N, \partial_s G_s\rangle \de s 
    \\&-\int_0^t \frac{1}{N} \sum_{x=2}^{N-2}g(\eta_s(x))\Delta G_s\left(\frac{x}{N}\right) \de s
    \\&-\int_0^t\left\{g(\eta_s(1))\partial_u G_s(0) -  g(\eta_s(N-1))\partial_u G_s(1)\right\}\de s
    \\&-\kappa\int_0^t \left\{ \left(\frac{\alpha-\lambda g(\eta_s(1))}{N^{\theta-1}}\right) G_s(0) + \left(\frac{\beta-\delta g(\eta_s(N-1))}{N^{\theta-1}}\right) G_s(1)\right\}\de s 
    \\&+\mathscr{R}^{1,\theta}_N(G,t),
    \end{split}
\end{equation*}
where $\expected_{\mu^N}[|\mathscr{R}^{1,\theta}_N(G,t)|]$ goes to zero as $N \to \infty$ uniformly in $t \in [0,T]$. This is a consequence of the use of Taylor's expansion, which allows to replace the discrete operators by the respective continuous ones. 

\vspace{1em}
\ \textbf{The Bulk Terms.} For $\eps>0$, let
\begin{equation}\label{eq:bulk_eps}
    I_N^\eps:= \{1+\eps N, \ldots, N-1-\eps N \},
\end{equation}
where $\eps N$ must be understood as $\lfloor \eps N\rfloor$. Note that the bulk terms can be rewritten as
\begin{equation}\label{eq:bulk_1}
    \int_0^t \frac{1}{N}\sum_{x\in I^\eps_N}g(\eta_s(x))\Delta G_s\left(\frac{x}{N}\right) \de s + \mathscr{R}^2_{N, \eps}(G, t),
\end{equation}
where
\begin{equation*}
     \mathscr{R}^2_{N, \eps}(G, t) := \int_0^t\frac{1}{N} \Big\{\sum_{x =2}^{\eps N}g(\eta_s(x))\Delta G_s\left(\frac{x}{N}\right) + \sum_{N-\eps N}^{N-2}g(\eta_s(x))\Delta G_s\left(\frac{x}{N}\right)\Big\}\de s.
\end{equation*}
By Lemma \ref{lemma:g_moments},
$
    \expected_{\mu^N}\left[\left|\mathscr{R}^2_{N, \eps}(G,t)\right|\right]
    \lesssim  \eps  T \|\Delta G\|_{\infty},
$ and hence $\expected_{\mu^N}[|\mathscr{R}^2_{N, \eps}(G,t)|]\to 0$ as $N\to\infty $ and $\eps\to0$. 
We now sum and subtract the expression 
\begin{equation}
    \frac{1}{2\eps N +1}\sum_{|y-x|\le \eps N}g(\eta_s(y))
\end{equation}
to the leftmost term of \eqref{eq:bulk_1}, which we rewrite as 
\begin{equation}\label{eq:bulk_2}
    \int_0^t \frac{1}{N} \sum_{x\in I^\eps_N}\Delta G_s\left(\frac{x}{N}\right)\frac{1}{2\eps N+1}\sum_{|y-x|\le\eps N}g(\eta_s(y)) \de s + \mathscr{R}^3_{N, \eps}(G,t),
\end{equation}
where 
\begin{equation*}
    \mathscr{R}^3_{N, \eps}(G,t):= \int_0^t \frac{1}{N} \sum_{x\in I^\eps_N}\Big\{g(\eta_s(x))-\frac{1}{2\eps N+1}\sum_{|y-x|\le\eps N}g(\eta_s(y))\Big\}\Delta G_s\left(\frac{x}{N}\right) \de s.
\end{equation*}
A simple computation based on a summation by parts shows that 
$\expected_{\mu^N}[|\mathscr{R}^3_{N, \eps}(G,t)|]$ can be bounded from above by
\begin{equation*}
\frac{1}{N(2\eps N+1)}\sum_{x\in I^\eps_N}\int_0^t \sum_{|y-x| \le \eps N}\expected_{\mu^N}[g(\eta_s(y))]\left|\Delta G_s\left(\frac{x}{N}\right)-\Delta G_s\left(\frac{y}{N}\right)\right|\de s,\end{equation*}
which vanishes as $N\to\infty$ and $\eps\to0$ since $\expected_{\mu^N}[g(\eta_s(x))]\le C$ and because of the continuity of $\Delta G_s$. Hence, we are only left to handle the integral term in \eqref{eq:bulk_2}. To that end, we introduce the {centred} empirical average $\bar\eta^{\varepsilon N}_s(x)$ in the box of size $2\eps N+1$ around $x$, namely
\begin{equation}\label{eq:bar_eta}
    \bar\eta^{\eps N}_s(x):=\frac{1}{2\eps N+1}\sum_{|y-x|\le \eps N}\eta_s(y)
\end{equation}
for each $x \in I^\eps_N$. Recalling the definition of the function $\Phi:[0, \infty)\to [0,\varphi^*)$ in \eqref{eq:Phi}, we write \eqref{eq:bulk_2} as
\begin{equation}\label{eq:bulk_3}
    \int_0^t \frac{1}{N} \sum_{x\in I^\eps_N}\Delta G_s\left(\frac{x}{N}\right)\Phi(\bar\eta^{\eps N}_s(x)) \de s + \mathscr{R}^4_{N, \varepsilon}(G,t),
\end{equation}
where 
\begin{equation*}
    \mathscr{R}^4_{N, \varepsilon}(G,t):= \int_0^t \frac{1}{N}\sum_{x\in I^\eps_N}\Delta G_s\Big(\frac{x}{N}\Big)\Big\{\frac{1}{2\eps N+1}\sum_{|y-x|\le\eps N}g(\eta_s(y))-\Phi(\bar\eta^{\eps N}_s(x))\Big\} \de s.
\end{equation*}
By the bulk replacement lemma (Lemma \ref{lemma:repl_bulk}, stated and proved in the section that follows), $\expected_{\mu^N}[|\mathscr{R}^4_{N, \eps}(G,t)|]$ vanishes as $N\to\infty$ and $\eps\to 0$, and we are therefore able to rewrite the martingale $M^N_t(G)$ as
\begin{equation}\label{eq:martingale_boundaries}
    \begin{split}
  \langle \pi^N_t, G_t \rangle&-\langle \pi^N_0, G_0\rangle - \int_0^t\langle \pi_s^N, \partial_s G_s\rangle \de s 
    -\int_0^t \frac{1}{N} \sum_{x\in I_N^\eps}\Phi(\bar\eta^{\eps N}_s(x))\Delta G_s\left(\frac{x}{N}\right) \de s
    \\&-\int_0^t \left\{g(\eta_s(1))\partial_u G_s(0) -  g(\eta_s(N-1))\partial_u G_s(1)\right\}\de s
    \\&-\kappa\int_0^t \left\{\left(\frac{\alpha -\lambda g(\eta_s(1))}{N^{\theta-1}}\right) G_s(0) + \left(\frac{ \beta-\delta g(\eta_s(N-1))}{N^{\theta - 1}}\right) G_s(1)\right\}\de s
    \end{split}
\end{equation}
up to a term which vanishes as $N\to\infty$ and $\eps\to0$.

\vspace{1em}
\textbf{The Boundary Terms.} We are now left to treat the boundary terms. First, we sum and subtract the expressions
\begin{equation*}
    \frac{1}{\eps N}\sum_{y=1}^{\eps N}g(\eta_s(y))
\quad \textrm{and}\quad 
    \frac{1}{\eps N}\sum_{y=N-\eps N}^{N-1}g(\eta_s(y)),
\end{equation*}
with appropriate coefficients, to the boundary terms of \eqref{eq:martingale_boundaries}, which we rewrite as
\begin{equation}
   \begin{split}\label{eq:martingale_boundaries_2}
    &-\int_0^t \Big\{\frac{1}{\eps N}\sum_{y=1}^{\eps N}g(\eta_s(y))\partial_u G_s(0) -  \frac{1}{\eps N }\sum_{y=N-\eps N}^{N-1}g(\eta_s(y))\partial_u G_s(1)\Big\}\de s
    \\&-\kappa\int_0^t \Big(\frac{\alpha -\lambda \frac{1}{\eps N}\sum_{y=1}^{\eps N}g(\eta_s(y))}{N^{\theta-1}}\Big) G_s(0)\de s
    \\&-\kappa \int_0^t \Big(\frac{ \beta-\delta \frac{1}{\eps N }\sum_{y=N-\eps N}^{N-1}g(\eta_s(y))}{N^{\theta - 1}}\Big) G_s(1)+\mathscr{R}_{N, \eps}^{5, \theta}(G, t),
   \end{split}
\end{equation}
with
\begin{align*}
    \mathscr{R}_{N, \eps}^{5, \theta}(G, t):=&-\int_0^t \Big\{g(\eta_s(1))-\frac{1}{\eps N}\sum_{y=1}^{\eps N}g(\eta_s(y))\Big\}\partial_u G_s(0)\de s
    \\&+  \int_0^t \Big\{g(\eta_s(N-1))-\frac{1}{\eps N }\sum_{y=N-\eps N}^{N-1}g(\eta_s(y))\Big\}\partial_u G_s(1)\de s
    \\&+\kappa\int_0^t \frac{\lambda\left\{g(\eta_s(1))- \frac{1}{\eps N}\sum_{y=1}^{\eps N}g(\eta_s(y))\right\}}{N^{\theta-1}} G_s(0)\de s
    \\&+\kappa \int_0^t \frac{\delta\left\{g(\eta_s(N-1)- \frac{1}{\eps N }\sum_{y=N-\eps N}^{N-1}g(\eta_s(y))\right\}}{N^{\theta - 1}} G_s(1)\de s.
\end{align*}
Since $\theta\ge1$, by the same argument given for $\mathscr{R}_{N, \eps}^3(G, t)$, we have that $\expected_{\mu^N}[|\mathscr{R}^{5, \theta}_{N, \eps}(G, t)|]\to0$ as $N\to\infty$ and $\eps\to0$. Let now $\vec\eta^{\eps N}_s(x)$ and $\cev\eta^{\eps N}_s(x)$ denote the {directed}  density averages in the box of size $\eps N$ to the right and left, respectively, of $x\in\{1, \ldots, N-1-\eps N\}$ and $x\in\{\eps N+1, \ldots, N-1\}$, respectively, namely
\begin{equation}\label{eq:vec_eta}
    \vec\eta^{\eps N}_s(x):=\frac{1}{\eps N}\sum_{y=x+1}^{x+\eps N}\eta_s(y), \ \ \ \ \cev\eta^{\eps N}_s(x):=\frac{1}{\eps N}\sum_{y=x-\eps N}^{x-1}\eta_s(y).
\end{equation}
Note that, for $\theta>1$, the two bottom lines of \eqref{eq:martingale_boundaries_2} vanish as $N\to\infty$: hence, up to a term vanishing in the limit, we can rewrite \eqref{eq:martingale_boundaries_2} as
\begin{equation*}
    \begin{split}
    &-\int_0^t \left\{\Phi\big(\vec{\eta}^{\eps N}_s(1)\big)\partial_u G_s(0) - \Phi\big(\cev{\eta}^{\eps N}_s(N-1)\big)\partial_u G_s(1)\right\}\de s
    \\&-\kappa\boldsymbol{1}_{\{1\}}(\theta)\int_0^t \left\{\left(\alpha -\lambda \Phi\big(\vec{\eta}^{\eps N}_s(1)\big)\right) G_s(0) + \left(\beta-\delta \Phi\big(\cev{\eta}^{\eps N}_s(N-1)\big)\right) G_s(1)\right\}\de s
    \\&+\mathscr{R}_{N, \eps}^{6, \theta}(G, t),
    \end{split}
\end{equation*}
where 
\begin{equation*}
    \begin{split}
    \mathscr{R}_{N, \eps}^{6, \theta}(G, t)&:=-\int_0^t \Big\{\frac{1}{\eps N}\sum_{y=1}^{\eps N}g(\eta_s(y))-\Phi\big(\vec{\eta}^{\eps N}_s(1)\big)\Big\}\partial_u G_s(0)\de s
    \\&\phantom{:=}+\int_0^t \Big\{\frac{1}{\eps N }\sum_{y=N-\eps N}^{N-1}g(\eta_s(y))-\Phi\big(\cev{\eta}^{\eps N}_s(N-1)\big)\Big\}\partial_u G_s(1)\de s
    \\&\phantom{:=}+\kappa\boldsymbol{1}_{\{1\}}(\theta)\int_0^t \lambda \Big\{\frac{1}{\eps N}\sum_{y=1}^{\eps N}g(\eta_s(y))-\Phi\big(\vec{\eta}^{\eps N}_s(1)\big)\Big\}G_s(0)\de s
    \\&\phantom{:=}+\kappa\boldsymbol{1}_{\{1\}}(\theta)\int_0^t\delta \Big\{\frac{1}{\eps N }\sum_{y=N-\eps N}^{N-1}g(\eta_s(y))-\Phi\big(\cev{\eta}^{\eps N}_s(N-1)\big)\Big\} G_s(1)\de s.
    \end{split}
\end{equation*}
By the boundary replacement lemma (namely Lemma \ref{lemma:repl_boundaries} in Section \ref{sec:replacement_boundary}), we have that $\expected_{\mu^N}[|\mathscr{R}^{6, \theta}_{N, \eps}(G,t)|]$ vanishes as $N\to\infty$ and $\eps\to 0$.

\vspace{1em}
\textbf{Identification of the Density Solving \eqref{eq:intergal_pde}.} We are finally left with the martingale equation
\begin{align*}
    M^N_t&(G)=\langle \pi^N_t, G_t \rangle-\langle \pi^N_0, G_0\rangle - \int_0^t\langle \pi_s^N, \partial_s G_s\rangle \de s 
    \\&-\int_0^t \frac{1}{N} \sum_{x\in I_N^\eps}\Phi(\bar\eta^{\eps N}_s(x))\Delta G_s\left(\frac{x}{N}\right) \de s
    \\&-\int_0^t \left\{\Phi\big(\vec{\eta}^{\eps N}_s(1)\big)\partial_u G_s(0) -  \Phi\big(\cev{\eta}^{\eps N}_s(N-1)\big)\partial_u G_s(1)\right\}\de s
    \\&-\kappa\boldsymbol{1}_{\{1\}}(\theta)\int_0^t \left\{\left(\alpha -\lambda \Phi\big(\vec{\eta}^{\eps N}_s(1)\big)\right) G_s(0) + \left(\beta-\delta \Phi\big(\cev{\eta}^{\eps N}_s(N-1)\big)\right) G_s(1)\right\}\de s
\end{align*}
up to a term which vanishes in $L^1(\prob_{\mu^N})$ as $N\to\infty$ and $\eps\to 0.$ Calling $Q^*$ a limit point of the sequence of measures $\{Q^N\}_N$ -- whose existence is guaranteed by Proposition \ref{prop:tightness} -- and using Portmanteau's Theorem, it is not hard to see that
\begin{align*}
    Q^*\bigg\{&\pi: \pi_t(u)=\rho_t(u)\de u \ \text{and} \ \ \langle \rho_t , G_t\rangle - \langle \gamma, G_0 \rangle 
    \\&- \int_0^t \{\langle \rho_s, \partial_s G_s \rangle + \langle \Phi(\rho_s), \Delta G_s\rangle\}\de s
    \\ &+\int_0^t \{ \Phi(\rho_s(1))\partial_u G_s(1)-\Phi(\rho_s(0))\partial_u G_s(0) \}\de s
    \\&-\tilde\kappa\int_0^t (\beta- \delta\Phi(\rho_s(1)))  G_s(1)-(\lambda \Phi(\rho_s(0))-\alpha)G_s(0) \de s =0
    \\&\text{for all } t\in [0, T], \ \ G\in C^{1, 2}([0, T]\times [0, 1]) \bigg\}=1,
\end{align*}
where $\tilde\kappa:=\kappa\boldsymbol{1}_{\{1\}}(\theta)$. It remains to show that $Q^*$ is concentrated on paths whose density $\rho$ with respect to the Lebesgue measure satisfies $\rho\in L^2([0, T]\times[0, 1])$ and $\Phi(\rho)\in L^2([0, T], \mathcal{H}^1)$. We start with the former.

\subsubsection{The Density Solves Item i) of Definition \ref{def:weak_solutions}}
In order to show that $\rho\in L^2([0, T]\times[0, 1])$, we will need the next result, which will also be used repeatedly in the two sections that follow. This lemma can be viewed as a corollary of the proof of \cite[Lemma 4.3]{fmn21}, but for the sake of completeness we present its proof.

\begin{lemma}\label{lemma:bounded_moments}
Let $0<\varphi<\varphi^*$ and let $\bar{\nu}_N$ denote the invariant measure from Lemma \ref{lemma:inv_measure}: for each positive integer $\ell$, we have that
\begin{equation*}
    \mysup_{N}\mysup_{x\in I_N}E_{\bar{\nu}_\varphi}[(\eta(x))^\ell]<\infty
\quad \textrm{and}\quad
    \mysup_{N}\mysup_{x\in I_N}E_{\bar{\nu}_N}[(\eta(x))^\ell]<\infty.
\end{equation*}
\end{lemma}

\begin{proof}  
Using the same ideas as in \eqref{eq:R}, for all $x\in I_N$ we have that $E_{\bar\nu_\varphi}[(\eta(x))^\ell]=R_\ell(\varphi)$ and $E_{\bar\nu_N}[(\eta(x))^\ell]=R_\ell(\bar\varphi_N(x))$, where $R_\ell:[0, \varphi*)\to[0, \infty)$ is given by
\begin{equation*}
    R_\ell(\varphi):=\frac{1}{Z(\varphi)}\sum_{k\ge0}k^\ell\frac{\varphi^k}{g(k)!},
\end{equation*}
with $\varphi^*$ being the radius of convergence of $Z$. Note that the function $R_\ell$ is analytic on $[0,\varphi^*)$: to see this, we write $R_\ell(\varphi)=\frac{A_\ell(\varphi)}{Z(\varphi)}$, where $A_\ell$ is defined inductively by 
\begin{equation*}
    \begin{cases}
    A_0(\varphi)=Z(\varphi),
    \\A_n(\varphi)=\varphi A'_{n-1}(\varphi).
    \end{cases}
\end{equation*}
Hence, we immediately get
\begin{equation*}
    \mysup_{N}\mysup_{x\in I_N}E_{\bar{\nu}_\varphi}[(\eta(x))^\ell]=R_\ell(\varphi)<\infty.
\end{equation*}
Let now $\varphi^{**}:=\mysup_{N}\mysup_{x\in I_N}\bar\varphi_N(x)$, where $\bar\varphi_N(x)$ was defined in \eqref{eq:fugacity}: by the condition \eqref{eq:params_condition}, we have that $\varphi^{**}\le \varphi^*$. Therefore, 
\begin{equation*}
    \mysup_{N}\mysup_{x\in I_N}E_{\bar{\nu}_N}[(\eta(x))^\ell]=\mysup_{N}\mysup_{x\in I_N}R_\ell(\bar\varphi_N(x))\le\mysup_{\varphi\in[0, \varphi^{*}]}R_\ell(\varphi)<\infty,
\end{equation*}
as claimed.
\end{proof}

\begin{lemma} \label{lemma:rho_L2} Under the hypotheses of Theorem \ref{thm:main}, any limit point $Q^*$ of $\{Q^N\}_N$ is concentrated on paths $\pi_t(\de u)=\rho_t(u)\de u$ such that
\begin{equation*}
    \int_0^T\int_0^1 \rho_t^2(u)\de u\de t<\infty
\end{equation*}
almost surely.
\end{lemma}
\begin{proof}
Note that the map 
\begin{equation*}
    \eta\mapsto \sum_{x=1}^{\eps N} \big(\vec\eta^{\eps N}(x)\big)^2+\sum_{x\in I_N^\eps}\big(\bar\eta^{\eps N}(x)\big)^2+\sum_{x=N-\eps N}^{N-1}\big(\cev\eta^{\eps N}(x)\big)^2=:H(\eps N, \eta)
\end{equation*}
is non decreasing, where $I_N^\eps$ is defined in \eqref{eq:bulk_eps}, and $\vec\eta^{\eps N}(x), \bar\eta^{\eps N}(x)$ and $\cev\eta^{\eps N}(x)$ are defined in \eqref{eq:bar_eta} and \eqref{eq:vec_eta}. Hence, by \textsc{Assumption (ND)} and attractiveness, we have that
\begin{align*}
    \expected_{\mu^N}\left[\int_0^T \frac{1}{N} H(\eps N, \eta_s) \de s \right]&\le \expected_{\bar\nu_N}\left[\int_0^T \frac{1}{N} H(\eps N, \eta_s) \de s \right]
    = T E_{\bar\nu_N}\left[\frac{1}{N} H(\eps N, \eta)\right] 
    \\&\lesssim T\mysup_{x\in I_N} E_{\bar\nu_N}\left[\eta(x)^2\right].
\end{align*}
By Lemma \ref{lemma:bounded_moments}, this implies that
\begin{equation*}
    \mylimsup_{\eps\to 0}\mylimsup_{N\to\infty}\expected_{\mu^N}\left[\int_0^T \frac{1}{N} H(\eps N, \eta_s) \de s \right]\le C 
\end{equation*}
for some finite constant $C$. But now, define the function $\vec i_\eps$ on $[0, 1]$ as a continuous approximation of the map $\frac{1}{\eps}\boldsymbol{1}_{(0, \eps]}$ in such a way that $\mylim_{\eps\to0}\|\frac{1}{\eps}\boldsymbol{1}_{(0, \eps]}-\vec i_\eps\|_{L^1([0, 1])}=0$. Then, we can write $\vec\eta_s^{\eps N}(x)=\langle\pi_s^N, \vec i_\eps(\cdot-\frac{x}{N})\rangle$ up to an error which vanishes in $L^1$ as $\eps\to0$. By writing $\bar\eta_s^{\eps N}(x)$ and $\cev\eta_s^{\eps N}(x)$ in a similar way and applying Lebesgue's differentiation theorem, it is easy to conclude. \end{proof}

\subsubsection{The Density Solves Item ii) of Definition \ref{def:weak_solutions}}
Finally, to show that $\Phi(\rho)\in L^2([0, T], \mathcal{H}^1)$, we observe that, by a standard argument, it suffices to show the following result, whose proof is postponed to Appendix \ref{sec:app_uniqueness}. Define the inner product $\llangle \cdot, \cdot \rrangle$ on the set $C^{0, 1}([0, T]\times(0, 1))$ via
\begin{equation*}
    \llangle G, H \rrangle:=\int_0^T\int_0^1 G_s(u) H_s(u)\de u\de s.
\end{equation*}
\begin{proposition}[Energy Estimate]\label{prop:energy_estimate}
Let $Q^*$ be a limit point of the sequence of measures $\{Q^N\}_N$. There exist positive constants $K_0, c<\infty$ such that
\begin{equation}
    \expected_{Q^*}\left[\mysup_{H\in C_c^{0, 1}([0, T]\times(0, 1))}\left\{\llangle \Phi(\rho), \partial_u H\rrangle -c\llangle \Phi(\rho), H^2\rrangle \right\} \right]\le K_0,
\end{equation}
where $C_c^{0, 1}([0, T]\times(0, 1))$ denotes the set of $C^{0, 1}$ functions on $[0, T]\times(0, 1)$ with compact support.
\end{proposition}
Hence, we conclude that $Q^*$ is concentrated on weak solutions of \eqref{eq:pde_zr} in the sense of Definition \ref{def:weak_solutions}.

\subsection{Replacement Lemma at the Bulk}\label{sec:replacement_bulk}
We prove in this subsection a replacement lemma for a sequence of probability measures stochastically dominated by the product invariant measure $\bar\nu_N$. Recall that, for a fixed time $T>0$ and for a probability measure $\mu^N$ on $\mathbb{N}^{I_N}$, we denote by $\mathbb{P}_{\mu^N}$ the probability measure on the space $\mathcal D([0,T], \mathbb{N}^{I_N})$ corresponding to the zero-range process with generator $\mathcal{L}_N$, defined in \eqref{eq:generator}, speeded up by $N^2$ and starting from a measure $\mu^N$, and by $\expected_{\mu^N}[\,\cdot\,]$ expectations with respect to $\prob_{\mu^N}$. Given $\ell\in\mathbb{N}$, let $\bar\eta^\ell_s(x) :=\frac{1}{2\ell+1} \sum_{|y-x|\le \ell}\eta_s(y)$, and for $x\in I_N$, let $\tau_x$ denote the translation operator by $x$, so that $\tau_x\eta(\cdot):=\eta(x+\cdot)$.
\begin{lemma}[Bulk Replacement Lemma]\label{lemma:repl_bulk}
Under the hypotheses of Theorem \ref{thm:main}, for any $t\in[0,T]$ and any $G\in C^{1, 2}([0, T]\times[0, 1])$, we have that
\begin{equation}
    \mylimsup_{\eps\to0}\mylimsup_{N\to\infty} \expected_{\mu^N}\left[\left|\mathscr{R}_{N,\eps}^0(G,t,\eta)\right|\right]=0,
\end{equation}
where 
 \begin{equation*}
    \mathscr{R}_{N,\eps}^0(G, t, \eta):=\int_0^t \frac{1}{N} \sum_{x\in I^\eps_N} \Delta G_s\left(\frac{x}{N}\right) \tau_xV_{\eps N}(\eta_s)\de s,
\end{equation*}
and, for $\ell\in\mathbb{N}$,
\begin{equation}\label{eq:Vl}
    V_\ell(\eta):= \frac{1}{2\ell+1} \sum_{|y|\le \ell} g(\eta(y))-\Phi(\bar\eta^{\ell}(0)).
\end{equation}
\end{lemma}

\begin{remark} The function $g$ does not play a special role in Lemma \ref{lemma:repl_bulk}, nor in Lemma \ref{lemma:repl_boundaries} in the next subsection. In fact, recalling Remark \ref{remark:Psi} and Corollary \ref{corol:Psi}, if $\Psi$ is a cylinder Lipschitz function, one can check that the same proof goes through with $g$ replaced by $\Psi$ and the corresponding $\tilde\Psi$ taking the role of $\Phi$.
\end{remark}

\begin{remark}
Before proving Lemma \ref{lemma:repl_bulk}, let us understand the meaning of  $\Phi(\overline{\eta}^{\varepsilon N}_{s}(x))$. Using Lemma \ref{lemma:Phi_lips}, we have that   
\begin{equation}\label{eq:bound_Phi}
    \begin{split}
    \Bigg|\Phi(\bar\eta^{\eps N}_s(x))&-\frac{1}{2\eps N+1}\sum_{|y|\le \eps N}\Phi(\eta_s(x+y))\Bigg|
    \\&\le \frac{1}{2\eps N+1}\sum_{|y|\le \eps N} \Big|\Phi(\bar\eta^{\eps N}_s(x))-\Phi(\eta_s(x+y))\Big|
    \\& \le \frac{1}{2\eps N+1}\sum_{|y|\le \eps N} g^*\Big|\overline{\eta}^{\eps N}_s(x)-\eta_s(x+y)\Big|
    \\&\le g^*\,\Big\{ \bar\eta^{\eps N}_s(x)+\frac{1}{2\eps N+1}\sum_{|y|\le \eps N} \eta_s(x+y)\Big\}
    \\&= 2\,g^*\,\bar\eta^{\eps N}_s(x).
    \end{split}
\end{equation}
\end{remark}

\begin{proof}[Proof of Lemma \ref{lemma:repl_bulk}]
Let $\varphi\ge\mymax\{\mysup_{x\in I_N}\bar\varphi_N(x), \frac{\alpha}{\lambda}, \frac{\beta}{\delta}\}$: we start by bounding the expectation in the statement from above by the sum of 
\begin{equation}\label{eq:equivalence_term}
    \expected_{\mu^N}\left[\left| \int_0^t \frac{1}{N} \sum_{x \in I^\eps_N} \Delta G_s\left( \frac{x}{N} \right) \tau_x\widehat V_{\eps N}(\eta_s)\de s\right|\right]
\end{equation}
and 
\begin{equation}\label{eq:blocks_term}
    \expected_{\mu^N}\left[\left|\int_0^t \frac{1}{N} \sum_{x \in I^\eps_N} \Delta G_s\left( \frac{x}{N} \right) \tau_x\widetilde{V}_{\eps N}(\eta_s)\de s\right|\right],
\end{equation}
where, for $\ell$ positive integer,
\begin{equation*}
    \widehat{V}_{\ell}(\eta):= E_{\bar\nu_\varphi}\Big[g(\eta(0))\Big|\bar\eta^{\ell}(0)\Big]- \Phi(\bar\eta^{\ell}(0))
\end{equation*}
and 
\begin{equation}\label{eq:Vl_tilde}
    \widetilde V_{\ell}(\eta):= \frac{1}{2\ell+1} \sum_{|y|\le\ell} g(\eta(y))-E_{\bar\nu_\varphi}\Big[g(\eta(0))\Big|\bar\eta^{\ell}(0)\Big].
\end{equation}
Here, the measure $\bar\nu_\varphi$ denotes the one defined in \eqref{eq:measure_alpha}. Before proceeding with the proof, we emphasise that the argument is reduced to estimating the last two displays, which involve expectations with respect to the reference measure $\bar\nu_\varphi$ rather than the invariant measure $\bar\nu_N$. This choice is essential: many of our arguments rely on properties of the reference measure such as translation invariance, a spectral gap estimate and an equivalence of ensembles, which are available in the literature for $\bar\nu_\varphi$ but not for $\bar\nu_N$. Throughout the proof, we will indicate precisely where these properties are invoked.

\vspace{1em}
\textbf{Cutoff of Large Densities and Equivalence of Ensembles.} Note that, given $A>0$, we can bound \eqref{eq:equivalence_term} from above by 
\begin{equation}\label{eq:equivalence_term_large}
    \expected_{\mu^N}\left[\left| \int_0^t \frac{1}{N} \sum_{x \in I^\eps_N} \Delta G_s\left( \frac{x}{N} \right) \tau_x\widehat V_{\eps N}(\eta_s)\boldsymbol{1}_{(A, \infty)}\big(\bar\eta_s^{\eps N}(x)\big)\de s\right|\right]
\end{equation}
and
\begin{equation}\label{eq:equivalence_term_small}
    \expected_{\mu^N}\left[\left| \int_0^t \frac{1}{N} \sum_{x \in I^\eps_N} \Delta G_s\left( \frac{x}{N} \right) \tau_x\widehat V_{\eps N}(\eta_s)\boldsymbol{1}_{[0, A]}\big(\bar\eta_s^{\eps N}(x)\big)\de s\right|\right].
\end{equation}
To treat \eqref{eq:equivalence_term_large}, first note that, by \eqref{eq:function_g}, \eqref{eq:expectation_alpha_eta} and \eqref{eq:alpha}, for all $\alpha\ge0$ and $x\in I_N$,
\begin{equation*}
    |\Phi(\alpha)|\le E_{\nu_{\alpha}}\left|g(\eta(x))\right|=E_{\nu_{\alpha}}\left|g(\eta(x))-g(0)\right|\le g^* E_{\nu_{\alpha}}[\eta(x)]= g^*\alpha.
\end{equation*}
Moreover, by the product form and translation invariance of $\bar\nu_\varphi$,
\begin{equation}\label{eq:cond_expectation}
    E_{\bar\nu_\varphi}\Big[g(\eta(x))\Big|\bar \eta^{\ell}(x)\Big]=\frac{1}{2\ell+1}\sum_{|y-x|\le \ell}E_{\bar\nu_\varphi}\Big[g(\eta(y))\Big|\bar\eta^{\ell}(x)\Big]
\end{equation}
for each $x\in I_N^\eps$. This implies that
\begin{equation*}
    \begin{split}
    \left|\tau_x\widehat V_\ell(\eta)\right|&\le \frac{1}{2\ell+1}\sum_{|y-x|\le \ell}E_{\bar\nu_\varphi}\Big[g(\eta(y))\Big|\bar\eta^{\ell}(x)\Big] +|\Phi(\bar\eta^{\ell}(x))|
    \\&\le g^*\Bigg(\frac{1}{2\ell+1} \sum_{|y-x|\leq \ell} \eta(y)\Bigg) +g^* \bar\eta^{\ell}(x)
    \\&=2g^*\bar\eta^{\ell}(x),
    \end{split}
\end{equation*}
and hence $|\Delta G_s\left(\frac{x}{N} \right)\tau_x \widehat V_{\eps N}(\eta)|\le 2g^*\|\Delta G\|_\infty \bar\eta^{\eps N}(x)$, where $\|\cdot\|_\infty$ denotes the $L^\infty$-norm on $[0, T]\times[0, 1]$. Observe now that the map $\eta \mapsto \bar\eta^{\eps N}(x)\boldsymbol{1}_{(A, \infty)}(\bar\eta^{\eps N}(x))$ is non decreasing. Hence, by attractiveness  (\textsc{Assumption (ND)}, \eqref{eq:stoch_domination} and Lemma \ref{lemma:attractiveness}) and the inequality $\boldsymbol{1}_{(A, \infty)}(x)\le \frac{x}{A}$ for $x\ge0$, we get
\begin{align*}
    \eqref{eq:equivalence_term_large}&\le 2g^*\|\Delta G\|_\infty \expected_{\mu^N}\left[\int_0^t \frac{1}{N} \sum_{x \in I^\eps_N} \bar\eta_s^{\eps N}(x)\boldsymbol{1}_{(A, \infty)}\big(\bar\eta_s^{\eps N}(x)\big)\de s\right]
    \\&\le 2g^*\|\Delta G\|_\infty \expected_{\bar\nu_N}\left[\int_0^t \frac{1}{N} \sum_{x \in I^\eps_N} \bar\eta_s^{\eps N}(x)\boldsymbol{1}_{(A, \infty)}\big(\bar\eta_s^{\eps N}(x)\big)\de s\right]
    \\&\le \frac{2t g^*\|\Delta G\|_\infty}{AN}\sum_{x\in I^\eps_N} E_{\bar\nu_N}\left[\bar\eta^{\eps N}(x)^2\right]
    \\&\le \frac{2t g^*\|\Delta G\|_\infty}{A}\mysup_{x\in I_N} E_{\bar\nu_N}\left[\eta(x)^2\right].
\end{align*}
By Lemma \ref{lemma:bounded_moments}, the last display can be bounded from above by a constant times  $A^{-1}$, and therefore it  goes to zero as $A\to\infty.$

As for \eqref{eq:equivalence_term_small}, by the equivalence of ensembles, namely \cite[Corollary A2.1.7]{kl99}, we know that, for any $\ell>0$, 
\begin{equation*}
    \mysup_{\eta\in\Omega_N}\left|\widehat{V}_{\ell}(\eta)\right|\boldsymbol{1}_{[0, A]}\big(\bar\eta^{\ell}(0)\big)\le \frac {C}{\ell},
\end{equation*}
where the constant $C$ does not depend on $N, \ell$ nor $A$. This yields
\begin{equation*}
    \eqref{eq:equivalence_term_small}\le \frac{tC\|\Delta G\|_\infty}{\eps N},
\end{equation*}
which vanishes as $N\to\infty.$   

\vspace{1em}
\textbf{The One-Block and Two-Block Estimates.} It remains to analyse \eqref{eq:blocks_term}. By the same argument presented in \cite[Section 5.3]{kl99}, this can be achieved by two lemmas, which are well-known in the literature by the name of \textit{one-block} and \textit{two-block estimates}. We state and prove them below. \end{proof}

\begin{lemma}[One-Block Estimate] \label{lemma:one_block} Under the hypotheses of Theorem \ref{thm:main}, for any $t\in[0,T]$ and any $G\in C^{1, 2}([0, T]\times[0, 1])$, we have that
\begin{equation}\label{eq:one_block}
    \mylimsup_{\ell\to\infty}\mylimsup_{N\to\infty}\expected_{\mu^N}\left[\left|\int_0^t \frac{1}{N} \sum_{x\in I^\eps_N} \Delta G_s\left(\frac{x}{N}\right) \tau_x\widetilde V_{\ell}(\eta_s)\de s\right|\right]=0.
\end{equation}
\end{lemma}
\begin{proof} \textbf{Cutoff of Large Densities.} Fix $A>0$: we start by bounding the expectation in the statement from above by the sum of
\begin{equation}\label{eq:largeDensities1}
    \expected_{\mu^N}\left[\left|\int_0^t \frac{1}{N} \sum_{x \in I^\eps_N} \Delta G_s\left(\frac{x}{N}\right) \tau_x\widetilde V_{\ell}(\eta_s)\boldsymbol{1}_{(A, \infty)}\big(\bar\eta_s^{\ell}(x)\big)\de s\right|\right]
\end{equation}
and
\begin{equation}\label{eq:smallDensities1}
    \expected_{\mu^N}\left[\left|\int_0^t \frac{1}{N} \sum_{x \in I^\eps_N} \Delta G_s\left(\frac{x}{N}\right) \tau_x\widetilde V_{\ell}(\eta_s)\boldsymbol{1}_{[0, A]}\big(\bar\eta_s^{\ell}(x)\big)\de s\right|\right].
\end{equation}
Recalling \eqref{eq:Vl_tilde}, by \eqref{eq:cond_expectation} and Lemma \ref{lemma:Phi_lips}, we have that
\begin{align*}
    \left|\tau_x\widetilde V_{\ell}(\eta)\right|&\le \frac{1}{2\ell+1} \sum_{|y-x|\le \ell} \left\{g(\eta(y))+E_{\bar\nu_\varphi}\Big[g(\eta(y))\Big|\overline{\rm \eta}^{\ell}(x)\Big]\right\}
    \le 2 g^*\bar\eta^\ell(x).
\end{align*}
But now, since the map $\eta\mapsto \bar\eta^\ell(x)\boldsymbol{1}_{(A, \infty)}\big(\bar\eta^\ell(x)\big)$ is non decreasing, by \textsc{Assumption (ND)}, \eqref{eq:stoch_domination}, Lemma \ref{lemma:attractiveness} and the inequality $\boldsymbol{1}_{(A, \infty)}(x)\le\frac{x}{A}$ for $x\ge0$, we see that
\begin{align*}
    \eqref{eq:largeDensities1}&\le 2g^*\expected_{\bar\nu_N} \left[\int_0^t \frac{1}{N} \sum_{x \in I^\eps_N} \left|\Delta G_s\left(\frac{x}{N}\right)\right| \bar\eta^\ell_s(x)\boldsymbol{1}_{(A, \infty)}\big(\bar\eta_s^{\ell}(x)\big)\de s\right]
    \\&\le \frac{2g^*t\|\Delta G\|_\infty}{AN}\sum_{x\in I_N^\eps }E_{\bar\nu_N}\left[\bar\eta^\ell(x)^2\right]
    \\&\le \frac{2Cg^*t\|\Delta G\|_\infty}{A}\mysup_{x\in I_N}E_{\bar\nu_N}\left[\eta(x)^2\right],
\end{align*}
which vanishes as $N\to\infty$ and $A\to\infty$ by Lemma \ref{lemma:bounded_moments}.

\vspace{1em} \textbf{Entropy Inequality and Feynmac-Kac.} 
Recall the choice of $\varphi$ in \eqref{eq:Vl_tilde}: by the entropy inequality and Jensen's inequality, for each $M>0$, \eqref{eq:smallDensities1} is bounded from above by
\begin{equation*}
    \frac{H(\mu^N|\bar\nu_\varphi)}{MN}+\frac{1}{MN}\mylog\expected_{\bar\nu_\varphi}\left[e^{MN\left|\int_0^t\frac{1}{N} \sum_{x \in I^\eps_N} \Delta G_s\left( \tfrac{x}{N} \right) \tau_x\widetilde V_{\ell}(\eta_s)\boldsymbol{1}_{[0, A]}\big(\bar\eta_s^{\ell}(x)\big)\de s\right|}\right].
\end{equation*}
Note that, in fact, the entropy condition stated in Theorem \ref{thm:main} is $H(\mu^N|\bar\nu_N)\lesssim N$, but from Appendix \ref{sec:app_entropy}  this implies  $H(\mu^N|\bar\nu_\varphi)\lesssim N$. Therefore, the leftmost term in the previous display  vanishes as $N\to\infty$ and $M\to\infty$, so we are only left to bound the second term. Using the fact that $e^{|x|}\le e^x+e^{-x}$, we can remove the absolute value in the exponential: then, setting
\begin{equation}\label{def:mx}
    m_x^\ell(\eta):=\tau_x\widetilde V_{\ell}(\eta)\boldsymbol{1}_{[0, A]}\big(\bar\eta^{\ell}(x)\big),
\end{equation}
by the Feynman-Kac formula \cite[Lemma A.1]{bmns17}, this term can be bounded from above by
\begin{align}
    &\mysup_f\left\{\int_0^t \frac{1}{N} \sum_{x \in I^\eps_N} G_s\left(\frac{x}{N}\right)\left\langle m_x^\ell, f\right\rangle _{\bar\nu_\varphi}\de s+\frac{tN}{M}\left\langle\mathcal{L}^N\sqrt{f}, \sqrt{f}\right\rangle_{\bar\nu_\varphi}\right\}\nonumber
    \\&\le \mysup_f\left\{\int_0^t \frac{1}{N} \sum_{x \in I^\eps_N} G_s\left(\frac{x}{N}\right)\left\langle m_x^\ell, f\right\rangle_{\bar\nu_\varphi}\de s-\frac{tN}{2M}\mathscr{D}^0_{N}(\sqrt{f},\bar{\nu}_{\varphi})\right\}+\frac{tCN}{MN^\theta},\label{eq:supremum1}
\end{align} where the supremum runs over all probability densities $f$ with respect to $\bar\nu_\varphi$. Here, $\mathscr{D}_N^0$ denotes the non-negative function defined by, for positive densities $f$ with respect to $\bar\nu_\varphi$,
\begin{equation}
    \begin{split}&\mathscr{D}^0_{N}\left(\sqrt{f},\bar{\nu}_{\varphi}\right) :=  \int_{\Omega_N}\Bigg\{\sum_{x=1}^{N-2} g(\eta(x)) \left[\sqrt{f(\eta^{x,x+1})}-\sqrt{f(\eta)}\right]^2 
    \\&\phantom{\mathscr{D}^0_{N}(\sqrt{f},\bar{\nu}_{\varphi}) :=  \int\Bigg\{}+ \sum_{x=2}^{N-1} g(\eta(x)) \left[\sqrt{f(\eta^{x,x-1})}-\sqrt{f(\eta)}\right]^2\Bigg\}\de \bar{\nu}_{\varphi},\end{split}\label{eq:dirichlet_bulk}
\end{equation} 
$C$ is a constant independent of $N$, and the inequality follows from Lemmas  \ref{lemma:left_boundary}, \ref{lemma:right_boundary} and \ref{lemma:bulk}.  

\vspace{1em}
\textbf{Localisation of the Dirichlet Form.} From here, we follow an approach similar to that in the proof of \cite[Lemma 3.1]{fmts23}. In the following bounds, it is crucial to work with a translation-invariant reference measure. Without this property, our computations would become much more involved. We start by \textit{localising} the bulk Dirichlet form, in the following sense. Fix $x\in I_N$, $\ell\in\mathbb N$ and let $\Lambda_x^{\ell}:=\{y\in I_N^\eps: |y-x|\le\ell\}$.  Setting
\begin{equation*}
    I_{x, y}(f):=\int_{\Omega_N}g(\eta(x))\left[f(\eta^{x, y})-f(\eta)\right]^2\de \bar\nu_\varphi,
\end{equation*}
we write
\begin{align*}
    \mathscr{D}^0_{N}(\sqrt{f},\bar{\nu}_{\varphi})&=\sum_{\substack{x, y\in I_N\\ |x-y|=1}}\sum_{z\in\Lambda_x^\ell}\frac{1}{|\Lambda_x^\ell|}I_{x, y}(\sqrt{f})
    \\&\ge\frac{1}{2\ell+1}\sum_{z\in I_N^\eps}\Big\{\sum_{\substack{x, y\in \Lambda_z^\ell\\ |y-x|=1}} I_{x, y}(\sqrt{f})\Big\}
    \\&=:\frac{1}{2\ell+1}\sum_{z\in I_N^\eps }\mathscr{D}_z^\ell(\sqrt{f}),
\end{align*}
where we used the fact that $|\Lambda_x^\ell|\le2\ell+1$ for all $x$, and $I_{x, y}\ge0$ for all $x, y$. Thus, we have that
\begin{equation}\label{eq:supremum2}
    \eqref{eq:supremum1}\le \int_0^t \sum_{x \in I^\eps_N} G_s\left(\frac{x}{N}\right)\mysup_f \left\{\frac{1}{N} \left\langle m_x^\ell, f\right\rangle _{\bar\nu_\varphi}-\frac{N}{2M(2\ell+1)}\mathscr{D}_x^\ell(\sqrt{f})\right\}\de s+\frac{tCN^{1-\theta}}{M}.
\end{equation}

Now, recalling the definition of $I_{x, y}$, by the same computations made in the proof of Lemma \ref{lemma:bulk} in the Appendix, we see that
\begin{equation}
    \mathscr{D}_x^\ell(\sqrt{f})=-2\left\langle \mathcal{L}_x^\ell\sqrt{f}, \sqrt{f}\right\rangle_{\bar\nu_\varphi},
\end{equation}
where $\mathcal{L}_x^\ell$ denotes the generator $\mathcal{L}^N$ restricted to $\Lambda_x^\ell$. Let now
\begin{equation*}
    \mathcal{X}_{x, j}^\ell:=\Big\{\eta\in \mathbb{N}^{\Lambda_x^\ell}: \sum_{y\in\Lambda_x^\ell}\eta_y=j\Big\}
\end{equation*}
denote the set of particle configurations on $\Lambda_x^\ell$ with exactly $j$ particles. For $ f:\mathcal{X}_{x, j}^\ell\to\mathbb{R}$, we introduce the generator 
\begin{align*}
    \mathcal{L}_{x, j}^\ell f(\eta):=\sum_{\substack{y, z\in\Lambda_x^\ell\\ |y-z|=1}}g(\eta(y))\left[f(\eta^{y, z})-f(\eta)\right]
\end{align*}
and the quadratic form 
\begin{equation*}
    \mathscr{D}_{x, j}^\ell(f):=-2\left\langle f, \mathcal{L}_{x, j}^\ell f \right\rangle_{\bar\nu_{x, j}^\ell}
\end{equation*}
with respect to the conditional measure $\bar\nu_{x, j}^\ell :=\bar\nu_\varphi\{\cdot\ |\ \eta\in\mathcal{X}_{x, j}^\ell\}$. Then, by conditioning on the event $\eta\in\mathcal{X}_{x, j}^\ell$ and taking the maximum over $j\le (2\ell+1)A$ (which we can do since $m_x^\ell$ lives on the event $\bar\eta^\ell(x)\le A$), we see that
\begin{equation}\label{eq:supremum3}
    \begin{split}\eqref{eq:supremum2}&\le \frac{N}{M(2\ell+1)}\int_0^t \sum_{x\in I_N^\eps} G_s\left(\frac{x}{N}\right)\times
    \\&\phantom{\le \frac{N}{M(2\ell+1)}\int_0^t}\times\mymax_{j\le (2\ell+1)A}\mysup_f \left\{\frac{M(2\ell+1)}{N^2}\left\langle m_x^\ell, f\right\rangle_{\bar\nu_{x, j}^\ell} -\frac{1}{2}\mathscr{D}_{x, j}^\ell(\sqrt{f})\right\}\de s
    \\&\phantom{\le}+\frac{tCN^{1-\theta}}{M},
    \end{split}
\end{equation}
where the supremum now runs along all probability densities $f$ on $\mathcal{X}_{x, j}^\ell$ with respect to $\bar\nu_{x, j}^\ell$.

\vspace{1em}
\textbf{Spectral Gap and Rayleigh Estimate.} Recalling \eqref{eq:Vl_tilde} and \eqref{def:mx}, by \eqref{eq:cond_expectation} and Lemma \ref{lemma:Phi_lips} we see that
\begin{align*}
   \left|m_x^\ell(\eta)\right|&=\left|\tau_x\widetilde V_{\ell}(\eta)\boldsymbol{1}_{[0, A]}\big(\bar\eta^{\ell}(x)\big)\right|
   \\&= \left|\frac{1}{2\ell+1} \sum_{y\in\Lambda_x^\ell} g(\eta(y))-E_{\bar\nu_\varphi}\Big[g(\eta(x))\Big|\bar\eta^{\ell}(x)\Big]\right|\boldsymbol{1}_{[0, A]}\big(\bar\eta^{\ell}(x)\big)
   \\&\le 2Ag^*.
\end{align*}
Moreover, since $j\le (2\ell+1)A$ and again by \eqref{eq:cond_expectation},
\begin{equation*}
    \langle m_x^\ell\rangle_{\bar\nu^\ell_{x, j}} = E_{\bar\nu^\ell_{x, j}}\left[\frac{1}{2\ell+1}\sum_{y\in\Lambda_x^\ell}\left\{g(\eta(y))-E_{\bar\nu_\varphi}\Big[g(\eta(y))\Big|\bar\eta^{\ell}(x)\Big]\right\}\right]=0.
\end{equation*}
Thus, since $m_x^\ell$ is centred and uniformly bounded, we can apply the Rayleigh estimate \cite[Theorem A3.1.1]{kl99}, so that for each $x$ the maximum in \eqref{eq:supremum3} satisfies
\begin{align}
    &\mymax_{j\le (2\ell+1)A}\mysup_f \left\{\frac{M(2\ell+1)}{N^2}\left\langle m_x^\ell, f\right\rangle_{\bar\nu_{x, j}^\ell} - \frac{1}{2}\mathscr{D}_{x, j}^\ell(\sqrt{f})\right\}\nonumber
    \\&\le \frac{M^2(2\ell+1)^2}{N^4} \mymax_{j\le (2\ell+1)A}\frac{\left\langle(-\mathcal{L}_{x, j}^\ell)^{-1} m_x^\ell, m_x^\ell\right\rangle_{\bar\nu_{x, j}^\ell}}{1-\frac{4M(2\ell+1)\|m_x^\ell\|_\infty}{N^2\text{gap}(j, \ell)}}.\label{eq:maxTerm}
\end{align}
Here, $\text{gap}(j, \ell)$ is the spectral gap of $\mathcal{L}_{x, j}^\ell$. Since, by \textsc{Assumption (SG)}, $\text{gap}(j, \ell)\ge\frac{C}{\ell^2}$, the denominator in the maximum in \eqref{eq:maxTerm} is bounded from below by $1-CM\ell^3N^{-2}$ for some constant $C$ independent of $\ell$ and $N$; the numerator, instead, satisfies
\begin{equation*}
    \left\langle(-\mathcal{L}_{x, j}^\ell)^{-1} m_x^\ell, m_x^\ell\right\rangle_{\bar\nu_{x, j}^\ell}\le \frac{\|m_x^\ell\|_\infty^2}{\text{gap}(j, \ell)}\le C\ell^2.
\end{equation*}
Combining everything together, we finally obtain that, for each $M>0$, the expectation in \eqref{eq:one_block} is bounded from above by
\begin{equation*}
\frac{C}{M}+ \frac{tCM\ell^3}{N^2(1-\frac{CM\ell^3}{N^2})} +\frac{tCN^{1-\theta}}{M},
\end{equation*}
where $C$ denotes a generic positive constant independent of $M$ and $\ell$. Since $M$ is arbitrary, by sending $N\to\infty$ the claim follows.
\end{proof}

\begin{lemma}[Two-Block Estimate] \label{lemma:two_block} Under the hypotheses of Theorem \ref{thm:main}, for any $t\in[0,T]$ and any $G\in C^{1, 2}([0, T]\times[0, 1])$, 
\begin{equation*}
    \begin{split}\mylimsup_{\ell \to \infty}\mylimsup_{\eps \to 0}\mylimsup_{N\to\infty}  \mysup_{|y|\leq \varepsilon N} \expected_{\mu^N}\Bigg[&\int_0^t \frac{1}{N}  \sum_{x \in I^\eps_N} \Delta G_s\left(\frac{x}{N}\right)\times
    \\&\times \left|\bar\eta^{\ell}_s(x+y)-\bar\eta^{\eps N}_s(x) \right| 
    \de  s\Bigg]=0.
    \end{split}
\end{equation*}
\end{lemma}
\begin{proof} \textbf{Reduction to Boxes of Size $\boldsymbol{\ell}$}. We start by noting that, for each {$x\in I_N$} and given $|y|\le 2\eps N$, 
\begin{align}
    \left|\bar\eta^{\ell}(x+y)-\bar\eta^{\eps N}(x) \right| &\le\frac{1}{2\eps N+1}\sum_{\substack{|z|\le \eps N\\ |y-z|>2\ell}}\left|\bar\eta^{\ell}(x+y)-\bar\eta^\ell(x+z) \right|\label{eq:box_sum}
    \\&\phantom{\le}+\mathscr{R}_{x, y}^N(\eta).\nonumber
\end{align}
Here, $\mathscr{R}_{x, y}^N(\eta)$ is the ``remainder" term which takes into account the terms outside of $\{x-\eps N, \ldots, x+\eps N\}$ as well as the boxes of size $\ell$ centred at $x+z$ with $|y-z|\le 2\ell$, namely
\begin{align*}
    &\mathscr{R}_{x, y}^N(\eta):= \frac{1}{2\eps N+1}\bigg\{\bar\eta^\ell(x+y-2\ell)+\ldots +\bar\eta^\ell(x+y+2\ell)
    \\&+\eta(x-\eps N-\ell)+2\eta(x-\eps N-\ell+1)+\ldots + (\ell-1)\eta(x-\eps N-1)
    \\&+(\ell-1)\eta(x+\eps N+1)+(\ell-2)\eta(x+\eps N+2)+\ldots + \eta(x+\eps N+\ell)\bigg\}.
\end{align*}
Now note that the map $\eta\mapsto\mathscr{R}_{x, y}^N(\eta)$ is non-decreasing: hence, by attractiveness (that is, \textsc{Assumption (ND)}, \eqref{eq:stoch_domination} and Lemma \ref{lemma:attractiveness}), the term in the expectation corresponding to $\mathscr{R}_{x, y}^N(\eta)$ satisfies
\begin{align*}
    \expected_{\mu^N}\left[\int_0^t \frac{1}{N}\sum_{x\in I^\eps_N}\Delta G_s\left(\frac{x}{N}\right)\mathscr{R}_{x, y}^N(\eta)\de s \right]&\le \expected_{\bar\nu_N}\left[\int_0^t \frac{1}{N}\sum_{x\in I^\eps_N}\Delta G_s\left(\frac{x}{N}\right)\mathscr{R}_{x, y}^N(\eta)\de s \right]
    \\&\le \frac{Ct\ell^2\|\Delta G\|_\infty}{2\eps N+1}\mysup_{x\in I_N}E_{\bar\nu_N}[\eta(x)]
\end{align*}
for some constant $C$ independent of $N$, which vanishes as $N\to\infty$ by Lemma \ref{lemma:bounded_moments}. Hence, we are only left to treat \eqref{eq:box_sum}.

\vspace{1em} \textbf{Cutoff of Large Densities.} Let $A>0$: we can write the expectation in the statement corresponding to the rightmost term in the first line of \eqref{eq:box_sum} as the sum of
\begin{equation}\label{eq:largeDensities}
    \begin{split}\expected_{\mu^N}\Bigg[\int_0^t \frac{1}{N(2\eps N+1)}  &\sum_{x \in I^\eps_N} \Delta G_s\left(\frac{x}{N}\right)\sum_{\substack{|z|\le \eps N\\ |y-z|>2\ell}}\left|\bar\eta^{\ell}(x+y)-\bar\eta^\ell(x+z) \right|\times
    \\&\times \boldsymbol{1}_{(A, \infty)}\big(\bar\eta_s^{\ell}(x+y)\vee \bar\eta_s^{\ell}(x+z)\big)
    \de  s\Bigg]
    \end{split}
\end{equation}
and
\begin{equation}\label{eq:smallDensities}
    \begin{split}\expected_{\mu^N}\Bigg[\int_0^t \frac{1}{N(2\eps N+1)}&\sum_{x \in I^\eps_N} \Delta G_s\left(\frac{x}{N}\right)\sum_{\substack{|z|\le \eps N\\ |y-z|>2\ell}}\left|\bar\eta^{\ell}(x+y)-\bar\eta^\ell(x+z) \right|\times
    \\&\times \boldsymbol{1}_{[0, A]}\big(\bar\eta_s^{\ell}(x+y)\vee \bar\eta_s^{\ell}(x+z)\big)
    \de  s\Bigg].
    \end{split}
\end{equation}
Now, \eqref{eq:largeDensities} can be bounded from above by
\begin{equation*}
    \begin{split}\frac{\|\Delta G\|_\infty}{N(2\eps N+1)}\expected_{\mu^N}\Bigg[\int_0^t \sum_{x \in I^\eps_N}&\sum_{\substack{|z|\le \eps N\\ |y-z|>2\ell}}\left[\bar\eta^{\ell}(x+y)+\bar\eta^\ell(x+z) \right]\times
    \\&\times\boldsymbol{1}_{(A, \infty)}\big(\bar\eta_s^{\ell}(x+y)\vee \bar\eta_s^{\ell}(x+z)\big)\de s\Bigg].
    \end{split}
\end{equation*}
Note that the map $\eta\mapsto \left[\bar\eta^{\ell}(x+y)+\bar\eta^\ell(x+z) \right] \boldsymbol{1}_{(A, \infty)}\big(\bar\eta^{\ell}(x+y)\vee \bar\eta^\ell(x+z)\big)$ is non decreasing. Thus, by attractiveness (\textsc{Assumption (ND)}, \eqref{eq:stoch_domination} and Lemma \ref{lemma:attractiveness}), the expression in the last display is bounded from above by
\begin{align*}
    &\frac{t\|\Delta G\|_\infty}{N(2\eps N+1)}  \sum_{x \in I^\eps_N}\hspace{-.3em}\sum_{\substack{|z|\le \eps N\\ |y-z|>2\ell}}\hspace{-.4em}E_{\bar\nu_N}\left[\left[\bar\eta^{\ell}(x+y)+\bar\eta^\ell(x+z) \right]\boldsymbol{1}_{(A, \infty)}\big(\bar\eta^{\ell}(x+y)\vee \bar\eta^{\ell}(x+z)\big)\right]
    \\&\le\frac{t\|\Delta G\|_\infty}{AN(2\eps N+1)}  \sum_{x \in I^\eps_N}\sum_{\substack{|z|\le \eps N\\ |y-z|>2\ell}}E_{\bar\nu_N}\left[\left[\bar\eta^{\ell}(x+y)+\bar\eta^\ell(x+z) \right]^2\right]
    \\&\le\frac{2t\|\Delta G\|_\infty}{AN(2\eps N+1)} \sum_{x \in I^\eps_N}\sum_{\substack{|z|\le \eps N\\ |y-z|>2\ell}}\bigg\{E_{\bar\nu_N}\left[\bar\eta^{\ell}(x+y)^2\right]+E_{\bar\nu_N}\left[\bar\eta^\ell(x+z)^2\right]\bigg\}
    \\&\le\frac{2tC\|\Delta G\|_\infty}{A}\mysup_{x\in I_N}E_{\bar\nu_N}[\eta(x)^2],
\end{align*}
for some positive constant $C$ independent of $N$, where in first passage we used the inequality $\boldsymbol{1}_{(A, \infty)}(x\vee y)\le\frac{x+y}{A}$ for $x, y\ge0$. This vanishes as $A\to\infty$ thanks to Lemma \ref{lemma:bounded_moments}.

\vspace{1em} \textbf{Entropy Bound and Feynman-Kac.} We are only left to treat \eqref{eq:smallDensities}. To that end, let 
\begin{equation}\label{eq:p}
    p_{y, z}^{\ell}(\eta):=\left|\bar\eta^{\ell}(y)-\bar\eta^\ell(z) \right|\boldsymbol{1}_{[0, A]}\big(\bar\eta^\ell(y)\vee \bar\eta^\ell(z)\big).
\end{equation}
By repeating the same argument given in the proof of Lemma \ref{lemma:one_block} above, for each $M>0$ and $\varphi\ge\mymax\{\mysup_{x\in I_N}\bar\varphi_N(x), \frac{\alpha}{\lambda}, \frac{\beta}{\delta}\}$, the expectation in the statement corresponding to \eqref{eq:box_sum} is bounded from above by
\begin{equation}\label{eq:supremum4}
    \begin{split}&\mysup_f\left\{\int_0^t \frac{1}{N(2\eps N+1)} \sum_{x \in I^\eps_N}{\Delta G}_s\left(\frac{x}{N}\right)\hspace{-.4em}\sum_{\substack{|z|\le \eps N\\ |y-z|>2\ell}} \hspace{-.3em}\left\langle p_{x+y, x+z}^\ell, f\right\rangle_{\bar\nu_\varphi}\de s-\frac{tN}{2M}\mathscr{D}^0_{N}(\sqrt{f},\bar{\nu}_{\varphi})\right\}
    \\&+\frac{C}{M}+\frac{tCN^{1-\theta}}{M},
    \end{split}
\end{equation}
where the supremum runs over all probability densities $f$ with respect to $\bar\nu_\varphi$.

\vspace{1em}
\textbf{Comparison to the Two-Coordinate Process.} In this step, we adapt an argument given in the proof of \cite[Lemma 5.3.5]{kl99}. Let $\Lambda_{y, z}^{\ell}:=\Lambda_y^\ell \times \Lambda_z^\ell$, and let $\bar\nu_\varphi^{\ell, \ell}$ denote the restriction of $\bar\nu_\varphi$ to $\Lambda_{y, z}^{\ell}$. For positive densities $f$ on $\mathbb{N}^{\Lambda_{y, z}^\ell}$, we define the Dirichlet form
\begin{equation*}
    \mathscr{D}_{y, z}^\ell(f):=\mathscr{D}_{y, z}^{1, \ell}(f)+\mathscr{D}_{y, z}^{2, \ell}(f)+\mathcal{I}_{y, z}^{1, \ell}(f)+\mathcal{I}_{y, z}^{2, \ell}(f),
\end{equation*} where
\begin{align*}
    &\mathscr{D}_{y, z}^{1, \ell}(f):=\frac{1}{2}\sum_{\substack{w, w'\in \Lambda_y^\ell\\ |w-w'|=1}}\int g(\xi_1(w))\left[f\big(\xi_1^{w, w'}, \xi_2\big)-f(\xi)\right]^2\de \bar\nu_\varphi^{\ell, \ell},
    \\&\mathscr{D}_{y, z}^{2, \ell}(f):=\frac{1}{2}\sum_{\substack{w, w'\in \Lambda_z^\ell\\ |w-w'|=1}}\int g(\xi_2(w))\left[f\big(\xi_1, \xi_2^{w, w'}\big)-f(\xi)\right]^2\de \bar\nu_\varphi^{\ell, \ell},
    \\&\mathcal{I}_{y, z}^{1, \ell}(f):=\frac{1}{2}\int g(\xi_1(y))\left[f\big(\xi_1^{y-}, \xi_2^{z+}\big)-f(\xi)\right]^2\de \bar\nu_\varphi^{\ell, \ell},
    \\&\mathcal{I}_{y, z}^{2, \ell}(f):=\frac{1}{2}\int g(\xi_2(z))\left[f\big(\xi_1^{y+}, \xi_2^{z-}\big)-f(\xi)\right]^2\de \bar\nu_\varphi^{\ell, \ell},
\end{align*}
and the integrals run on $\mathbb{N}^{\Lambda_{y, z}^\ell}$. This Dirichlet form corresponds to an interacting particle system on $\Lambda_{y, z}^\ell$ where each marginal process evolves as a zero-range process with rate $g$, and where particles are allowed to jump from the midpoint of one of the boxes to the midpoint of the other, and vice versa.

Now, given a density $f$ on $\Omega_N$ with respect to $\bar\nu_\varphi$, let $f_{y, z}^\ell$ denote the conditional expectation of $f$ given $\sigma\{\eta(w): w\in\Lambda_y^\ell\cup \Lambda_z^\ell\}$. For $y, z$ satisfying $|y-z|>2\ell$, we see that $f_{y, z}^\ell$ can be viewed as a density on $\mathbb{N}^{\Lambda_{y, z}^\ell}$. Note that, since the Dirichlet form is convex and since conditional expectation is an average,
\begin{equation*}\label{eq:single_box_bound}
    \mathscr{D}_{y, z}^{1, \ell}(f_{y, z}^\ell)\le \frac{1}{2}\sum_{\substack{w, w'\in \Lambda_y^\ell\\ |w-w'|=1}}\int_{\Omega_N}g(\eta(w))\left[f(\eta^{w, w'})-f(\eta)\right]^2\de \bar\nu_\varphi,
\end{equation*}
and an analogue bound holds for $\mathscr{D}_{y, z}^{2, \ell}$. Hence, we get
\begin{equation}\label{eq:bound_separate}
    \begin{split}&\sum_{x\in I_N^\eps}\sum_{\substack{|z|\le \eps N\\ |y-z|>2\ell}} \left\{\mathscr{D}_{x+y, x+z}^{1, \ell}\left(\sqrt{f_{x+y, x+z}^\ell}\right)+\mathscr{D}_{x+y, x+z}^{2, \ell}\left(\sqrt{f_{x+y, x+z}^\ell}\right)\right\}
    \\&\lesssim \ell\eps N\mathscr{D}_N^0(\sqrt{f}, \bar\nu_\varphi).
    \end{split}
\end{equation}

It remains to derive an upper bound for the Dirichlet forms that connect the two boxes. To that end, first note that, in the same way as we derived \eqref{eq:single_box_bound}, we have that
\begin{equation*}
    \mathcal{I}_{y, z}^{1, \ell}(f_{y, z}^\ell)\le \frac{1}{2}\int_{\Omega_N} g(\eta(y)) \left[f(\eta^{y, z})-f(\eta)\right]^2\de \bar\nu_\varphi.
\end{equation*}
Now let $\{w_k\}_{k=0, \ldots, |y-z|-1}$ be the shortest path from $y$ to $z$, namely the sequence of shortest length of consecutive sites such that the first one is the $y$, the last one is $z$ and the distance between two consecutive sites is equal to $1$. Defining inductively 
\begin{equation*}
    \begin{cases}\eta_{(0)}:=\eta,
    \\ \eta_{(k)}:=\eta_{(k-1)}^{w_{k-1}, w_k}, & k=1, \ldots, |y-z|-1,
    \end{cases} 
\end{equation*}
we write
\begin{equation*}
    f(\eta^{y, z})-f(\eta)=\sum_{k=0}^{|y-z|-1}\left[f\big(\eta_{(k+1)}\big)-f\big(\eta_{(k)}\big)\right].
\end{equation*}
By the Cauchy-Schwarz inequality, we see that
\begin{equation*}
    [f(\eta^{y, z})-f(\eta)]^2 \le  |y-z| \sum_{k=0}^{|y-z|-1}\left[f\big(\eta_{(k+1)}\big)-f\big(\eta_{(k)}\big)\right]^2,
\end{equation*}
and thus, using the fact that $\frac{\bar\nu_\varphi(\eta_{(k+1)})}{\bar\nu_\varphi(\eta_{(k)})}=\frac{g(\eta_{(k)}(w_k))}{g(\eta_{(k+1)}(w_{k+1}))}$ for each $k=0, \ldots, |y-z|-1$,
\begin{equation*}
    \mathcal{I}_{y, z}^{1, \ell}(f_{y, z}^\ell)\lesssim \eps N \sum_{k=0}^{|y-z|-1}\int_{\Omega_N} g(\eta_{(k)}(w_k))\left[f\big(\eta_{(k+1)}\big)-f\big(\eta_{(k)}\big)\right]^2\de \bar\nu_\varphi(\eta_{(k)}).
\end{equation*}
By repeating the same argument for $\mathcal{I}_{y, z}^{2, \ell}$, this yields
\begin{align*}
    &\sum_{x\in I_N^\eps}\sum_{\substack{|z|\le \eps N\\ |y-z|>2\ell}} \left\{\mathcal{I}_{x+y, x+z}^{1, \ell}\left(\sqrt{f_{x+y, x+z}^\ell}\right)+\mathcal{I}_{x+y, x+z}^{2, \ell}\left(\sqrt{f_{x+y, x+z}^\ell}\right)\right\}
    \\&\lesssim (\eps N)^3\mathscr{D}_N^0(\sqrt{f}, \bar\nu_\varphi).
\end{align*}
Finally, combining this bound with \eqref{eq:bound_separate}, we obtain
\begin{equation}\label{eq:dir_loc}
    N\mathscr{D}^0_{N}(\sqrt{f},\bar{\nu}_{\varphi})\ge \frac{C}{\eps^2N(2\eps N+1)} \sum_{x\in I_N^\eps} \sum_{\substack{|z|\le \eps N\\ |y-z|>2\ell}}  \mathscr{D}_{x+y, x+z}^\ell\left(\sqrt{f_{x+y, x+z}^\ell}\right)
\end{equation}
for some positive constant $C$ independent of $\eps, \ell$ and $N$.

\vspace{1em}
\textbf{Spectral Gap and Rayleigh Estimate.} Let now
\begin{equation*}
    \mathcal{X}_{y, z,  j}^\ell:=\Big\{\eta\in \mathbb{N}^{\Lambda_{y, z}^\ell}: \sum_{w\in\Lambda_{y}^\ell\cup \Lambda_z^\ell}\eta_w=j\Big\}
\end{equation*}
denote the set of particle configurations on $\Lambda_{y}^\ell\cup \Lambda_z^\ell$ with exactly $j$ particles. For $f:\mathcal{X}_{y, z, j}^\ell\to\mathbb{R}$, we introduce the generator  
\begin{align*}
    \mathcal{L}_{y, z, j}^\ell f(\eta):=\,& \sum_{\substack{w, w'\in\Lambda_y^\ell\\ |w-w'|=1}}g(\xi_1(w))\left[f(\xi_1^{w, w'}, \xi_2)-f(\xi)\right]
    \\&+\sum_{\substack{w, w'\in\Lambda_z^\ell\\ |w-w'|=1}}g(\xi_2(w))\left[f(\xi_1, \xi_2^{w, w'})-f(\xi)\right]
    \\&+g(\xi_1(y))\left[f\big(\xi_1^{y-}, \xi_2^{z+}\big)-f(\xi)\right]
    +g(\xi_2(z))\left[f\big(\xi_1^{y+}, \xi_2^{z-}\big)-f(\xi)\right] 
\end{align*} 
and the quadratic form
\begin{equation*}
    \mathscr{D}_{y, z, j}^\ell(f):=-2\left\langle f, \mathcal{L}_{y, z, j}^\ell f \right\rangle_{\bar\nu_{y, z, j}^\ell}
\end{equation*}
with respect to the measure $\bar\nu_{y, z, j}^\ell :=\bar\nu_\varphi\{\cdot\ |\ \eta\in\mathcal{X}_{y, z, j}^\ell\}$. Then, by conditioning on the event $\eta\in\mathcal{X}_{x+y, x+z, j}^\ell$, taking the maximum over $j\le 2(2\ell+1)A$ and using \eqref{eq:dir_loc}, we see that the supremum in \eqref{eq:supremum4} is bounded from above by
\begin{equation}\label{eq:supremum5}
\begin{split}
    &\frac{C}{M\eps ^2N(2\eps N+1)}\int_0^t \sum_{x \in I^\eps_N} G_s\left(\frac{x}{N}\right)\sum_{\substack{|z|\le \eps N\\ |y-z|>2\ell}}\mymax_{j\le 2(2\ell+1)A}\mysup_f \bigg\{\frac{M\eps ^2}{C} \times
    \\&\times \left\langle p_{x+y, x+z}^\ell, f\right\rangle_{\bar\nu_{x+y, x+z, j}}-\frac{1}{2}\mathscr{D}_{x+y, x+z, j}^\ell(\sqrt{f})\bigg\}\de s,
    \end{split}
\end{equation}
where the supremum now runs along all probability densities $f$ on $\mathcal{X}_{x+y, x+z,  j}^\ell$ with respect to $\bar\nu_{x+y, x+z, j}^\ell$. Note now that the function $p^\ell_{x+y, x+z}$ is \textit{not} centred, so we cannot directly apply the Rayleigh estimate. Nonetheless, since $\|p^\ell_{x+y, x+z}\|_\infty$ is finite, by \cite[Corollary A3.1.2]{kl99} the maximum in \eqref{eq:supremum5} satisfies
\begin{align}
    &\mymax_{j\le2(2\ell+1)A}\mysup_f \left\{\frac{M\eps^2}{C}\left\langle p_{x+y, x+z}^\ell, f\right\rangle_{\bar\nu_{x+y, x+z, j}} -\frac{1}{2}\mathscr{D}_{x+y, x+z,  j}^\ell(\sqrt{f})\right\}\nonumber
    \\\begin{split}&\le \frac{M^2\eps^4}{C} \mymax_{j\le 2(2\ell+1)A}\frac{\left\langle(-\mathcal{L}_{x+y, x+z, j}^\ell)^{-1} p_{x+y, x+z}^\ell, p_{x+y, x+z}^\ell\right\rangle_{\bar\nu_{x+y, x+z, j}^\ell}}{1-\frac{2M\eps ^2\|p_{x+y, x+z}^\ell\|_\infty}{\text{gap}(j, \ell, \ell)}}
    \\&\phantom{\le}+\mymax_{j\le2(2\ell+1)A}\frac{M\eps^2}{C}\langle p^\ell_{x+y, x+z}\rangle_{\bar\nu^\ell_{x+y, x+z, j}}.
    \end{split}\label{eq:maximum_twoB}
\end{align}
Here, $\text{gap}(j, \ell, \ell)$ denotes the spectral gap of the two-coordinate process generated by $\mathcal{L}_{y, z, j}^\ell$. Under \textsc{Assumption (SG)}, since we are restricting to $j\lesssim \ell$, by Lemma \ref{lemma:coupled_gap} we have that $\text{gap}(j, \ell, \ell)$ is bounded from below by a positive constant depending only on $\ell$ and $g$: hence, the first line of \eqref{eq:maximum_twoB} is bounded from above by
\begin{equation*}
    \frac{C(\ell)M^2\eps^4}{1-C(\ell)M\eps^2},
\end{equation*}
with $C(\ell)$ denoting a generic positive constant depending only on $\ell$ (and $g$).
As for the second line, note first that, for any $|z-y|<2\ell$,
\begin{align*}
    0&=\text{Var}_{\bar\nu_{y, z, j}}\left(\sum_{x\in \Lambda_y^\ell\cup\Lambda_z^\ell} \eta(x)\right)
    \\&=\sum_{x\in \Lambda_y^\ell\cup\Lambda_z^\ell} \text{Var}_{\bar\nu_{y, z, j}}(\eta(x))+\sum_{\substack{x, x'\in \Lambda_y^\ell\cup\Lambda_z^\ell \\ x\ne x'}}\text{Cov}_{\bar\nu_{y, z, j}}(\eta(x), 
    \eta(x')),
\end{align*}
and hence, since $\text{Cov}_{\bar\nu_{y, z, j}}(\eta(x), \eta(x'))$ is the same for all pairs $(x, x')$ with $x\ne x'$, we get
\begin{equation}\label{eq:negative_cov}
    \text{Cov}_{\bar\nu_{y, z, j}}(\eta(x), \eta(x'))\le 0
\end{equation}
for all $x\ne x'$ in $\Lambda_y^\ell\cup\Lambda_z^\ell$. Moreover, by the equivalence of ensembles, namely \cite[Corollary A2.1.7]{kl99}, choosing $\varphi_j:=\Phi\big(\frac{j}{2(2\ell+1)}\big)$ we have that
\begin{equation*}
    \text{Var}_{\bar\nu_{y, z, j}}(\eta(x))\le \text{Var}_{\bar\nu_{\varphi_j}}(\eta(x))+\frac{C}{\ell}
\end{equation*}
for some finite constant $C$. But then, we see that
\begin{align*}
    \langle p^\ell_{y, z}\rangle_{\bar\nu^\ell_{y, z, j}}&\le \langle |\bar\eta^\ell(y)-\bar\eta^\ell(z)|\rangle_{\bar\nu^\ell_{y, z, j}}
     \le \sqrt{\text{Var}_{\bar\nu_{y, z, j}}(\bar\eta^\ell(y)-\bar\eta^\ell(z))}
    \\&\le \frac{2}{\sqrt{2\ell+1}} \sqrt{\text{Var}_{\bar\nu_{y, z, j}}(\eta(y))}
    \le \frac{2}{\sqrt{2\ell+1}} \sqrt{\text{Var}_{\bar\nu_{\varphi_j}}(\eta(y))+\frac{C}{\ell}}
    \\&\lesssim \frac{1}{\sqrt{2\ell+1}},
\end{align*}
where in the third inequality we used $\bar\eta^\ell(y)-\bar\eta^\ell(z)=2\bar\eta^\ell(y)-\frac{j}{2(2\ell+1)}$ and \eqref{eq:negative_cov}, and in the last bound we applied Lemma \ref{lemma:bounded_moments}. Combining everything together, we finally get the upper bound 
\begin{align*}
    \eqref{eq:supremum4}&\le \frac{tC\eps N^2}{M\eps^2 N(2\eps N+1)} \left\{ \frac{C(\ell)M\eps ^4}{1-C(\ell)M\eps^2}+\frac{M\eps^2}{\sqrt{2\ell+1}}\right\}+\frac{C}{M}+\frac{tCN^{1-\theta}}{M}
    \\&\le \frac{tC(\ell)\eps^2}{1-C(\ell)M\eps^2}+\frac{tC}{\sqrt{\ell}}+\frac{C}{M}+\frac{tCN^{1-\theta}}{M}.
\end{align*}
Since $M$ is arbitrary, sending $N\to\infty$, then $\eps\to0$ and then $\ell\to\infty$ completes the proof. \end{proof}

\subsection{Replacement Lemma at the Boundary}\label{sec:replacement_boundary}
\begin{lemma}[Boundary Replacement Lemma]\label{lemma:repl_boundaries} Under the hypotheses of Theorem \ref{thm:main}, for any $t\in[0,T]$ and any continuous functions $G_1, G_2:[0, T]\to\mathbb{R}$, we have that
\begin{equation}
    \mylimsup_{\eps \to 0}\mylimsup_{N \to \infty} \expected_{\mu^N}\left[\left|\mathscr{R}_{N, \eps}^b(G_1, G_2, t, \eta)\right|\right]=0,
\end{equation}
where 
 \begin{equation*}
    \begin{split}
    \mathscr{R}_{N,\eps}^b(G_1, G_2, t, \eta):= \;&\int_0^t G_1(s)\left\{ \frac{1}{\eps N}\sum_{y=1}^{\eps N} g(\eta_s(y))-\Phi\left(\vec\eta^{\eps N}_s(1)\right)\right\}\de s
    \\&+\int_0^t G_2(s)\left\{\frac{1}{\eps N}\sum_{y=N-\eps N}^{N-1} g(\eta_s(y))-\Phi\left(\cev\eta^{\eps N}_s(N-1)\right)\right\}\de s,
    \end{split}
\end{equation*}
and $\vec\eta^{\eps N}_s(1)$ and $\cev\eta^{\eps N}_s(N-1)$ are defined in \eqref{eq:vec_eta}.
\end{lemma}
The proof of this lemma follows from an argument which is almost identical to the one given in the previous section, so we omit it. The reader is invited to retrace all the steps of the proof of Lemma \ref{lemma:repl_bulk}, replacing the centred averages with the directed averages, and will see that no additional argument is required.

\appendix
\section{Energy Estimate and Uniqueness of Weak Solutions}\label{sec:app_uniqueness}

\subsection{Energy Estimate}
In this section we prove Proposition \ref{prop:energy_estimate}. Our approach follows that of the proof of \cite[Lemma 5.7.2]{kl99}, and by the argument given therein, it suffices to show the following estimate. 
\begin{lemma}\label{lemma:energy_pre} Let $\{H_k\}_{k\ge1}$ be a sequence of functions dense in $C_c^{0, 1}([0, T]\times(0, 1))$. There exists $K_0<\infty$ such that, for each $m\ge1$ and $\eps>0$,
\begin{equation}\label{eq:energy_pregoal}
    \mylimsup_{\delta\to0}\mylimsup_{N\to\infty} \expected_{\mu^N}\left[\mymax_{1\le i\le m} \int_0^T U_N(\eps, \delta, H_i(s, \cdot), \eta_s)\de s\right]\le K_0,
\end{equation}
where, for $H\in C^1((0, 1))$, 
\begin{equation*}
    \begin{split}
    U_N(\eps, \delta, H, \eta):=&\;\sum_{x\in I_N^\eps} H\left(\frac{x}{N}\right)\frac{1}{\eps N}\left\{\Phi(\bar\eta^{\delta N}(x-\eps N))-\Phi(\bar\eta^{\delta N}(x+\eps N))\right\}
    \\&-2\sum_{x\in I_N^\eps} H\left(\frac{x}{N}\right)^2\frac{1}{\eps N} \sum_{|y-x|\le \eps N} \Phi(\bar\eta^{\delta N}(y)).
    \end{split}
\end{equation*}
\end{lemma}
\begin{proof} 
By the bulk replacement lemma, namely Lemma \ref{lemma:repl_boundaries}, and the argument given in Section \ref{sec:characterisation}, it is not hard to see that it suffices to show that
\begin{equation}\label{eq:energy_goal}
    \mylimsup_{N\to\infty} \expected_{\mu^N}\left[\mymax_{1\le i\le m} \int_0^TW_N(\eps, H_i(s, \cdot), \eta_s)\de s\right]\le K_0,
\end{equation}
where, for $H\in C^1((0, 1))$, 
\begin{equation*}
    \begin{split}
    W_N(\eps, H, \eta):=&\;\sum_{x\in I_N^\eps} H\left(\frac{x}{N}\right)\frac{1}{\eps N}\left\{g(\eta(x-\eps N))-g(\eta(x+\eps N))\right\}
    \\&-2\sum_{x\in I_N^\eps} H\left(\frac{x}{N}\right)^2\frac{1}{\eps N} \sum_{|y-x|\le \eps N} g(\eta(y)).
    \end{split}
\end{equation*}
Let $\varphi\ge\mymax\{\mysup_{x\in I_N}\bar\varphi_N(x), \frac{\alpha}{\lambda}, \frac{\beta}{\delta}\}$: using the fact that $e^{\mymax_{1\le i\le m} a_i}\le\sum_{i=1}^m e^{a_i}$ and $\mylimsup_{N\to\infty}\frac{1}{N}\{a_N+b_N\}\le \mymax\{\mylimsup_{N\to\infty}\frac{1}{N}\mylog a_N, \mylimsup_{N\to\infty}\frac{1}{N}\mylog b_N\}$, by the same argument given in the proof of Lemma \ref{lemma:one_block}, it suffices to bound the expression
\begin{equation}\label{eq:energy_supremum}
\begin{split}
    &\mymax_{1\le i \le m}\mysup_f \left\{\int_0^T \int_{\Omega_N} W_N(\eps, H_i(s, \cdot), \eta)f(\eta)\de \bar\nu_\varphi-\frac{TN}{2}\mathscr{D}_N^0\left(\sqrt{f}, \bar\nu_\varphi\right) \right\}
    \\&+C+TCN^{1-\theta},
    \end{split}
\end{equation}
where the supremum runs along all probability densities $f$ with respect to $\bar\nu_\varphi$. From here, the conclusion follows from the exact same argument given in the proof of \cite[Lemma 7.3]{kl99}. \end{proof}

\subsection{Uniqueness of Weak Solutions}
Finally, we prove Lemma \ref{lemma:uniqueness}. Throughout, we denote by $C_0^2([0, 1])$ the set of functions $h$ in $C^2([0, 1])$ such that $h(0)=h(1)=0$. We start by recalling some technical results, namely \cite[Lemmas 7.1 and 7.2]{bdgn20}, which will be used in what follows. We refer the reader to  \cite{bdgn20} for the proofs. 

\begin{lemma}\label{lemma:parabolic_uniqueness}
Suppose that: \begin{enumerate}[i)] 
\item $a=a(t, u)$ is a positive $C^{2,2}([0,T]\times [0,1])$ function, 
\item $b=b(t, u)$ is a function defined on $[0, T]\times[0, 1]$ such that $b(\cdot, 0)$ and $b(\cdot, 1)$ are in $C^2([0,T])$,
\item $h=h(u)$ is a  $C^2_0([0,1])$ function.
\end{enumerate}
Then, for each $t \in (0,T]$, the problem with Robin boundary conditions 
\begin{equation}\label{eq:parabolic_pde}
    \left\{
    \begin{array}{ll} 
    \partial_s\phi + a\Delta \phi  =0 &  \mbox{ for } (s,u) \in [0,t)\times(0,1),
    \\\partial_u\phi(s, 0)=b(s, 0)\,\phi(s, 0) & \mbox{ for } s \in [0,t),
    \\ \partial_u\phi(s, 1) = -b(s, 1)\,\phi(s, 1) & \mbox{ for } s \in [0,t),
    \\ \phi(t, u)  = h(u) &\mbox{ for } u \in (0,1)
    \end{array}\right.
\end{equation}
has a unique solution $\phi_0$ in $C^{1,2}([0,t]\times[0,1])$. Moreover, if $0 \le h\le1$, then 
\begin{equation*}
    0\le \phi_0(s, u) \le 1  \ \mbox{ for } (s,u) \in [0,t] \times [0,1]. 
\end{equation*}
\end{lemma}

\begin{lemma}\label{lemma:parabolic_bound}
Let $\phi_0$ be the unique solution of the parabolic equation \eqref{eq:parabolic_pde}. There exists a positive constant $C=C(b,h)$ such that 
\begin{equation*}
    \int_0^t\int_0^1 a(s, u)(\Delta \phi_0(s, u))^2\de u\de s \le C(b,h).
\end{equation*}
\end{lemma}

The lemma that follows is in fact a generalisation of \cite[Lemma 7.3]{bdgn20}, so we present its proof.

\begin{lemma}\label{lemma:convergent_sequence}
Let $a$ be a non-negative, bounded, measurable function on $[0,T] \times [0,1]$. There exists a sequence $\{a_k\}_{k\ge0}$ of positive functions in $C^{\infty}([0,T]\times [0,1])$ such that, for any $\omega\in L^2([0,T]\times[0,1])$,
\begin{equation*}
    \frac{1}{k}\le a_k \le \|a\|_{L^{\infty}}+\frac{1}{k} \ \ \text{and } \ \ \left\|\omega\;\frac{a-a_k}{\sqrt{a_k}}\right\|_{L^2([0,T]\times[0,1])} \to 0, 
\end{equation*}
as $k\to\infty$.
\end{lemma}
\begin{proof}
To ease the notation, throughout this proof we will write $L^2$ for $L^2([0, T]\times[0, 1])$. We start by observing that, for any positive sequence $\{a_k\}_{k\ge0}$, we have that
\begin{equation}\label{eq:bound_A}
    \left\|\omega\,\frac{a-a_k}{\sqrt{a_k}}\right\|_{L^2}\le \left\|\omega\,\frac{a-a_k}{\sqrt{a_k}}\boldsymbol{1}_{\mathcal{A}}\right\|_{L^2}+\left\|\omega\,\sqrt{a_k}\,\boldsymbol{1}_{\mathcal{A}^c}\right\|_{L^2},
\end{equation}
where $\mathcal{A}:=\{(s,u)\in[0, T]\times[0, 1]: a(s,u)>0\}$. We will construct a suitable sequence $\{a_k\}_{k\ge0}$, and we will treat the integrals on $\mathcal A$ and $\mathcal A^c$ separately. 

Note first that, on the set $\mathcal A$, we have that
\begin{equation*}
    \left|\omega\left\{\frac{\sqrt{a}}{\sqrt{a+1/k}} - 1\right\}\boldsymbol{1}_{\mathcal{A}} \right|\to 0 \ \ \text{as } k \to \infty\ \ \ \text{and } \ \ \left|\omega\left\{\frac{\sqrt{a}}{\sqrt{a+1/k}} - 1\right\}\boldsymbol{1}_{\mathcal{A}} \right|\le 2|\omega|.
\end{equation*}
Thus, we can use the dominated convergence theorem to see that, for all $\eps>0$, there exists $k_0\in\mathbb{N}$ such that
\begin{equation}\label{eq:convergent_integral_1}
    \left\|\omega\left\{\frac{\sqrt{a}}{\sqrt{a+1/k}} - 1\right\}\boldsymbol{1}_{\mathcal{A}} \right\|_{L^2} <\frac{\eps}{2} \ \ \ \text{for all }k\ge k_0.
\end{equation}
Let us now take a step towards constructing the sequence $\{a_k\}_{k\ge0}$. To this end, for each $k\ge k_0$ consider a sequence
$\{\xi_m^k\}_{m\ge 0}$ in $C^{\infty}([0,T]\times[0,1])$ such that $\xi_m^k \to a+\frac{1}{k}$ in $L^2$ as $m\to\infty$. Since $a+\frac{1}{k}\ge\frac{1}{k}$, we can assume that $ \frac{1}{k}\le \xi_m^k\le \|a\|_{\infty}+ \frac{1}{k}$ for all $m$. Using only this assumption, we obtain the inequality
\begin{equation*}
    \left|\omega\left\{\frac{\sqrt{a}}{\sqrt{\xi_m^k}} - \frac{\sqrt{a}}{\sqrt{a+1/k}}\right\}\boldsymbol{1}_{\mathcal{A}} \right|\le \frac{k^{\frac{3}{2}}}{2}\sqrt{\|a\|_{\infty}}\;\,|\omega|\,\left|\xi_m^k-\left(a+\frac{1}{k}\right)\right|.
\end{equation*}
Here, the set $\mathcal{A}$ and the function $\omega$ do not play an essential role; we only include them in view of what follows. Since, for each $k\ge k_0$, $\xi_m^k \to a+\frac{1}{k}$ in $L^2$ as $m\to\infty$, there exists an increasing sequence $\{m_k\}_{k\ge k_0}$ such that
\begin{equation}\label{eq:convergent_integral_2}
    \left\|\omega\left\{\frac{\sqrt{a}}{\sqrt{\xi_{m_k}^k}} - \frac{\sqrt{a}}{\sqrt{a+1/k}}\right\}\boldsymbol{1}_{\mathcal{A}} \right\|_{L^2} <\frac{\eps}{2}
\end{equation}
for all $k\ge k_0$. Moreover, it is possible to choose the sequence $\{m_k\}_{k\ge k_0}$ in such a way that
\begin{equation}\label{eq:convergent_sequence}
    \left|\left\{\xi_{m_k}^k -\frac{1}{k}\right\}\boldsymbol{1}_{\mathcal{A}^c} \right|\to 0 \ \ \ \text{as } k\to \infty,
\end{equation}
because, for all $k\ge k_0$, $\|\{\xi_{m}^k -1/k\}\boldsymbol{1}_{\mathcal{A}^c} \|_{L^2}\to0$ as $m\to\infty$.

Finally, for each $k$, we set $a_k := \xi_{m_k}^k$, and we prove that the both integrals on the right-hand side of \eqref{eq:bound_A} tend to zero as $k\to \infty$. To this end, we apply the Cauchy–Schwarz inequality, sum and subtract $\frac{1}{k}$, and obtain
\begin{equation*}
    \left\|\omega\,\sqrt{a_k}\,\boldsymbol{1}_{\mathcal{A}^c}\right\|_{L^2}\leq \left\|\omega\right\|_{L^2}\left\|\omega\,\left\{a_k-\frac{1}{k}\right\}\,\boldsymbol{1}_{\mathcal{A}^c}\right\|_{L^2}+\frac{\left\|\omega\right\|_{L^2}^2}{k}.
\end{equation*}
Since $\omega \in L^2$, the last term above converges to zero as $k\to\infty$. To show that the other term also vanishes as $k\to \infty$, we use \eqref{eq:convergent_sequence} together with the dominated convergence theorem.

Thus, to conclude this proof, we only need to prove the convergence of $\big\|\omega\,\frac{a-a_k}{\sqrt{a_k}}\boldsymbol{1}_{\mathcal{A}}\big\|_{L^2}$ to zero as $k\to \infty.$ But now, note that
\begin{equation*}
\begin{split}
    &\left\|\omega\,\frac{a-a_k}{\sqrt{a_k}}\boldsymbol{1}_{\mathcal{A}}\right\|_{L^2}=\left\|\omega\,\left\{\frac{\sqrt{a}}{\sqrt{a_k}}-1\right\}\big(\sqrt{a}+\sqrt{a_k}\,\big)\boldsymbol{1}_{\mathcal{A}}\right\|_{L^2}\\&\le C(a)\left\{\left\|\omega\,\left\{\frac{\sqrt{a}}{\sqrt{a_k}}-\frac{\sqrt{a}}{\sqrt{a+1/k}}\right\} \boldsymbol{1}_{\mathcal{A}}\right\|_{L^2}+\left\|\omega\,\left\{\frac{\sqrt{a}}{\sqrt{a+1/k}}-1\right\}\boldsymbol{1}_{\mathcal{A}}\right\|_{L^2}\right\}.
    \end{split}
\end{equation*}
Hence, by \eqref{eq:convergent_integral_1} and \eqref{eq:convergent_integral_2} with $a_k=\xi_{m_k}^k$, we have that 
\begin{equation*}
    \left\|\omega\,\frac{a-a_k}{\sqrt{a_k}}\boldsymbol{1}_{\mathcal{A}}\right\|_{L^2}<C(a)\eps
\end{equation*}
for all $k\ge k_0$, which completes the proof.
\end{proof}

\begin{proof}[Proof of Lemma \ref{lemma:uniqueness}]
Suppose that $\rho^1_t(u)$ and $\rho^2_t(u)$ are both weak solutions of \eqref{eq:pde_zr} with the same initial condition $\gamma$, and let
\begin{equation*}
    \omega_t(u) := (\rho_t^1 - \rho_t^2)(u).
\end{equation*}
From \eqref{eq:intergal_pde} for $\rho^1$ and $\rho^2$, subtracting these equations and putting some terms in evidence, we obtain 
\begin{equation}\label{eq:omega_formulation} \begin{split}
    \langle \omega_t , G_t\rangle =&\; \int_0^t \big\{\langle \omega_s, \partial_s G_s \rangle + \big\langle (\Phi(\rho_s^1)-\Phi(\rho_s^2)), \Delta G_s\big\rangle\big\}\de s
    \\& -\int_0^t (\Phi(\rho_s^1)-\Phi(\rho_s^2))(1)\,\big\{\partial_uG_s(1) +\tilde\kappa \delta G_s(1)\big\}\de s
    \\&+\int_0^t (\Phi(\rho_s^1)-\Phi(\rho_s^2))(0)\,\big\{ \partial_uG_s(0)- \tilde\kappa \lambda G_s(0)\big\} \de s.
    \end{split}
\end{equation}
By Lemma \ref{lemma:Phi_lips}, we have that $|(\Phi(\rho_s^1)-\Phi(\rho_s^2))(u)|\leq g^*|\omega_s(u)|$ for all $u\in[0,1]$ and $s\in[0, T]$. Then, we can define the bounded function 
\begin{equation*}
    v_s(u):=\frac{(\Phi(\rho_s^1)-\Phi(\rho_s^2))(u)}{(\rho_s^1-\rho_s^2)(u)}\ \ \text{ for }u\in[0,1]\text{ and }s\in[0, T],
\end{equation*}
so that
\begin{equation*}
    (\Phi(\rho_s^1)-\Phi(\rho_s^2))(u) = v_s(u)\,\omega_s(u) \ \ \text{ for }u\in[0,1]\text{ and }s\in[0, T].
\end{equation*}
Using the function $v_t$, we can rewrite  \eqref{eq:omega_formulation} as
\begin{equation}\label{eq:omega_formulation_2}
    \begin{split}
    \langle \omega_t , G_t\rangle =&\; \int_0^t \langle \omega_s, \partial_s G_s+  v_s \Delta G_s\rangle\de s
    \\& - \int_0^t  v_s(1)\,\omega_s(1)\,\big\{\partial_uG_s(1) +\tilde\kappa \delta G_s(1)\big\}\de s
    \\&+\int_0^t v_s(0)\,\omega_s(0) \,\big\{ \partial_uG_s(0)- \tilde\kappa \lambda G_s(0)\big\} \de s.
    \end{split}
\end{equation}
To estimate the integrals in \eqref{eq:omega_formulation_2}, we employ a suitable test function, namely the solution to the parabolic equation \eqref{eq:parabolic_pde}. Since the function $\Phi$ is strictly increasing, it follows that $v$ is non negative, and by Lemma \ref{lemma:Phi_lips} we also have that $v\le g^*$. However, the function $v$ does not possess the required regularity: to address this, we apply Lemma  \ref{lemma:convergent_sequence} with $a=v$, which guarantees the existence of a sequence of functions $\{a_n\}_{n\ge 0}$ which are $C^{\infty}$ in both space and time such that
\begin{equation*}
    \frac{1}{n} \le a_n \le g^*+\frac{1}{n} \ \ \ \text{and} \ \ \ \left\|\omega\, \frac{a_n-v}{\sqrt{a_n}}\right\|_{L^2([0,T] \times [0,1])} \to 0
\end{equation*}
as $n\to\infty$. Note that as $\rho^1$ and $\rho^2$ are in $L^2([0,T] \times [0,1])$, $\omega$ also belongs to $ L^2([0,T] \times [0,1])$. For fixed $h \in C^2_0([0,1])$, consider the parabolic problem \eqref{eq:parabolic_pde} with $a$ replaced by $a_n$ and with $b(s,0)=\tilde \kappa \lambda$, $b(s,1)=\tilde\kappa \delta$. Then, from Lemma \ref{lemma:parabolic_uniqueness}, there exists a unique solution $\phi^n_s(u)$ to this problem. Let us now choose the test function in \eqref{eq:omega_formulation_2} to be $G_s(u)=\phi^n_s(u)$. Due to the boundary conditions satisfied by $\phi^n_t$, the two boundary integrals vanish, and thus we obtain 
\begin{equation*}
    \begin{split}
    \langle \omega_t , \phi^n_t\rangle &= \int_0^t \langle \omega_s, \partial_s \phi^n_s+  v_s \Delta \phi^n_s\rangle\de s.
    \end{split}
\end{equation*}
Since $\phi^n_t(u)=h(u)$, the integral on the left-hand side above is equal to $\langle \omega_t, h\rangle$.   To handle the integral on the right-hand side, we sum and subtract $ \langle \omega_s, a^n_s\Delta \phi^n_s\rangle$ and use the fact that 
\begin{equation*}
    \partial_s \phi^n_s + a^n_s \Delta \phi^n_s = 0,
\end{equation*}
in the interval $(0,1)$, so that
\begin{equation*}
    \begin{split}
    \langle \omega_t , h\rangle &=
    \int_0^t \langle \omega_s, ( v_s-a^n_s) \Delta \phi^n_s\rangle\de s .
   \end{split}
\end{equation*}
Now, applying the Cauchy-Schwarz inequality to the integral on the right-hand side above, we get that $\langle \omega_t , h\rangle$ is bounded from above by
\begin{equation*}
    \begin{split}
    \left(\int_0^t \int_0^1\Big\{ \omega_s(u) \frac{v_s(u)-a_s^n(u)}{\sqrt{a^n_s(u)}}\Big\}^2\de u\de s\right)^{\frac{1}{2}} \left(\int_0^t \int_0^1 a^n_s(u)\big\{  \Delta \phi^n_s(u)\big\}^2\de u\de s\right)^{\frac{1}{2}}.
   \end{split}
\end{equation*}
By Lemma \ref{lemma:parabolic_bound}, there exists a constant $C$ depending only on $\lambda, \delta, \theta$ ad $h$ such that
\begin{equation*}
    \begin{split}
    \langle \omega_t , h\rangle\le C\,\left\|\omega\,\frac{v-a^n}{\sqrt{a^n}}\right\|_{L^2([0,T]\times[0,1])} 
    \end{split}
\end{equation*}
Sending $n\to\infty$ and using Lemma \ref{lemma:convergent_sequence}, it follows that 
\begin{equation*}
    \langle \omega_t, h\rangle \le 0
\end{equation*}
for any $h \in C^2_0([0,1])$. Now consider a sequence of functions $\{h_j\}_{j\ge1}$ in $C^2_0([0,1])$ such that $h_j(\cdot) \to \boldsymbol{1}_{\{u \in [0,1]:\, \omega_t(u)>0\}}(\cdot)$ in $L^2([0,1])$ as $j\to\infty$. Then, from the previous inequality, we obtain 
\begin{equation*}
    \int_0^1 \omega^+(t,u) \de u \le 0,
\end{equation*}
where $\omega^+ := \mymax\{\omega, 0\}$ denotes the positive part of $\omega$. Therefore, for any $t \in [0,T]$, $\rho^1_t(u)\le \rho^2_t(u)$ for almost every $u \in [0,1]$; that is, $\rho^1\le\rho^2$ for almost every $(t,u) \in [0,T]\times[0,1]$. In the same way, $\rho^2 \leq \rho^1$ almost everywhere on $[0,T]\times[0,1]$, completing the proof.
\end{proof}

\section{Auxiliary Computations}

\subsection{Dirichlet Forms}
Recall the definition of $\mathscr{D}_N^0$ in \eqref{eq:dirichlet_bulk}, and define the following additional non-negative functions, defined for positive densities $f$ with respect to $\bar{\nu}_{\varphi}$:
\begin{align}
    \begin{split}&\mathscr{D}^l_{N}\left(\sqrt{f},\bar{\nu}_{\varphi}\right) := \int_{\Omega_N}\Bigg\{ \frac{\alpha}{N^{\theta}} \left[\sqrt{f(\eta^{1+})}-\sqrt{f(\eta)}\right]^2 \\&\phantom{\mathscr{D}^{\ell}_{N}\left(\sqrt{f},\bar{\nu}_{\varphi}\right) := \int\Bigg\{ }+\frac{\lambda g(\eta(1))}{N^{\theta}} \left[\sqrt{f(\eta^{1-})}-\sqrt{f(\eta)}\right]^2\Bigg\}\de \bar{\nu}_{\varphi},\end{split}\nonumber
    \\\begin{split}&\mathscr{D}^{r}_{N}\left(\sqrt{f}, \bar{\nu}_{\varphi}\right) :=\int_{\Omega_N} \Bigg\{ \frac{\beta}{N^{\theta}} \left[\sqrt{f(\eta^{(N-1)+})}-\sqrt{f(\eta)}\right]^2 
    \\&\phantom{\mathscr{D}^{r}_{N}\left(\sqrt{f}, \bar{\nu}_{\varphi}\right) :=\int \Bigg\{}+ \frac{\delta g(\eta(N-1))}{N^{\theta}} \left[\sqrt{f(\eta^{(N-1)-})}-\sqrt{f(\eta)}\right]^2\Bigg\}\de \bar{\nu}_{\varphi}.\nonumber
    \end{split}
\end{align}

\begin{lemma}\label{lemma:left_boundary} Let $f$ be a density with respect to $\bar{\nu}_{\varphi}$. If $\alpha\leq \lambda\varphi$, then
\begin{equation*}
    \left\langle \mathcal{L}^N_l \sqrt{f}, \sqrt{f} \right\rangle_{\bar\nu_\varphi}\le - \frac{1}{2}\mathscr{D}^l_N\left(\sqrt{f},\bar{\nu}_{\varphi}\right) + \frac{\lambda\varphi-\alpha}{2N^{\theta}}.
\end{equation*}
\end{lemma}
\begin{proof} Note that 
\begin{equation*}
    \begin{split}
    \left\langle \mathcal{L}^N_l \sqrt{f}, \sqrt{f} \right\rangle_{\bar\nu_\varphi}=&\int_{\Omega_N} \frac{\alpha}{N^{\theta}}\left[\sqrt{f(\eta^{1+})}-\sqrt{f(\eta)}\right]\sqrt{f(\eta)}\de \bar{\nu}_\varphi
    \\&+\int_{\Omega_N}\frac{\lambda g(\eta(1))}{N^{\theta}}\left[\sqrt{f(\eta^{1-})}-\sqrt{f(\eta)}\right]\sqrt{f(\eta)}\de \bar{\nu}_{\varphi}.
\end{split}
\end{equation*}
By writing the terms on the right-hand side of last display as their half plus their half, and summing
and subtracting appropriate terms to complete the squares, we see that
\begin{equation*}
    \begin{split}
    \left\langle \mathcal{L}^N_l \sqrt{f}, \sqrt{f} \right\rangle_{\bar\nu_\varphi}=& \int \frac{\alpha}{2N^{\theta}}\left[\sqrt{f(\eta^{1+})}-\sqrt{f(\eta)}\right]\sqrt{f(\eta)}\de \bar\nu_\varphi 
    \\&- \int_{\Omega_N} \frac{\alpha}{2N^{\theta}}\left[\sqrt{f(\eta^{1+})}-\sqrt{f(\eta)}\right]\sqrt{f(\eta^{1+})}\de \bar\nu_\varphi
    \\& +\int_{\Omega_N} \frac{\alpha}{2N^{\theta}}\left[\sqrt{f(\eta^{1+})}-\sqrt{f(\eta)}\right]\sqrt{f(\eta^{1+})}\de \bar\nu_\varphi
    \\&+\int_{\Omega_N} \frac{\alpha}{2N^{\theta}}\left[\sqrt{f(\eta^{1+})}-\sqrt{f(\eta)}\right]\sqrt{f(\eta)}\de \bar\nu_\varphi
    \\& +\int_{\Omega_N} \frac{\lambda g(\eta(1))}{2N^{\theta}}\left[\sqrt{f(\eta^{1-})}-\sqrt{f(\eta)}\right]\sqrt{f(\eta)}\de \bar\nu_\varphi
    \\&- \int_{\Omega_N} \frac{\lambda g(\eta(1))}{2N^{\theta}}\left[\sqrt{f(\eta^{1-})}-\sqrt{f(\eta)}\right]\sqrt{f(\eta^{1-})}\de \bar\nu_\varphi
    \\ & + \int_{\Omega_N} \frac{\lambda g(\eta(1))}{2N^{\theta}}\left[\sqrt{f(\eta^{1-})}-\sqrt{f(\eta)}\right]\sqrt{f(\eta^{1-})}\de \bar\nu_\varphi
    \\&+ \int_{\Omega_N} \frac{\lambda g(\eta(1))}{2N^{\theta}}\left[\sqrt{f(\eta^{1-})}-\sqrt{f(\eta)}\right]\sqrt{f(\eta)}\de \bar\nu_\varphi.
    \end{split}
\end{equation*}
By grouping the terms which yield the Dirichlet form $\mathscr{D}_N^l$ and expanding the product, we obtain
\begin{equation*}
    \left\langle \mathcal{L}^N_l \sqrt{f}, \sqrt{f} \right\rangle_{\bar\nu_\varphi}= - \frac{1}{2}\mathscr{D}^l_N\left(\sqrt{f}, \bar{\nu}_{\varphi}\right) + \mathscr{R}^l,
\end{equation*}
where
\begin{equation*}
    \begin{split}
    \mathscr{R}^l := & \int_{\Omega_N}\frac{\alpha}{2N^{\theta}}\left[f(\eta^{1+})-f(\eta)\right]\de \bar\nu_\varphi + \int_{\Omega_N} \frac{\lambda g(\eta(1))}{2N^{\theta}}\left[f(\eta^{1-})-f(\eta)\right]\de \bar\nu_\varphi.
    \end{split}
\end{equation*}
If $\bar\nu_\varphi$ were the invariant measure, we would have the identity $\mathscr{R}^l=\frac{1}{2}\int_{\Omega_N} \mathcal{L}^N_l f(\eta) \de \bar\nu_\varphi\equiv 0$; nonetheless, performing the change of variable  $\eta^{1+}\mapsto \eta$ and $\eta^{1-}\mapsto\eta$ and noting that $\frac{\bar \nu_\varphi (\eta^{1-})}{\bar\nu_\varphi(\eta)}=\frac{g(\eta_1)}{\varphi}$ and  that $\frac{\bar\nu_\varphi(\eta^{1+})}{\bar\nu_\varphi(\eta)}=\frac{\varphi}{g(\eta_1+1)}$, we obtain 
\begin{equation*}
    \mathscr{R}^l = \frac{\lambda \varphi - \alpha}{2N^{\theta}}\int_{\Omega_N} f(\eta)\de \bar\nu_\varphi + \left(\frac{\alpha}{2N^{\theta}\varphi} - \frac{\lambda}{2N^{\theta}} \right)\int_{\Omega_N} g(\eta(1))f(\eta)\de \bar\nu_\varphi.
\end{equation*}
Under the assumption $\alpha\le \lambda \varphi$ and since $f$ is a density, we can discard the rightmost term in the previous display, so the proof ends.
\end{proof}

\begin{lemma}\label{lemma:right_boundary} Let $f$ be the density with respect to $\bar\nu_\varphi$. If $\beta\le \delta\varphi$, then
\begin{equation*}
    \left\langle \mathcal{L}^N_r\sqrt{f}, \sqrt{f} \right\rangle_{\bar\nu_\varphi}\le-\frac{1}{2}\mathscr{D}^{r}_N
    \left(\sqrt{f}, \bar\nu_\varphi\right)+  \frac{\delta\varphi-\beta}{2N^{\theta}} .
\end{equation*}
\end{lemma}
\begin{proof} It follows from the exact same argument given in the proof of Lemma \ref{lemma:left_boundary}.\end{proof}

\begin{lemma}\label{lemma:bulk}
Let $f$ be a density with respect to $\bar\nu_\varphi$: then,
\begin{equation*}
    \left\langle \mathcal{L}^N_0\sqrt{f}, \sqrt{f} \right\rangle_{\bar\nu_\varphi}=-\frac{1}{2}\mathscr{D}_N^0
    \left(\sqrt{f}, \bar\nu_\varphi\right).
\end{equation*}
\end{lemma}
\begin{proof}
Note that 
\begin{equation}\label{eq:lemma_bulk}
    \begin{split}
    \left\langle \mathcal{L}^N_0\sqrt{f}, \sqrt{f} \right\rangle_{\bar\nu_\varphi} =& \int_{\Omega_N} \sum_{x=1}^{N-2} g(\eta(x))\left[\sqrt{f(\eta^{x,x+1})}-\sqrt{f(\eta)}\right]\sqrt{f(\eta)} \de \bar\nu_\varphi
    \\&+\int\sum_{x=2}^{N-1} g(\eta(x)) \left[\sqrt{f(\eta^{x,x-1})}-\sqrt{f(\eta)}\right]\sqrt{f(\eta)} \de \bar{\nu}_\varphi.
    \end{split}
\end{equation}
By writing the term at the left hand side of \eqref{eq:lemma_bulk} as its half plus its half and summing and subtracting an appropriate term to complete the square, we have that
\begin{equation*}
    \begin{split}
    \left\langle \mathcal{L}^N_0\sqrt{f}, \sqrt{f} \right\rangle_{\bar\nu_\varphi} = &-\frac{1}{2}\mathscr{D}^0_N
    \left(\sqrt{f},\bar{\nu}_{\varphi}\right)
    \\&+\frac{1}{2}\int_{\Omega_N}\sum_{x=1}^{N-2} g(\eta(x))\left[\sqrt{f(\eta^{x,x+1})}-\sqrt{f(\eta)}\right]\sqrt{f(\eta)} \de \bar{\nu}_\varphi
    \\&+\frac{1}{2}\int_{\Omega_N} \sum_{x=1}^{N-2} g(\eta(x))\left[\sqrt{f(\eta^{x,x+1})}-\sqrt{f(\eta)}\right]\sqrt{f(\eta^{x,x+1})} \de \bar{\nu}_\varphi
    \\&+\frac{1}{2}\int_{\Omega_N} \sum_{x=2}^{N-1} g(\eta(x)) \left[\sqrt{f(\eta^{x,x-1})}-\sqrt{f(\eta)}\right]\sqrt{f(\eta)} \de \bar{\nu}_\varphi
    \\&+\frac{1}{2}\int_{\Omega_N} \sum_{x=2}^{N-1} g(\eta(x)) \left[\sqrt{f(\eta^{x,x-1})}-\sqrt{f(\eta)}\right]\sqrt{f(\eta^{x,x-1})} \de \bar{\nu}_\varphi. 
    \end{split}
\end{equation*}
By performing an appropriate change of variable in the third and fifth lines, we can rewrite the equation above as
\begin{equation*}
    \begin{split}
    \left\langle \mathcal{L}^N_0\sqrt{f}, \sqrt{f} \right\rangle_{\bar\nu_\varphi} =& -\frac{1}{2}\mathscr{D}^0_N
    \left(\sqrt{f},\bar{\nu}_{\varphi}\right)
    \\&+\frac{1}{2}\int_{\Omega_N} \sum_{x=1}^{N-2} g(\eta(x))\left[\sqrt{f(\eta^{x,x+1})}-\sqrt{f(\eta)}\right]\sqrt{f(\eta)} \de \bar{\nu}_\varphi
    \\&+\frac{1}{2}\int_{\Omega_N} \sum_{x=2}^{N-1} g(\eta(x))\left[\sqrt{f(\eta)}-\sqrt{f(\eta^{x,x-1})}\right]\sqrt{f(\eta)} \de \bar{\nu}_\varphi
    \\&+\frac{1}{2}\int_{\Omega_N} \sum_{x=2}^{N-1} g(\eta(x)) \left[\sqrt{f(\eta^{x,x-1})}-\sqrt{f(\eta)}\right]\sqrt{f(\eta)} \de \bar{\nu}_\varphi
    \\&+\frac{1}{2}\int_{\Omega_N} \sum_{x=1}^{N-2} g(\eta(x)) \left[\sqrt{f(\eta)}-\sqrt{f(\eta^{x,x+1})}\right]\sqrt{f(\eta)} \de \bar{\nu}_\varphi.
    \end{split}
\end{equation*}
But then, it is not hard to see that the sum of the second to fifth line vanishes, so our claim follows. \end{proof}

\subsection{Relative Entropy Bound} \label{sec:app_entropy}
In this brief section, we prove the following statement:

\begin{lemma} If a sequence of measures $\{\mu^N\}_N$ on $\Omega_N$ satisfies the relative entropy bound $H(\mu^N|\bar\nu_N)\lesssim N$, then it also satisfies $H(\mu^N|\bar\nu_\varphi)\lesssim N$ for any $\varphi\ge\mysup_{x\in I_N}\bar\varphi_N(x)$.
\end{lemma}
\begin{proof} Note that
\begin{align*}
    H(\mu^N| \bar\nu_{\varphi})&=\sum_{\eta\in\Omega_N}\mu^N(\eta) \mylog\left(\frac{\mu^N(\eta)}{\bar\nu_N(\eta)}\cdot \frac{\bar\nu_N(\eta)}{\bar\nu_\varphi(\eta)}\right) \\&=H(\mu^N|\bar\nu_N)+\sum_{\eta\in\Omega_N}\mu^N(\eta) \mylog\left(\frac{\bar\nu_N(\eta)}{\bar\nu_\varphi(\eta)}\right).
\end{align*}
A simple computation shows that 
\begin{align*}
    \mylog\left(\frac{\bar\nu_N(\eta)}{\bar\nu_{\varphi}(\eta)}\right)&= \sum_{x\in I_N} \left\{\mylog \left(\frac{Z(\varphi)}{Z(\bar\varphi_N(x))}\right)+\eta(x)\mylog\left(\frac{\bar\varphi_N(x)}{\varphi}\right)\right\}
    \\&\le N\mysup_{x\in I_N}\mylog \left(\frac{Z(\varphi)}{Z(\bar\varphi_N(x))}\right),
\end{align*}
where we used the fact that $\mylog\left(\frac{\bar\varphi_N(x)}{\varphi}\right)\le0$. But now, for each $x\in I_N$ we have that
\begin{equation*}
    Z(\bar\varphi_N(x)) =\sum_{k\ge0}\frac{\bar\varphi_N(x)^k}{g(k)!} \ge \sum_{k\ge0}\frac{\left\{\myinf_{x\in I_N}\bar\varphi_N(x)\right\}^k}{g(k)!}\ge C,
\end{equation*} 
where $C$ is a positive constant which only depends on the parameters $\alpha, \lambda, \beta, \delta$ and $\theta$. This implies that $\mylog \left(\frac{Z(\varphi)}{Z(\bar\varphi_N(x))}\right)$ is uniformly bounded from above in $N$, hence the claim follows.\end{proof}

\bibliographystyle{Martin} 
\bibliography{references}

\begin{thebibliography}{BDGN20}
\def\myhref#1#2{\href{#2}{\nolinkurl{#1}}}

\bibitem[Ald78]{ald78}
\textsc{D.~Aldous}.
\newblock {Stopping Times and Tightness}.
\newblock \emph{Ann. Probab.} \textbf{6}, no.~2, (1978), 335--340.
\newblock
  \myhref{doi:10.1214/aop/1176995579}{https://doi.org/10.1214/aop/1176995579}.

\bibitem[BCGS23]{bcgs23}
\textsc{C.~Bernardin}, \textsc{P.~Cardoso}, \textsc{P.~Gonçalves}, and
  \textsc{S.~Scotta}.
\newblock {Hydrodynamic Limit for a Boundary Driven Super-Diffusive Symmetric
  Exclusion}.
\newblock \emph{Stochastic Process. Appl.} \textbf{165}, (2023), 43--95.
\newblock
  \myhref{doi:10.1016/j.spa.2023.08.002}{https://doi.org/10.1016/j.spa.2023.08.002}.

\bibitem[BDGN20]{bdgn20}
\textsc{L.~Bonorino}, \textsc{R.~{De Paula}}, \textsc{P.~Gonçalves}, and
  \textsc{A.~Neumann}.
\newblock {Hydrodynamics of Porous Medium Model with Slow Reservoirs}.
\newblock \emph{J. Stat. Phys.} \textbf{179}, no.~3, (2020), 748--788.
\newblock
  \myhref{doi:10.1007/s10955-020-02550-y}{https://doi.org/10.1007/s10955-020-02550-y}.

\bibitem[BGJ19]{bgj19}
\textsc{C.~Bernardin}, \textsc{P.~Gonçalves}, and
  \textsc{B.~{Jim\'enez-Oviedo}}.
\newblock {Slow to Fast Infinitely Extended Reservoirs for the Symmetric
  Exclusion Process with Long Jumps}.
\newblock \emph{Markov Process. Related Fields} \textbf{25}, no.~2, (2019),
  217--274.

\bibitem[BGJ21]{bgj21}
\textsc{C.~Bernardin}, \textsc{P.~Gonçalves}, and
  \textsc{B.~{Jim\'enez-Oviedo}}.
\newblock {A Microscopic Model for a One Parameter Class of Fractional
  Laplacians with Dirichlet Boundary Conditions}.
\newblock \emph{Arch. Ration. Mech. Anal.} \textbf{239}, no.~1, (2021), 1--48.
\newblock
  \myhref{doi:10.1007/s00205-020-01549-9}{https://doi.org/10.1007/s00205-020-01549-9}.

\bibitem[BGJS22]{bgjs22}
\textsc{C.~Bernardin}, \textsc{P.~Gonçalves}, \textsc{B.~{Jim\'enez-Oviedo}},
  and \textsc{S.~Scotta}.
\newblock {Non-Equilibrium Stationary Properties of the Boundary Driven
  Zero-Range process with Long Jumps}.
\newblock \emph{J. Stat. Phys.} \textbf{189}, no.~3, (2022), Paper No. 32, 32.
\newblock
  \myhref{doi:10.1007/s10955-022-02987-3}{https://doi.org/10.1007/s10955-022-02987-3}.

\bibitem[BMNS17]{bmns17}
\textsc{R.~Baldasso}, \textsc{O.~Menezes}, \textsc{A.~Neumann}, and
  \textsc{R.~R. Souza}.
\newblock {Exclusion Process with Slow Boundary}.
\newblock \emph{J. Stat. Phys.} \textbf{167}, no.~5, (2017), 1112--1142.
\newblock
  \myhref{doi:10.1007/s10955-017-1763-5}{https://doi.org/10.1007/s10955-017-1763-5}.

\bibitem[FC23]{fg23}
\textsc{S.~Floreani} and \textsc{A.~G. Casanova}.
\newblock {Non-Equilibrium Steady State of the Symmetric Exclusion Process with
  Reservoirs}.
\newblock \emph{ArXiv e-prints} (2023).
\newblock
  \myhref{doi:10.48550/arXiv.2307.02481}{https://doi.org/10.48550/arXiv.2307.02481}.

\bibitem[FGN17]{fgn17}
\textsc{T.~Franco}, \textsc{P.~Gonçalves}, and \textsc{A.~Neumann}.
\newblock {Equilibrium Fluctuations for the Slow Boundary Exclusion Process}.
\newblock In \emph{{From Particle Systems to Partial Differential Equations}},
  vol. 209 of \emph{Springer Proc. Math. Stat.},  177--197. Springer, Cham,
  2017.
\newblock
  \myhref{doi:10.1007/978-3-319-66839-0_9}{https://doi.org/10.1007/978-3-319-66839-0_9}.

\bibitem[FGN19]{fgn19}
\textsc{T.~Franco}, \textsc{P.~Gonçalves}, and \textsc{A.~Neumann}.
\newblock {Non-Equilibrium and Stationary Fluctuations of a Slowed Boundary
  Symmetric Exclusion}.
\newblock \emph{Stochastic Process. Appl.} \textbf{129}, no.~4, (2019),
  1413--1442.
\newblock
  \myhref{doi:10.1016/j.spa.2018.05.005}{https://doi.org/10.1016/j.spa.2018.05.005}.

\bibitem[FMN21]{fmn21}
\textsc{S.~Fr\'ometa}, \textsc{R.~Misturini}, and \textsc{A.~Neumann}.
\newblock {The Boundary Driven Zero-Range Process}.
\newblock In \emph{{From Particle Systems to Partial Differential Equations}},
  vol. 352 of \emph{Springer Proc. Math. Stat.},  253--281. Springer, Cham,
  2021.
\newblock
  \myhref{doi:10.1007/978-3-030-69784-6_12}{https://doi.org/10.1007/978-3-030-69784-6_12}.

\bibitem[FVST23]{fmts23}
\textsc{T.~Funaki}, \textsc{P.~{Van Meurs}}, \textsc{S.~Sethuraman}, and
  \textsc{K.~Tsunoda}.
\newblock {Motion by Mean Curvature from Glauber-Kawasaki Dynamics with Speed
  Change}.
\newblock \emph{J. Stat. Phys.} \textbf{190}, no.~3, (2023), Paper No. 45, 30.
\newblock
  \myhref{doi:10.1007/s10955-022-03044-9}{https://doi.org/10.1007/s10955-022-03044-9}.

\bibitem[Gon19]{gon19}
\textsc{P.~Gonçalves}.
\newblock {Hydrodynamics for Symmetric Exclusion in Contact with Reservoirs}.
\newblock In \emph{{Stochastic Dynamics Out of Equilibrium}}, vol. 282 of
  \emph{Springer Proc. Math. Stat.},  137--205. Springer, Cham, 2019.
\newblock
  \myhref{doi:10.1007/978-3-030-15096-9\\_4}{https://doi.org/10.1007/978-3-030-15096-9\%5C_4}.

\bibitem[GPV88]{gpv88}
\textsc{M.~Z. Guo}, \textsc{G.~C. Papanicolaou}, and \textsc{S.~R.~S.
  Varadhan}.
\newblock {Nonlinear Diffusion Limit for a System with Nearest Neighbor
  Interactions}.
\newblock \emph{Comm. Math. Phys.} \textbf{118}, no.~1, (1988), 31--59.
\newblock
  \myhref{http://projecteuclid.org/euclid.cmp/1104161907}{http://projecteuclid.org/euclid.cmp/1104161907}.

\bibitem[GS22]{gs22}
\textsc{P.~Gonçalves} and \textsc{S.~Scotta}.
\newblock {Diffusive to Super-Diffusive Behavior in Boundary Driven Exclusion}.
\newblock \emph{Markov Process. Related Fields} \textbf{28}, no.~1, (2022),
  149--178.

\bibitem[KL99]{kl99}
\textsc{C.~Kipnis} and \textsc{C.~Landim}.
\newblock \emph{Scaling Limits of Interacting Particle Systems}.
\newblock Springer-Verlag, Berlin, 1999.
\newblock
  \myhref{doi:10.1007/978-3-662-03752-2}{https://doi.org/10.1007/978-3-662-03752-2}.

\bibitem[Lig05]{lig05}
\textsc{T.~M. Liggett}.
\newblock \emph{{Interacting Particle Systems}}.
\newblock Classics in Mathematics. Springer-Verlag, Berlin, 2005.
\newblock \myhref{doi:10.1007/b138374}{https://doi.org/10.1007/b138374}.

\bibitem[LMS05]{lms05}
\textsc{E.~Levine}, \textsc{D.~Mukamel}, and \textsc{G.~M. Sch\"utz}.
\newblock {Zero-Range Process with Open Boundaries}.
\newblock \emph{J. Stat. Phys.} \textbf{120}, no. 5-6, (2005), 759--778.
\newblock
  \myhref{doi:10.1007/s10955-005-7000-7}{https://doi.org/10.1007/s10955-005-7000-7}.

\bibitem[LSV96]{lsv96}
\textsc{C.~Landim}, \textsc{S.~Sethuraman}, and \textsc{S.~Varadhan}.
\newblock {Spectral Gap for Zero-Range Dynamics}.
\newblock \emph{Ann. Probab.} \textbf{24}, no.~4, (1996), 1871--1902.
\newblock
  \myhref{doi:10.1214/aop/1041903209}{https://doi.org/10.1214/aop/1041903209}.

\bibitem[MMM24]{mmm24}
\textsc{D.~Marahrens}, \textsc{A.~Menegaki}, and \textsc{C.~Mouhot}.
\newblock {A Consistency-Stability Approach to Scaling Limits of Zero-Range
  Processes}.
\newblock \emph{ArXiv e-prints} (2024).
\newblock
  \myhref{doi:10.48550/arXiv.2412.16714}{https://doi.org/10.48550/arXiv.2412.16714}.

\bibitem[Mor06]{mor06}
\textsc{B.~Morris}.
\newblock {Spectral Gap for the Zero Range Process with Constant Rate}.
\newblock \emph{Ann. Probab.} \textbf{34}, no.~5, (2006), 1645--1664.
\newblock
  \myhref{doi:10.1214/009117906000000304}{https://doi.org/10.1214/009117906000000304}.

\bibitem[Nag10]{nag10}
\textsc{Y.~Nagahata}.
\newblock {Spectral Gap for Zero-Range Processes with Jump Rate
  {$g(x)=x^\gamma$}}.
\newblock \emph{Stochastic Process. Appl.} \textbf{120}, no.~6, (2010),
  949--958.
\newblock
  \myhref{doi:10.1016/j.spa.2010.01.019}{https://doi.org/10.1016/j.spa.2010.01.019}.

\bibitem[Spi70]{spi70}
\textsc{F.~Spitzer}.
\newblock {Interaction of Markov Processes}.
\newblock \emph{Advances in Math.} \textbf{5}, (1970), 246--290.
\newblock
  \myhref{doi:10.1016/0001-8708(70)90034-4}{https://doi.org/10.1016/0001-8708(70)90034-4}.

\end{thebibliography}

\end{document}